\documentclass[11pt,twoside,a4paper]{article}
\usepackage{amsmath,amssymb,amsthm}
\usepackage{hyperref}
\usepackage{graphicx,psfrag}

\newtheorem{thm}{Theorem}[section]
\newtheorem{lem}[thm]{Lemma}
\newtheorem{pro}[thm]{Proposition}
\newtheorem{cor}[thm]{Corollary}

\theoremstyle{definition}

\newtheorem{exa}[thm]{Example}

\theoremstyle{remark}
\newtheorem{rem}[thm]{Remark}

\newcommand{\R}{\mathbb{R}}
\newcommand{\Z}{\mathbb{Z}}
\newcommand{\N}{\mathbb{N}}

\newcommand{\K}{\mathbb{K}}

\newcommand{\cD}{\mathcal{D}}

\newcommand{\cL}{\mathcal{L}}
\newcommand{\cM}{\mathcal{M}}

\newcommand{\cR}{\mathcal{R}}
\newcommand{\cS}{\mathcal{S}}

\newcommand{\al}{\alpha}
\newcommand{\be}{\beta}
\newcommand{\ga}{\gamma}
\newcommand{\Ga}{\Gamma}
\newcommand{\de}{\delta}

\newcommand{\ep}{\varepsilon}
\newcommand{\om}{\omega}

\newcommand{\si}{\sigma}

\newcommand{\la}{\lambda}
\newcommand{\La}{\Lambda}
\renewcommand{\phi}{\varphi}

\newcommand{\crt}{\operatorname{crt}}

\newcommand{\dist}{\operatorname{dist}}

\newcommand{\Hd}{\operatorname{Hd}}
\newcommand{\CAT}{\operatorname{CAT}}
\newcommand{\hyp}{\operatorname{H}}
\newcommand{\id}{\operatorname{id}}

\newcommand{\const}{\operatorname{const}}

\newcommand{\aut}{\operatorname{Aut}}

\newcommand{\lift}{\operatorname{lift}}
\newcommand{\sign}{\operatorname{sign}}

\newcommand{\slope}{\operatorname{slope}}
\newcommand{\ad}{\operatorname{ad}}

\newcommand{\Int}{\operatorname{int}}
\newcommand{\fil}{\operatorname{Fill}}

\newcommand{\es}{\emptyset}
\renewcommand{\d}{\partial}
\newcommand{\di}{\d_{\infty}}
\newcommand{\set}[2]{\{#1:\,\text{#2}\}}
\newcommand{\sm}{\setminus}
\newcommand{\sub}{\subset}
\newcommand{\sups}{\supset}

\newcommand{\ov}{\overline}
\newcommand{\wt}{\widetilde}
\newcommand{\wh}{\widehat}

\begin{document}

\title{M\"obius characterization of the boundary at infinity
of rank one symmetric spaces}
\author{Sergei Buyalo\footnote{Supported by RFBR Grant
11-01-00302-a and SNF Grant 20-119907/1}
\ \& Viktor Schroeder\footnote{Supported by Swiss National
Science Foundation Grant 20-119907/1}}

\date{}
\maketitle

\begin{abstract} A M\"obius structure (on a set 
$X$)
is a class of metrics having the same cross-ratios.
A M\"obius structure is ptolemaic if it is invariant
under inversion operations. The boundary at infinity of a
$\CAT(-1)$ 
space is in a natural way a M\"obius space, which is 
ptolemaic. We give a free of classification proof of the following
result that characterizes the rank one symmetric spaces
of noncompact type purely in terms of their M\"obius geometry: Let
$X$
be a compact Ptolemy space which contains a Ptolemy circle 
and allows many space inversions. Then 
$X$
is M\"obius equivalent to the boundary at infinity
of a rank one symmetric space.
\end{abstract}

\section{Introduction}
\label{sect:introduction}

Two metrics on a set 
$X$
are M\"obius equivalent if they have the same cross-ratios.
A M\"obius structure on 
$X$
is a class of M\"obius equivalent metrics. If a M\"obius structure
is fixed then
$X$
is called a M\"obius space. Ptolemy spaces are M\"obius spaces
with the property that the inversion operation preserves the
M\"obius structure. A classical example of a Ptolemy space is
the extended
$\wh\R^n=\R^n\cup\infty=S^n$, $n\ge 0$,
where the M\"obius structure is generated by an Euclidean metric
on
$\R^n$,
and
$\wh\R^n$
is identified with the unit sphere
$S^n\sub\R^{n+1}$
via the stereographic projection. For more detail see Section~\ref{sect:moebius_structures}.

There is a well known deep connection between the geometry
of the hyperbolic space
$\hyp^{n+1}$
and the M\"obius geometry of its boundary
$\di\hyp^{n+1}=\wh\R^n$.
More generally the boundary
$X=\di Y$
of a 
$\CAT(-1)$
space
$Y$
carries in a natural way a M\"obius structure and is actually
a Ptolemy space \cite{FS1}. An isometry of
$Y$
induces a M\"obius map of
$X$
and a Ptolemy circle
$\si$
in
$X$
corresponds to a totally geodesic subspace 
$Y_\si\sub Y$
isometric to
$\hyp^2$
with
$\di Y_\si=\si$.
Here a Ptolemy circle is a subspace 
M\"obius equivalent to the Ptolemy space
$\wh\R=S^1$.

Our motivation is to find a M\"obius characterization of the boundary
at infinity of rank one symmetric spaces 
$Y$
of non-compact type. In the case
$Y=\hyp^{n+1}$
this problem is solved in \cite{FS2} for every
$n\ge 1$: {\em every compact Ptolemy space such that through any three points there is
a Ptolemy circle is M\"obius equivalent to
$\wh\R^n=\di\hyp^{n+1}$.}

Given distinct points
$\om$, $\om'$
in a M\"obius space
$X$,
there is a well defined notion of a {\em sphere}
$S$
between
$\om$, $\om'$,
see sect.~\ref{subsect:spheres_between_points},
and a notion of a {\em space inversion} w.r.t.
$\om$, $\om'$, $S$,
which is a fixed point free M\"obius involution
$\phi_{\om,\om',S}:X\to X$,
permuting
$\om$
and
$\om'$,
preserving 
$S$
and any Ptolemy circle through
$\om$, $\om'$,
see sect.~\ref{subsect:space_inversions}.

We consider a M\"obius space
$X$
with the following basic properties. 

\noindent
(E) Existence: there is at least one Ptolemy circle in
$X$.

\noindent
(I) Inversion: for each distinct
$\om$, $\om'\in X$
and every sphere
$S\sub X$
between
$\om$, $\om'$
there is a unique space inversion
$\phi_{\om,\om',S}:X\to X$
w.r.t.
$\om$, $\om'$
and
$S$.

Our main result gives the following M\"obius characterization of 
the boundary at infinity of rank one symmetric spaces of non-compact
type.

\begin{thm}\label{thm:moebius} Let
$X$
be a compact Ptolemy space with properties 
($E$) and (I). Then
$X$
is M\"obius equivalent to the boundary at infinity
of a rank one symmetric space of non-compact type
taken with the canonical M\"obius structure. 
\end{thm}

The canonical M\"obius structure on the boundary at infinity
$X=\di Y$
of every rank one symmetric space
$Y$
of non-compact type can be described as follows. Assume that 
the metric of
$Y$
is normalized such that maximum of sectional curvatures equals
$-1$.
Take 
$\om\in X$
and a Busemann function
$b:Y\to\R$
centered at
$\om$.
Then the function
$d_b:X\times X\to\R$
defined by
$$d_b(\xi,\xi')=e^{-(\xi|\xi')_b},$$
where
$(\xi|\xi')_b\in\R$
is the Gromov product of
$\xi$, $\xi'$
w.r.t.
$b$,
is a metric on
$X$.
The canonical M\"obius structure on
$X$
is generated by all such metrics
$d_b$
with
$b\in\om\in X$.

\begin{rem}\label{rem:normalization} 
The property (E) plays a double role. First, it put a restriction 
on the topology of
$X$
excluding e.g. totally disconnected spaces. Second, it
also serves as a normalization condition,
because, for example, the Ptolemy space 
$X=\di\hyp^{n+1}$, 
where
$\hyp^{n+1}$
is a real hyperbolic space of curvature
$-1$,
contains a lot of Ptolemy circles, while for every
$0<\la<1$
the boundary at infinity
$X_\la$
of the rescaled hyperbolic space 
$\la\hyp^{n+1}$
is still a Ptolemy space which however contains no
Ptolemy circle. This is because the rescaling
$\hyp^{n+1}\mapsto\la\hyp^{n+1}$
implies the operation
$d\mapsto d^\la$
for every metric 
$d$
of the M\"obius structure of
$X$. 
\end{rem}

The proof of Theorem~\ref{thm:moebius} relies on two existence
results. The first one, which is formulated as the property
($\K$)
in Theorem~\ref{thm:basic_ptolemy}, provides for every 
$\om\in X$ 
the existence of a canonical fibration
$\pi_\om :X_\om\to B_\om$ 
and for every 
$x\in X_\om$ 
the existence of a Ptolemy circle through 
$x$, $\om$,
which hits a prescribed fiber of 
$\pi_\om$; 
here 
$X_\om=X\sm\om$.
Preparation to the proof of this result, the proof itself
and discussion of various consequences occupy the first
part of the paper, sections~\ref{sect:many_circles_auto} --
\ref{sect:topology_space}. Most important consequence obtained here
is the existence of a simply connected, nilpotent Lie
group structure 
$N_\om$
on every punctured space
$X_\om$, $\om\in X$,
see sect.~\ref{sect:topology_space}. At this point, one can 
formally conclude the proof of Theorem~\ref{thm:moebius}
referring to the Tits classification of 2-transitive group actions,
see \cite{Kr}. The Tits classification gives an alternative
between two possibilities. One of them is related to the rank
one symmetric space on non-compact type, and the first part of 
the paper is actually the proof that the other one cannot be 
realized as M\"obius transformations of any M\"obius structure on
$X$.

However, we are not satisfied with that formal classification
argument, and we prefer to give a direct, classification free 
proof of Theorem~\ref{thm:moebius} to clarify the geometry of the phenomenon.
We directly construct a symmetric space which has the M\"obius space 
$X$ 
as boundary at infinity. This is done in the second part of the paper, 
sections~\ref{sect:extension_circles} -- \ref{sect:filling}: 
We introduce a filling
$Y=\fil X$
of
$X$
as the set of all space inversions of
$X$
and a distance 
$\rho$
on
$Y$
defined via cross-ratios of the M\"obius structure of
$X$.
Then we show that 
$\rho$
is a Riemannian distance associated with
a Riemannian rank one symmetric space
of non-compact type with maximum of sectional curvatures
$-1$.
The proof is based on our second existence result that
every M\"obius map between two Ptolemy circles in
$X$
extends to a M\"obius automorphism of
$X$,
see sect.~\ref{sect:extension_circles}. Finally, we
conclude the proof of Theorem~\ref{thm:moebius} by
showing that 
$X=\di Y$
and that the canonical M\"obius structure
associated with
$Y$
coincides with the initial one on
$X$.

Our proof of Theorem~\ref{thm:moebius} neither uses the 
classification of rank one symmetric spaces of non-compact type,
nor gives a new approach to that classification.

Section~\ref{sect:moebius_structures} serves as a brief
introduction to M\"obius geometry.

{\em Acknowledgments.} We are thankful to
L.~Kramer for informing us about 2-transitive group
actions. The first author
is also grateful to the University of Z\"urich for 
hospitality and support.

\tableofcontents

\section{M\"obius structures and Ptolemy spaces}
\label{sect:moebius_structures}
This section is a brief introduction to M\"obius geometry.

\subsection{M\"obius structures}
\label{subsect:moebius_structures}

A quadruple
$Q=(x,y,z,u)$
of points in a set
$X$
is said to be {\em admissible} if no entry occurs three or
four times in 
$Q$.
Two metrics 
$d$, $d'$
on 
$X$ 
are {\em M\"obius equivalent} if for any admissible quadruple
$Q=(x,y,z,u)\sub X$
the respective {\em cross-ratio triples} coincide,
$\crt_d(Q)=\crt_{d'}(Q)$,
where
$$\crt_d(Q)=(d(x,y)d(z,u):d(x,z)d(y,u):d(x,u)d(y,z))\in\R P^2.$$
We actually consider {\em extended} metrics on
$X$
for which existence of an {\em infinitely remote} point
$\om\in X$
is allowed, that is,
$d(x,\om)=\infty$
for all
$x\in X$, $x\neq\om$.
We always assume that such a point is unique if exists, and that
$d(\om,\om)=0$.
We use notation
$X_\om:=X\sm\om$
and the standard conventions for the calculation with 
$\om=\infty$.
If 
$\infty$ 
occurs once in 
$Q$, 
say 
$u=\infty$,
then
$\crt_d(x,y,z,\infty)=(d(x,y):d(x,z):d(y,z))$.
If 
$\infty$ 
occurs twice, say 
$z=u=\infty$, 
then
$\crt_d(x,y,\infty,\infty)=(0:1:1)$.

A {\em M\"obius structure} on a set
$X$
is a class 
$\cM=\cM(X)$
of metrics on
$X$
which are pairwise M\"obius equivalent.

The topology considered on 
$(X,d)$ 
is the topology with the basis consisting of all open distance balls 
$B_r(x)$
around points in 
$x\in X_\om$ 
and the complements $X\sm D$ 
of all closed balls 
$D=\overline{B}_r(x)$. 
M\"obius equivalent metrics define
the same topology on
$X$.
When a M\"obius structure 
$\cM$
on
$X$
is fixed, we say that
$(X,\cM)$
or simply
$X$
is a {\em M\"obius space.}

A map
$f:X\to X'$
between two M\"obius spaces
is called {\em M\"obius}, if 
$f$ 
is injective and for all admissible quadruples
$Q\sub X$
$$\crt(f(Q))=\crt(Q),$$
where the cross-ratio triples are taken with respect to
some (and hence any) metric of the M\"obius structure
of
$X$
and of 
$X'$.
M\"obius maps are continuous. If a M\"obius map
$f:X\to X'$
is bijective, then 
$f^{-1}$
is M\"obius,
$f$
is homeomorphism, and the M\"obius
spaces
$X$, $X'$
are said to be {\em M\"obius equivalent}.

In general different metrics in a M\"obius structure
$\cM$ 
can look very differently. However if two metrics have the same infinitely remote
point, then they are homothetic. Since this result is crucial 
for our considerations, we state it as a lemma. 

\begin{lem}\label{lem:homothety_infinite}
Let 
$\cM$ be a M\"obius structure on a set
$X$, 
and let
$d$, $d'\in\cM$
have the same infinitely remote point
$\om\in X$.
Then there exists 
$\la>0$
such that
$d'(x,y)=\la d(x,y)$ 
for all 
$x$, $y\in X$.
\end{lem}

\begin{proof}
Since otherwise the result is trivial, we can assume that
there are distinct points
$x$, $y\in X_\om$.
Take 
$\la>0$ 
such that
$d'(x,y)=\la d(x,y)$.
If 
$z\in X_\om$, 
then
$\crt_d(x,y,z,\om)=\crt_{d'}(x,y,z,\om)$, 
hence
$(d'(x,y):d'(x,z):d'(y,z))=(d(x,y):d(x,z):d(y,z)).$
Since 
$d'(x,y)=\la d(x,y)$
we therefore obtain
$d'(x,z)=\la d(x,z)$ 
and
$d'(y,z)=\la d(y,x)$.
\end{proof}

In what follows we always consider
$X_\om=X\sm\om$
as a metric space with a metric from the M\"obius structure
for which the point 
$\om$
is infinitely remote.

A classical example of a M\"obius space is the extended
$\wh\R^n=\R^n\cup\infty=S^n$, $n\ge 1$,
where the M\"obius structure is generated by some extended
Euclidean metric on
$\wh\R^n$,
and
$\R^n\cup\infty$
is identified with the unit sphere
$S^n\sub\R^{n+1}$
via the stereographic projection. Note that Euclidean metrics 
which are not homothetic to each other generate different 
M\"obius structures by the lemma above, which however are M\"obius equivalent.

\subsection{Ptolemy spaces}
\label{subsect:Ptolemy_spaces}

A M\"obius space
$X$
is called a {\em Ptolemy space}, if it satisfies the
Ptolemy property, that is, for all admissible quadruples  
$Q\sub X$
the entries of the respective cross-ratio triple
$\crt(Q)\in\R P^2$
satisfies the triangle inequality.

The importance of the Ptolemy property comes from the following fact.
Given a metric
$d\in\cM(X)$
possibly with infinitely remote point
$\om\in X$ 
and a point
$z\in X_\om$,
the {\em metric inversion}, or m-inversion for brevity, of
$d$
of radius
$r>0$
with respect to
$z$
is a function
$d_z(x,y)=\frac{r^2d(x,y)}{d(z,x)d(z,y)}$
for all
$x$, $y\in X$
distinct from
$z$, $d_z(x,z)=\infty$
for all
$x\in X\sm\{z\}$
and
$d_z(z,z)=0$.
In particular,
$z$
is infinitely remote for 
$d_z$.
Using the standard convention we also have
$d_z(x,\om)=\frac{r^2}{d(x,z)}$.
A direct computation shows that
$d_z$
is M\"obius equivalent to
$d$.

\begin{rem}\label{rem:radius_one_inversion} When saying
about an m-inversion of a metric without specifying its radius, 
we mean that the radius is 1.
\end{rem}

In general
$d_z$
is not a metric because the triangle inequality may not
be satisfied. However, we have

\begin{pro}\label{pro:moeb_ptolemy}
A M\"obius structure
$\cM$ 
on a set
$X$ 
is Ptolemy if and only if 
$\cM$
is invariant under the metric inversion 
$d\mapsto d_z$
w.r.t. every
$z\in X$.
\end{pro}

\begin{proof} Since
$d_z$
is M\"obius equivalent to
$d$
$$(d_z(x,y):d_z(y,u):d_z(x,u))=\crt_{d_z}(x,y,z,u)
=\crt_d(x,y,z,u)$$
for 
$x$, $y$, $u\in X\sm z$.
Thus the triangle inequality for 
$d_z$
is equivalent to the Ptolemy property of
$d$.
\end{proof}

The classical example of Ptolemy space is
$\wh\R^n$
with a standard M\"obius structure as it follows from 
the proposition above. Here is the list of some known results 
on metric spaces with Ptolemy property.
A real normed vector space, which is ptolemaic, is an inner product space
(Schoenberg, 1952, \cite{Sch});
a Riemannian locally ptolemaic space is nonpositively curved (Kay, 1963, \cite{Kay});
all Bourdon and Hamenst\"adt metrics on 
$\di Y$,
where
$Y$
is CAT($-1$), generate a Ptolemy space (Foertsch-Schroeder, 2006, \cite{FS1});
a geodesic metric space is CAT(0) if and only if it is ptolemaic and 
Busemann convex, a ptolemaic proper geodesic metric space is uniquely geodesic 
(Foertsch-Lytchak-Schroeder, 2007, \cite{FLS}); any Hadamard space ptolemaic, 
a complete Riemannian manifold is ptolemaic if and only if it is a Hadamard manifold, 
a Finsler ptolemaic manifold is Riemannian (Buckley-Falk-Wraith, 2009, \cite{BFW}). 
These results allow to suggest that the Ptolemy property is a sort
of a M\"obius invariant nonpositive curvature condition.

\subsection{Spheres between points}
\label{subsect:spheres_between_points}

The notion of a sphere between two points is
a notion of M\"obius geometry and can be described in terms 
of the cross-ratio triple. Let
$X$
be a M\"obius space. Given distinct
$\om$, $\om'\in X$,
we say that 
$x$, $y\in X$
lie on some sphere between
$\om$, $\om'$
if
$\crt(\om,x,y,\om')=(1:1:\ast)$,
i.e. the first two entries are equal. In particular, both
$x$, $y$
are distinct from
$\om$, $\om'$.
One easily checks this defines an equivalence relation on
$X\sm\{\om,\om'\}$,
and any equivalence class
$S\sub X\sm\{\om,\om'\}$
is called a {\em sphere between}
$\om$, $\om'$.
M\"obius maps preserve the sets of spheres
between points: if
$\phi:X\to X$
is a M\"obius map,
and
$S$
is a sphere between
$\om$
and
$\om'$,
then
$\phi(S)$
is a sphere between
$\phi(\om)$
and
$\phi(\om')$.

If 
$\om$
is infinitely remote for some metric 
$d$
of the M\"obius structure then
$$S=\set{x\in X}{$d(x,\om')=r$}=S_r^d(\om')$$
for some 
$r>0$,
which justifies our terminology.

Given distinct points 
$x$, $y$, $\om$, $\om'\in X$
the symmetries of the cross-ratio triple implies that
$x$
and
$y$
lie on a some sphere between
$\om$
and
$\om'$
if and only if
$\om$
and
$\om'$
lie on a some sphere between
$x$
and
$y$.

If we take some point on the sphere as infinitely remote,
then the sphere becomes a bisector w.r.t. a respective metric. 

\begin{lem}\label{lem:bisector} Let 
$S\sub X$
be a sphere between distinct
$a$, $a'\in X$.
Then for every
$\om\in S$
the set 
$S_\om=S\sm\{\om\}$
is the bisector in
$X_\om$
between
$a$, $a'$, $S_\om=\set{x\in X_\om}{$d(x,a)=d(x,a')$}$.
\end{lem}

\begin{proof} In the space
$X_\om$
we have
$\crt(a,\om,x,a')=(d(x,a'):d(x,a):\ast)$.
Hence,
$x\in S_\om$
if and only if
$d(x,a)=d(x,a')$. 
\end{proof}

\subsection{Circles in Ptolemy spaces}
\label{subsect:circles_ptolemy_spaces}

A Ptolemy circle in a M\"obius space 
$X$ 
is a subset 
$\si\sub X$ 
homeomorphic to 
$S^1$  
such that for every quadruple
$(x,y,z,u)\in\si$
of distinct points the equality 
\begin{equation}\label{eq:PT_eq}
d(x,z)d(y,u)=d(x,y)d(z,u)+d(x,u)d(y,z)
\end{equation}
holds for some and hence for any metric 
$d$
of the M\"obius structure , where it is supposed that the pair
$(x,z)$
separates the pair
$(y,u)$,
i.e.
$y$
and
$u$
are in different components of
$\si\sm\{x,z\}$.
Recall the classical Ptolemy theorem that four points 
$x$, $y$, $z$, $u$ 
of the Euclidean plane lie on a circle (in this order) if and
only if their distances satisfy
the Ptolemy equality (\ref{eq:PT_eq}).

Let
$\si$ 
be a Ptolemy circle passing through the infinitely remote point
$\om$
for some metric
$d\in\cM$
and let
$\si_\om=\si\sm\om$.
Then (\ref{eq:PT_eq}) says that for 
$x$, $y$, $z\in\si_\om$
(in this order)
$d(x,y)+d(y,z)=d(x,z)$,
i.e. it implies that
$\si_\om$
is a geodesic, actually a complete geodesic isometric to
$\R$.

We recall the following fact from \cite{FS2}. 

\begin{pro} \label{pro:moebch-circ}
Let 
$\si$ 
and 
$\si'$ 
be Ptolemy circles. Let 
$x_1$, $x_2$, $x_3$ 
and 
$x'_1$, $x'_2$, $x'_3$ 
be distinct points on 
$\si$ 
respectively on 
$\si'$.
Then there exists a unique M\"obius homeomorphism
$\varphi:\si\to \si'$ 
with 
$\varphi(x_i)=x'_i$.
\qed
\end{pro}

In particular all Ptolemy circles are M\"obius equivalent.
The standard metric models of a circle are 
$(\wh\R,d)$,
where 
$d$
is the standard Euclidean metric, or
$(S^1,d_c)$,
where
$d_c$ 
is the chordal metric on 
$S^1$,
i.e. the metric induced by the standard embedding
$S^1\sub\R^2$
as a unit circle. These two standard realizations of a circle
are M\"obius equivalent via the  stereographic projection.
Note that by Lemma~\ref{lem:homothety_infinite} there is up to homothety only 
one metric on a circle with a infinitely remote point, while there are plenty 
of bounded metrics (for a description of all Ptolemy metrics on
$S^1$ 
see \cite{FS2}).

\subsection{Duality between Busemann and distance functions}
\label{subsect:duality_dist_busemann}

Let
$X$ 
be a Ptolemy space,
$d$
a metric of the M\"obius structure with infinitely remote point
$\om$, $X_\om=X\sm\om$.
If a Ptolemy circle
$\si\sub X$
passes through
$\om$,
then
$l=\si_\om$ 
is isometric w.r.t.
$d$
to a geodesic line and it is called a {\em Ptolemy} line
in
$X_\om$.

With every oriented Ptolemy line 
$l\sub X_\om$
and every point 
$\om'\in l$
we associate a function
$b:X_\om\to\R$,
called a {\em Busemann function} of
$l$, 
as follows. Given
$x\in X_\om$,
the difference
$d(x,y)-d(\om',y)$
is nonincreasing by triangle inequality as
$y\in l$
goes to infinity according the orientation of
$l$, $y>\om'$.
Thus the limit
$b(x)=\lim_{l\ni y\to\infty}(d(x,y)-d(\om',y))$
exists. Note that
$b(\om')=0$
and
$b(x)=-d(\om',x)$
for all
$l\ni x>\om'$.

For any Ptolemy space
$X$
there is a remarkable duality between Busemann and distance functions
which is described as follows. 

On a Ptolemy line
$l\sub X_\om$,
we fix
$\om'\in l$,
and let
$d'$
be the m-inversion of
$d$
w.r.t.
$\om'$.
Then
$d'$
is a metric of the M\"obius structure with infinitely remote point
$\om'$.
In particular,
$l'=\si_{\om'}$
is a Ptolemy line in
$X_{\om'}$.
One easily checks that
$d$
is the m-inversion of
$d'$
w.r.t.
$\om$,
that is, the inversion operation is involutive.

Let
$c:\R\to X_{\om'}$
be a unit speed parameterization of
$l'$
with
$c(0)=\om$, $b^\pm:X_\om\to\R$
the {\em opposite} Busemann functions of
$l$,
that is, associated with opposite ends of
$l$,
which are normalized by
$b^\pm(\om')=0$
and
$b^+\circ c(t)<0$
for
$t>0$, $b^-\circ c(t)<0$
for
$t<0$.
Since
$d(x,\om')\cdot d'(x,\om)=1$
for every
$x\in X\sm\{\om,\om'\}$,
we have
$b^\pm\circ c(t)=\mp 1/t$
for all
$t\neq 0$.

\begin{pro}\label{pro:duality_dist_busemann} For all
$x\in X\sm\{\om,\om'\}$
we have
\begin{equation}\label{eq:duality}
b^\pm(x)=\frac{d^\pm}{dt}\ln d'(x,c(t))|_{t=0},
\end{equation} 
where
$\frac{d^\pm}{dt}$
is the right/$-$left derivative.
\end{pro}

\begin{rem} Note that the left hand side of (\ref{eq:duality})
is computed in the space
$X_\om$
while the right hand side in the inverted space
$X_{\om'}$.
The equality (\ref{eq:duality}) is our first example of 
duality equalities which appear in different places of 
the paper, see e.g. Remark~\ref{rem:first_variation}.
\end{rem}

\begin{proof} We first note that the function
$t\mapsto d'(x,c(t))$
is convex by the Ptolemy condition, and thus it has
the right and the left derivatives at every point. Hence, the
right hand side of Equation~(\ref{eq:duality}) is well
defined. By definition,
$d(x,y)=\frac{d'(x,y)}{d'(\om,x)d'(\om,y)}$
and
$d(x,\om')=\frac{1}{d'(x,\om)}$
for all
$x$, $y\in X_\om$.
Now, we compute 
\begin{eqnarray*}
d(x,c(t))-d(\om',c(t))&=&\frac{d'(x,c(t))}{d'(\om,x)d'(\om,c(t))}
  -\frac{1}{d'(\om,c(t))}\\
&=&\frac{1}{|t|d'(x,c(0))}
   \left(d'(x,c(t))-d'(x,c(0)\right)
\end{eqnarray*}
for all
$t\neq 0$,
because
$d'(x,\om)=d'(x,c(0))$
and
$d'(\om,c(t))=|t|$.
Since
$b^\pm(x)=\lim_{t\to\pm 0}(d(x,c(t))-d(\om',c(t)))$,
we obtain
$$b^\pm(x)=\frac{d^\pm}{dt}\ln d'(x,c(t))|_{t=0}.$$
\end{proof}

Given a Ptolemy circle
$\si\in X$
and distinct points 
$\om$, $\om'\in\si$,
we denote by
$D_{\si,\om}^{\om'}$
the subset in
$X_{\om'}$
which consists of all
$x$
such that
$\om$
is a closest to 
$x$
point in the geodesic line
$\si_{\om'}$
(w.r.t. the metric of
$X_{\om'}$).

\begin{lem}\label{lem:omega_closest_subset} Let
$X$
be a Ptolemy space. Then for every Ptolemy circle
$\si\sub X$
and each pair of distinct points
$\om$, $\om'\in\si$
we have
\begin{equation}\label{eq:dist_busemann}
D_{\si,\om}^{\om'}\cup\om'=B_{\si,\om'}^\om\cup\om,
\end{equation}
where
$B_{\si,\om'}^\om=\set{x\in X_\om}{$b^+(x)\ge 0\ \text{and}\ b^-(x)\ge 0$}$,
$b^\pm:X_\om\to\R$
are the opposite Busemann functions of the Ptolemy line
$\si_\om\sub X_\om$
with
$b^\pm(\om')=0$.
\end{lem}

\begin{proof} Denote by
$d'$
the metric of
$X_{\om'}$
and let
$c:\R\to X_{\om'}$
be the unit speed parameterization of the Ptolemy line
$\si_{\om'}\sub X_{\om'}$
such that
$c(0)=\om$
and
$b^\pm\circ c(t)=\mp 1/t$,
see the paragraph preceding Proposition~\ref{pro:duality_dist_busemann}.
For every
$x\in D_{\si,\om}^{\om'}$
we have 
$\frac{d^+}{dt}d'(x,c(t))_{t=0}\ge 0$
for the right derivative, and 
$-\frac{d^-}{dt}d'(x,c(t))_{t=0}\le 0$
for left derivative because
$t=0$
is a minimum point of the convex function
$t\mapsto d'(x,c(t))$.
Equation~(\ref{eq:duality}) implies that 
$x\in B_{\si,\om'}^\om$.

Assume that
$b^+(x)\ge 0$
and
$b^-(x)\ge 0$
for some
$x\in X\sm\{\om,\om'\}$.
Equation~(\ref{eq:duality}) implies that the right derivative
$\frac{d^+}{dt}d'(x,c(t))_{t=0}\ge 0$
and the left derivative
$-\frac{d^-}{dt}d'(x,c(t))_{t=0}\le 0$.
Thus
$t=0$
is a minimum point of the convex function
$t\mapsto d'(x,c(t))$
and hence
$x\in D_{\si,\om}^{\om'}$.
\end{proof}

\subsection{Busemann flat Ptolemy spaces}
\label{subsect:busemann_flat_ptolemy}

A Ptolemy space
$X$
is said to be {\em (Busemann) flat} if for every Ptolemy
circle
$\si\sub X$
and every point
$\om\in\si$,
we have
\begin{equation}\label{eq:busemann_flat}
 b^++b^-\equiv\const
\end{equation}
for opposite Busemann functions
$b^\pm:X_\om\to\R$
associated with Ptolemy line
$\si_\om$.

Busemann flatness implies that Busemann functions
are affine in the following sense.

\begin{pro}\label{pro:busemann_affine} 
Let 
$X$
be a Busemann flat ptolemaic space and let
$l$, $l'\in X_\om$
be Ptolemy lines. Then the Busemann functions of
$l$
are affine functions on
$l'$.
\end{pro}

\begin{proof} The opposite Busemann functions 
$b^+$, $b^-$
of
$l$
are convex on
$l'$,
see \cite{FS2}. Since the sum
$b^++b^-$
is affine on
$l'$
by Busemann flatness,
$b^+$
and
$b^-$
are affine.
\end{proof}

The property of Busemann flatness is equivalent to that any horospheres of
$b^+$, $b^-$
coincide whenever they have a common point. Thus the horosphere
$H_{\si,\om'}^\om\sub X_\om$
of
$\si_\om$
through
$\om'\in\si_\om$
is well defined in a flat Ptolemy space.

\begin{pro}\label{pro:busemann_flat} A Ptolemy space
$X$
is flat if and only if for every
$\om\in X$
and every
$x\in X_\om$
the distance function
$d(x,\cdot)$
is $C^1$-smooth along any Ptolemy line
$l\sub X_\om$, $l\not\ni x$.
\end{pro}

\begin{proof} Assume that distance functions are $C^1$-smooth
along Ptolemy lines. We fix
$\om\in X$, 
a Ptolemy line
$l\sub X_\om$,
and let
$b^\pm$
be opposite Busemann functions of
$l$.
We suppose W.L.G. that
$b^\pm(\om')=0$
for some point
$\om'\in l$.
Then
$b^++b^-=0$
along
$l$.
Equation~(\ref{eq:duality}) implies that in fact
$b^+(x)+b^-(x)=0$
for every
$x\in X_\om$. 
Thus
$X$
is flat.

Conversely, assume that 
$X$
is flat. Given
$\om'\in X$,
a Ptolemy line
$l'\in X_{\om'}$
and
$x\in X_{\om'}\sm l'$,
we show that the distance function
$d'(x,\cdot)$
in
$X_{\om'}$
is $C^1$-smooth along
$l'$
at every point
$\om\in l'$.

Let
$c:\R\to X_{\om'}$
by a unit speed parameterization of
$l'$
with
$c(0)=\om$, $b^\pm:X_\om\to\R$
the opposite Busemann function associated 
with the Ptolemy line
$l=(l'\cup\om')\sm\om\sub X_\om$
such that
$b^\pm(\om')=0$, $b^+\circ c(t)<0$
for all
$t>0$.
Then 
$b^++b^-\equiv 0$
by the assumption, and
by Proposition~\ref{pro:duality_dist_busemann} we have 
$\frac{d^+}{dt}d'(x,c(t))|_{t=0}=-\frac{d^-}{dt}d'(x,c(t))|_{t=0}$,
where
$\frac{d^+}{dt}$
is the right derivative
and
$-\frac{d^-}{dt}$
is the left derivative. Hence
$d'(x,\cdot)$
is $C^1$-smooth.
\end{proof}

By Proposition~\ref{pro:busemann_flat}, the duality equation~(\ref{eq:duality}) 
in a flat Ptolemy space
$X$
takes the following form
\begin{equation}\label{eq:smooth_duality}
b^\pm(x)=\pm\frac{d}{dt}\ln d'(x,c(t))|_{t=0}.
\end{equation}

\begin{exa}\label{exa:hyp} The Ptolemy space
$\wh\hyp^n=\hyp^n\cup\{\infty\}$, $n\ge 2$,
generated by the real hyperbolic space
$\hyp^n$,
is not flat because the equality
$b^++b^-\equiv\const$
is violated in
$\hyp^n$.
(Recall that
$\hyp^n$
possesses the Ptolemy property and thus it generates
a Ptolemy space by taking all metrics on
$\wh\hyp^n$
which are M\"obius equivalent to the metric of
$\hyp^n$.)
Note that the distance function
$d(x,\cdot):\hyp^n\to\R$
is smooth for every
$x\in\hyp^n$
along any geodesic line
$l$, $x\not\in l\sub\hyp^n$.
This does not contradict Proposition~\ref{pro:busemann_flat} 
because the m-inversion of
$d$
with respect to any point
$x\in\hyp^n$
has a singularity at the infinity point of
$\wh\hyp^n$.
\end{exa}

In flat Ptolemy spaces, the duality between distance and 
Busemann functions takes the following form.

\begin{lem}\label{lem:flat_duality} Let
$X$
be a flat Ptolemy space,
$\si\sub X$
a Ptolemy circle, and
$\om$, $\om'\sub\si$
distinct points. Let
$H_{\si,\om'}^\om\sub X_\om$
be the horosphere through
$\om'$
of the Ptolemy line
$\si_\om\sub X_\om$,
$D_{\si,\om}^{\om'}\sub X_{\om'}$
the set of all
$x\in X_{\om'}$
such that
$\om$
is the closest to
$x$
point in the Ptolemy line
$\si_{\om'}$.
Then
\begin{equation}\label{eq:horosphere_distance}
H_{\si,\om'}^\om\cup\om=D_{\si,\om}^{\om'}\cup\om'.
\end{equation}
\end{lem}

\begin{proof} In a flat Ptolemy space we have
$H_{\si,\om'}^\om=B_{\si,\om'}^\om$
because level sets of opposite Busemann functions associated with 
a Ptolemy line coincide when they have a common point. On the other hand, by duality, Lemma~\ref{lem:omega_closest_subset}, we have
$B_{\si,\om'}^\om\cup\om=D_{\si,\om}^{\om'}\cup{\om'}$.
\end{proof}

\section{M\"obius spaces with circles and many space inversions}
\label{sect:many_circles_auto}

We begin this section with discussion of what is a space
inversion of an arbitrary M\"obius space.

\subsection{Space inversions}
\label{subsect:space_inversions}

We define a {\em space inversion}, or s-inversion for brevity,
w.r.t. distinct
$\om$, $\om'\in X$
and a sphere 
$S\sub X$
between
$\om$, $\om'$
as a M\"obius involution
$\phi=\phi_{\om,\om',S}:X\to X$, $\phi^2=\id$,
without fixed points such that
\begin{itemize}
 \item[(1)] $\phi(\om)=\om'$ (and thus
$\phi(\om')=\om$);
 \item[(2)] $\phi$
preserves
$S$,
$\phi(S)=S$;
 \item[(3)] $\phi(\si)=\si$
for any Ptolemy circle
$\si\sub X$
through
$\om$, $\om'$.
\end{itemize}

\begin{rem}\label{rem:sinversion_motivation} Motivation of this
definition comes from the fact that in the case
$X=\di Y$,
where
$Y$
is a symmetric rank one space of non-compact type, any central symmetry
$f:Y\to Y$
with a center 
$o\in Y$, $f(o)=o$,
induces a space inversion
$\di f=\phi_{\om,\om',S}:X\to X$,
where a geodesic line 
$l=(\om,\om')\sub Y$
with the end points
$\om$, $\om'$
passes through
$o$,
and
$S\sub X$
is a sphere between
$\om$, $\om'$.
\end{rem}

\begin{rem}\label{rem:weak_unique} In general, there is no reason 
that an s-inversion
$\phi=\phi_{\om,\om',S}$
is uniquely determined by its data
$\om$, $\om'$, $S$.
However, if 
$\phi'$
is another s-inversion with the same data, then
it coincides with 
$\phi$
along any Ptolemy circle through
$\om$, $\om'$
because any M\"obius automorphism of a Ptolemy circle
is uniquely determined by values at three distinct points,
see Proposition~\ref{pro:moebch-circ}.
\end{rem}

A M\"obius automorphism
$\phi:X\to X$
of a M\"obius space induces a map
$\phi^\ast:\cM\to\cM$, $(\phi^\ast d)(x,y)=d(\phi(x),\phi(y))$
for every metric
$d\in\cM$
and each
$x$, $y\in X$,
where
$\cM$
is the M\"obius structure of
$X$.
Note that a metric inversion of a bounded metric 
cannot be induced by any M\"obius automorphism
$X\to X$,
because a metric inversion w.r.t.
$\om\in X$
has
$\om$
as the infinitely remote point.

However if 
$\phi$
is a M\"obius automorphism with
$\phi(\om)=\om'$
and
$d$ 
is a metric with infinite remote point
$\om$,
then
$\om'$
is infinitely remote for
$\phi^\ast d$.
Thus by Lemma~\ref{lem:homothety_infinite}
$\phi^\ast d=\la d'$
for some
$\la > 0$, 
where
$d'$
is the m-inversion of
$d$
w.r.t.
$\om'$,
i.e.
$$(\phi^\ast d)(x,y)=\frac{\la d(x,y)}{d(x,\om')d(y,\om')}$$
for each
$x$, $y\in X$
which are not equal to
$\om'$
simultaneously.

In the case that
$\phi=\phi_{\om,\om',S}$
an s-inversion we can say in addition that

\begin{lem}\label{lem:sinversion_minversion} We have
 $$(\phi^\ast d)(x,y)=\frac{r^2d(x,y)}{d(x,\om')d(y,\om')}$$
where 
$r=r(d)$
is determined by the equation
$S=S^d_r(\om')$.
\end{lem}

\begin{proof}
We already know that
$(\phi^\ast d)(x,y)=\frac{\la d(x,y)}{d(x,\om')d(y,\om')}$
for each
$x$, $y\in X$
which are not equal to
$\om,$
simultaneously. We compute
$\la$
by taking
$x\in S$, $y=\phi(x)$.
Then
$\phi(y)=x$
and since
$(\phi^\ast d)(x,y)=d(x,y)$, $d(x,\om')=r=d(y,\om')$,
we have
$\la=r^2$. 
\end{proof}

Contrary to metric inversions which always exist, in general 
there is no reason for a space inversion to exist. The existence
of many space inversions put severe restrictions on the space.
In the following theorem we recover some geometric properties 
of a space satisfying the assumptions of Theorem~\ref{thm:moebius}.
Recall the basic properties, see Introduction

\noindent
(E) Existence: there is at least one Ptolemy circle in
$X$.

\noindent
(I) Inversion: for each distinct
$\om$, $\om'\in X$
and every sphere
$S\sub X$
between
$\om$, $\om'$
there is a unique space inversion
$\phi_{\om,\om',S}:X\to X$
w.r.t.
$\om$, $\om'$
and
$S$.

\begin{thm}\label{thm:basic_ptolemy} Let
$X$
be a compact Ptolemy space with properties 
($E$) and (I). Then for every
$\om\in X$
there is a 1-Lipschitz submetry
$\pi_\om:X_\om\to B_\om$
with the base
$B_\om$
isometric to an Euclidean space
$\R^k$, $k\ge 1$, 
such that any M\"obius automorphism
$\phi:X\to X$
with
$\phi(\om)=\om'$
induces a homothety
$\ov\phi:B_\om\to B_{\om'}$
with
$\pi_{\om'}\circ\phi=\ov\phi\circ\pi_\om$.

The fibers of
$\pi_\om$,
also called
$\K$-lines,
have the property

\begin{itemize}
 \item[($\K$)] given a $\K$-line
$F\sub X_\om$
and
$x\in X_\om\sm F$,
there is a unique Ptolemy line
$l\sub X_\om$
through
$x$
that intersects
$F$.
\end{itemize}
\end{thm}

\begin{rem}\label{rem:submetry} Recall that a map
$f:X\to Y$
between metric spaces is called a {\em submetry}
if for every ball 
$B_r(x)\sub X$
of radius 
$r>0$
centered at
$x$
its image
$f(B_r(x))$
coincides with the ball
$B_r(f(x))\sub Y$. 
\end{rem}

The proof occupies the rest of sect.~\ref{sect:many_circles_auto} 
and sections~\ref{sect:slope}--\ref{sect:fibration}.
In the remaining parts of this section we give first
consequences of the properties (E) and (I).
In particular we prove the existence of 
homotheties (H), the existence of shifts, two point homogeneity 
and the Busemann flatness of 
$X$.

In what follows, we always consider the weak topology on
the group 
$\aut X$
of M\"obius automorphisms of
$X$,
i.e. a sequence
$\phi_i\in\aut X$
converges to
$\phi\in\aut X$, $\phi_i\to\phi$,
if and only if
$\phi_i(x)\to\phi(x)$
for every 
$x\in X$.

\subsection{M\"obius automorphisms and homothety property (H)}
\label{subsect:moebius_automorphisms}

In this section we establish some important additional
properties of a Ptolemy space
$X$
which follow from (E) and (I).

Given two distinct points
$\om$, $\om'\in X$,
we denote by
$C_{\om,\om'}$
the set of all the Ptolemy circles
$\si\sub X$
through
$\om$, $\om'$,
and by
$\Ga_{\om,\om'}$
the group of M\"obius automorphisms
$\phi:X\to X$
such that
$\phi(\om)=\om$, $\phi(\om')=\om'$, $\phi(\si)=\si$
and
$\phi$
preserves orientations of
$\si$ 
for every
$\si\in C_{\om,\om'}$.

\begin{pro}\label{pro:homothety_property} Any Ptolemy space
$X$ 
with properties (E) and (I) possesses the following property

\noindent
(H) Homothety: for each distinct
$\om$, $\om'\in X$
the group
$\Ga_{\om,\om'}$
acts transitively on every arc of
$\si\sm\{\om,\om'\}$
for every circle
$\si\in C_{\om,\om'}$.
\end{pro}

\begin{rem}\label{rem:homothety} If one of the points
$\om$, $\om'$
is infinitely remote for a metric 
$d$
of the M\"obius structure, then every 
$\ga\in\Ga_{\om,\om'}$
is a homothety w.r.t.
$d$.
This is why we use (H) for the notation of the property above.
\end{rem}

\begin{proof} We assume that 
$\om'$
is infinitely remote for a metric 
$d\in\cM$.
Then for any
$\si\in C_{\om,\om'}$
the curve
$\si_{\om'}=\si\sm\om'$
is a Ptolemy line w.r.t.
$d$,
and any
$\ga\in\Ga_{\om,\om'}$
acts on
$\si_{\om'}$
as a preserving orientation homothety.

Composing s-inversions
$\phi=\phi_{\om,\om',S}$, $\phi'=\phi_{\om,\om',S'}$
of
$X$,
where  
$S$, $S'\sub X$
are spheres between
$\om$, $\om'$,
we obtain a M\"obius automorphism
$\ga=\phi'\circ\phi$
with properties
$\ga(\om)=\om$, $\ga(\om')=\om'$
and
$\ga(\si)=\si$
for any Ptolemy circle
$\si\in C_{\om,\om'}$.
Having no fixed point, both
$\phi$, $\phi'$
preserve orientations of
$\si$.
Hence, 
$\ga$
preserves its orientations, thus
$\ga$
acts on every arc of
$\si\sm\{\om,\om'\}$
as a homothety. That is,
$\ga\in\Ga_{\om,\om'}$.

Let
$r$, $r'>0$
be the radii of
$S$, $S'$
respectively w.r.t. the metric 
$d$, $S=S_r^d(\om)$, $S'=S_{r'}^d(\om)$.
Then for every
$x\in X\sm\{\om,\om'\}$
we have
$d(\phi(x),\om)=\frac{r^2}{d(x,\om)}$
and
$d(\ga(x),\om)=d(\phi'\circ\phi(x),\phi'(\om'))
=\frac{r'^2}{d(\phi(x),\om)}=(r'/r)^2d(x,\om)$.
Therefore, the dilatation coefficient of
$\ga$
equals
$\la:=(r'/r)^2$,
and it can be chosen arbitrarily by changing
$S$, $S'$
appropriately.
\end{proof}

\begin{cor}\label{cor:weak_unique} Any two distinct Ptolemy 
circles in a Ptolemy space with properties (E) and (I)
have in common at most two points.
\end{cor}

\begin{proof} Assume
$\om$, $\om'$, $x\in\si\cap\si'$
are distinct common points of Ptolemy circles
$\si$, $\si'\sub X$.
We have
$\ga(x)\in\si\cap\si'$
for every
$\ga\in\Ga_{\om,\om'}$.
Then by property (H), the arcs of
$\si$
and 
$\si'$
between
$\om$, $\om'$
which contain
$x$
coincide. Taking
$\om''$
inside of this common arc and applying the same
argument to
$\om'$, $\om''$, $x=\om$,
we obtain 
$\si=\si'$. 
\end{proof}

\subsection{Busemann parallel lines, pure homotheties and shifts}
\label{subsect:parallel_lines_pure_homothethies_shifts}

In this section we assume that the compact Ptolemy space
$X$
possesses the properties (E) an (I). 

We say that Ptolemy lines
$l$, $l'\sub X_\om$
are {\em Busemann parallel} if 
$l$, $l'$
share Busemann functions, that is, any Busemann function
associated with
$l$
is also a Busemann function associated with
$l'$
and vice versa.

\begin{lem}\label{lem:unique_line} Let
$l$, $l'\sub X_\om$
be Ptolemy lines with a common point,
$o\in l\cap l'$, $b:X_\om\to\R$
a Busemann function of
$l$
with
$b(o)=0$.
Assume 
$b\circ c(t)=-t=b\circ c'(t)$
for all 
$t\ge 0$
and for appropriate unit speed parameterizations
$c$, $c':\R\to X_\om$
of
$l$, $l'$
respectively
with
$c(0)=o=c'(0)$.
Then
$l=l'$.
In particular, Busemann parallel Ptolemy lines coincide if
they have a common point.
\end{lem}

\begin{proof} We show that the concatenation of
$c|(-\infty,0]$
with
$c'|[0,\infty)$
is also a Ptolemy line. Then
$l=l'$
by Corollary~\ref{cor:weak_unique}. It suffices to show	that for
$s$, $t\ge 0$
we have
$|c(-s)c'(t)|=t+s$.
By triangle inequality we have
$|c(-s)c'(t)|\le t+s$.
Letting 
$t_i\to \infty$ 
we have
$|c'(t)c(t_i)|-t_i\to b\circ c'(t)=-t$.
Thus by triangle inequality again, we have
$$|c(-s)c'(t)| \ge |c(-s)c(t_i)| - |c'(t)c(t_i)| = (t_i+s)-|c'(t)c(t_i)|\to t+s.$$
Thus
$|c(-s)c'(t)|=t+s$.
\end{proof}

Next, we show that a sublinear divergence of Ptolemy lines
is equivalent for them to be Busemann parallel.

\begin{lem}\label{lem:busparallel_sublinear}
If two Ptolemy lines 
$l$, $l'\sub X_\om$
are Busemann parallel, then
they diverge at most sublinearly, that is
$|c(t)c'(t)|/|t|\to 0$
as
$|t|\to\infty$
for appropriate unit speed parameterizations
$c$, $c'$
of 
$l$, $l'$.

Conversely, if
$|c(t_i)c'(t_i)|/|t_i|\to 0$
for some sequences
$t_i\to\pm\infty$,
then the lines
$l$, $l'$
are Busemann parallel.
\end{lem}

\begin{proof} Let
$c$, $c':\R\to X_\om$
be unit speed parameterizations of Busemann parallel lines
$l$, $l'\sub X_\om$
respectively, and a common Busemann function
$b:X_\om\to\R$
such that
$b\circ c(t)=b\circ c'(t)=-t$
for all
$t\in\R$.
Let
$\mu(t):=|c(t)c'(t)|$.
We claim that
$\mu(t)/|t|\to 0$
for
$t\to\pm\infty$.
Assume to the contrary, that W.L.G. there exists a sequence
$t_i\to\infty$ 
with
$\mu(t_i)/t_i\ge a >0$.

By the homothety property~(H) there exists a homothety
$\phi_i$
of
$X_\om$
with factor
$1/t_i$
such that
$\phi_i\circ c(s)=c(s/t_i)$
for all
$s\in\R$.
Note that
$c'_i(s)=\phi_i\circ c'(t_i s)$
is a unit speed parameterization of
the Ptolemy line
$\phi_i(l')$.
For fixed
$i$
we calculate
\begin{align*}
b\circ c'_i(t)) &= \lim_{s\to \infty}(|c'_i(t) c(s)|-s)
= \lim_{s\to \infty}(|\phi_i(c'(tt_i))c(s)|-s) \\
&= \lim_{s\to \infty}(|\phi_i(c'(tt_i))c(s/t_i)|-s/t_i)
= \lim_{s\to \infty}(|\phi_i(c'(tt_i))\phi_i(c(s))|-s/t_i)\\
&= \lim_{s\to \infty}\frac{1}{t_i}(|c'(tt_i)c(s)|-s)
=\frac{1}{t_i}b(c'(tt_i))=\frac{1}{t_i}(-tt_i )=-t
\end{align*}
for all 
$t\in\R$.
The Ptolemy lines
$\phi_i(l')$
subconverge to a Ptolemy line
$l''$
through
$c(0)$.
If
$c'':\R\to X$
is the limit unit speed parameterization of
$l''$,
then
$b\circ c''(t)=-t$
for all 
$t\in\R$,
and
$|c''(1)c(1)|\ge a>0$.
This contradicts Lemma \ref{lem:unique_line} by which
$l=l'$
and thus
$c''(t)=c(t)$
for all 
$t\in\R$.

Conversely, assume 
$c$, $c':\R\to X_\om$
are unit speed parameterizations of Ptolemy lines
$l$, $l'\sub X_\om$
with
$c(0)=o$, $c'(0)=o'$
such that
$b(o)=b(o')=0$
for the Busemann function
$b:X_\om\to\R$
of
$l$
with
$b\circ c(t)=-t$, $t\in\R$,
and
$\mu(t_i)/t_i\to 0$
for some sequence
$t_i\to\infty$,
where
$\mu(t)=|c(t)c'(t)|$.
Let
$b':X_\om\to\R$
be the Busemann function of
$l'$
with
$b'\circ c'(t)=-t$.
Applying the Ptolemy inequality to the cross-ratio triple
$\crt(Q_i)$
of the quadruple
$Q_i=(o,c(t_i),c'(t_i),o')$,
we obtain 
$$\left||oc'(t_i)||o'c(t_i)|-|oc(t_i)||o'c'(t_i)|\right|
  \le|oo'||c(t_i)c'(t_i)|.$$
Using
$|oc(t_i)|=t_i=|o'c'(t_i)|$,
$|o'c(t_i)|=b(o')+t_i+o(1)$,
$|oc'(t_i)|=b'(o)+t_i+o(1)$,
and
$|c(t_i)c'(t_i)|=\mu(t_i)=o(1)t_i$,
we obtain
$$|(b'(o)+t_i+o(1))(b(o')+t_i+o(1))-t_i^2|\le|oo'|o(1)t_i,$$
thus
$|b'(o)|\le o(1)$
and hence
$b'(o)=b(o')=0$.

Finally, for an arbitrary
$x\in X_\om$
consider the quadruple
$Q_{x,i}=(x,c(t_i),c'(t_i),o)$.
By the same argument as above, we have
$$\left||xc'(t_i)|t_i-|xc(t_i)||oc'(t_i)|\right|\le
  |ox|\mu(t_i).$$
Using
$|oc'(t_i)|=b'(o)+t_i+o(1)=t_i+o(1)$,
$|xc'(t_i)|=b'(x)+t_i+o(1)$,
$|xc(t_i)|=b(x)+t_i+o(1)$,
we finally obtain 
$|b'(x)-b(x)|\le o(1)$
and hence
$b(x)=b'(x)$.
Therefore, the lines
$l$, $l'$
are Busemann parallel.
\end{proof}

Given
$x$, $x'\in X_\om$,
we construct an isometry
$\eta_{xx'}:X_\om\to X_\om$
called a {\em shift} as follows. We take a sequence
$\la_i\to\infty$
and using the homothety property (H) for every 
$i$
consider homotheties
$\phi_i\in\Ga_{\om,x}$, $\psi_i\in\Ga_{\om,x'}$
with coefficient
$\la_i$.
Then
$\eta_i=\psi_i^{-1}\circ\phi_i$
is an isometry of
$X_\om$
for every
$i$
because the coefficient of the homothety
$\eta_i$
is 1. Furthermore, we have
$|\eta_i(x)x'|=\la_i^{-1}|xx'|\to 0$
as 
$i\to\infty$.
Since
$X$
is compact, the sequence
$\eta_i$
subconverges to an isometry
$\eta=\eta_{xx'}$
with
$\eta(x)=x'$.
The term shift for 
$\eta$
is justified by the following

\begin{lem}\label{lem:shift_busemann_parallel} A shift 
$\eta_{xx'}$
moves any Ptolemy line
$l$
through
$x$
to a Busemann parallel Ptolemy line 
$\eta_{xx'}(l)$
through
$x'$.
\end{lem}

\begin{proof} We show that the line 
$l'=\eta_{xx'}(l)$
cannot have at least linear divergence with 
$l$.
Assume to the contrary that
$\mu(t)\ge at$
for some 
$a>0$
and all
$t>0$,
where
$\mu(t)=|c(t)c'(t)|$, $c:\R\to X_\om$
is a unit speed parameterization of
$l$
with 
$c(0)=x$, $c'=\eta_{xx'}\circ c$.

Recall that
$\eta_{xx'}=\lim\eta_i$,
where
$\eta_i=\psi_i^{-1}\circ\phi_i$,
and
$\phi_i\in\Ga_{\om,x}$, $\psi_i\in\Ga_{\om,x'}$
are homotheties with the same coefficient
$\la_i\to\infty$.
By definition of the groups
$\Ga_{\om,x}$, $\Ga_{\om,x'}$,
we have
$\phi_i(l)=l$, $\psi_i(l')=l'$.
We take
$y=c(1)$, $y'=c'(1)$.
Then for 
$y_i=\phi_i(y)=c(\la_i)$
we have
$|y_ic'(\la_i)|=\mu(\la_i)\ge a\la_i$.
Thus for 
$y_i'=\psi_i^{-1}(y_i)$
the estimate
$|y_i'y'|=|\psi_i^{-1}(y_i)\psi_i^{-1}\circ c'(\la_i)|\ge a$
holds for all 
$i$
in contradiction with 
$y_i'\to y'$
as
$i\to\infty$.

Therefore, there are sequences
$t_i\to\pm\infty$, 
with 
$\mu(t_i)=o(1)|t_i|$.
By Lemma~\ref{lem:busparallel_sublinear} the lines
$l$, $l'$
are Busemann parallel. 
\end{proof}

From Lemma~\ref{lem:unique_line} and 
Lemma~\ref{lem:shift_busemann_parallel} we immediately obtain

\begin{cor}\label{cor:busparallel_foliation} Given a Ptolemy line
$l\sub X_\om$,
through any point 
$x\in X_\om$
there is a unique Ptolemy line
$l(x)$
Busemann parallel to
$l$.
\qed
\end{cor}

Recall that any M\"obius map 
$\phi:X\to X$
with 
$\phi(\om)=\om$
for
$\om\in X$
acts on
$X_\om$
as a homothety. A homothety
$\phi:X_\om\to X_\om$
is said to be {\em pure} if it preserves any
foliation of
$X_\om$
by Busemann parallel Ptolemy lines.

\begin{lem}\label{lem:pure_homothety} For every
$o\in X_\om$
the group
$\Ga_{\om,o}$
consists of pure homotheties. In particular,
every shift of
$X_\om$
preserves any foliation of
$X_\om$
by Busemann parallel Ptolemy lines. 
\end{lem}

\begin{proof} Let
$l\sub X_\om$
be a Ptolemy line through
$o$, $b:X_\om\to\R$
a Busemann function of
$l$
with
$b(o)=0$.
Then
$b\circ\phi=\la b$
for every homothety
$\phi\in\Ga_{\om,o}$,
where
$\la>0$
is the coefficient of
$\phi$.
By Corollary~\ref{cor:busparallel_foliation}, any Busemann
function of any Ptolemy line 
$l(x)$
through
$x\in X_\om$
is a Busemann function of a line
$l$
through
$o$.
Therefore, every 
$\phi\in\Ga_{\om,o}$
preserves any Busemann function
$b$
of
$l(x)$
with
$b(o)=0$
in the sense that
$\la^{-1}b\circ\phi=b$,
where
$\la>0$
is the coefficient of
$\phi$. 
Since
$\la^{-1}b\circ\phi$
is a Busemann function of the Ptolemy line
$\phi^{-1}(l(x))$, 
we see that this line is Busemann parallel to
$l(x)$.
Thus
$\phi$
preserves the foliation
$l(x)$, $x\in X_\om$
by Busemann parallel Ptolemy lines. 
\end{proof}

A construction of a homothety from the group
$\Ga_{\om,\om'}$
given in Proposition~\ref{pro:homothety_property} is not
uniquely determined because to obtain a homothety
with the same coefficient 
$\la$
one can take a composition of different pairs of 
s-inversions. Thus for given
$x$, $x'\in X_\om$
a shift
$\eta_{xx'}$
is not uniquely determined. We give a refined construction
of shifts with property
$\eta_{xx'}\to\id$
as
$x\to x'$
which will be used in the proof of Lemma~\ref{lem:zigzag_change_basepoint}
below.

\begin{lem}\label{lem:shift_identity} For
$x$, $x'\in X_\om$
there is a shift
$\eta_{xx'}:X_\om\to X_\om$
with 
$\eta_{xx'}(x)=x'$
such that
$\eta_{xx'}\to\id$
as 
$x\to x'$.
\end{lem}

\begin{proof} For
$\la_i\to\infty$
we denote by
$S=S_1(x)$, $S_i=S_{\la_i}(x)$
the metric spheres in
$X_\om$
centered at
$x$
of radius 1 and
$\la_i$
respectively. Similarly we put
$S'=S_1(x')$, $S_i'=S_{\la_i}(x')$.
Then
$\phi_i=\phi_{x,\om,S_i}\circ\phi_{x,\om,S}\in\Ga_{\om,x}$, 
$\psi_i=\phi_{x',\om,S_i'}\circ\phi_{x',\om,S'}\in\Ga_{\om,x'}$
are homotheties with the same coefficient
$\la_i^2$,
see the proof of Proposition~\ref{pro:homothety_property}.
The sequence of isometries
$\eta_i=\psi_i^{-1}\circ\phi_i:X_\om\to X_\om$
converges to a shift
$\eta:X_\om\to X_\om$
with
$\eta(x)=x'$.
We have
$$\eta_i=\phi_{x',\om,S'}\circ\phi_{x',\om,S_i'}\circ
   \phi_{x,\om,S_i}\circ\phi_{x,\om,S}$$
and
$\phi_{x',\om,S_i'}\circ\phi_{x,\om,S_i}\to\id$,
$\phi_{x',\om,S'}\circ\phi_{x,\om,S}\to\id$
as
$x\to x'$
because
$S_i\to S_i'$, $S\to S'$
in the Hausdorff metric, every s-inversion is uniquely determined by
its data according to our assumption, and s-inversions preserve the family
of metric spheres between data points, see sect.~\ref{subsect:spheres_between_points}. 
Thus
$\eta_i\to\id$
for every
$i$
as
$x\to x'$.
Moreover, the convergence of metric spheres around
$x$
to metric spheres around
$x'$
in the Hausdorff metric is uniform in radius as
$x\to x'$, $\Hd(S_r^d(x),S_r^d(x'))\le|xx'|$
for every
$r>0$.
Therefore, 
$\eta_i\to\id$
uniformly in
$i$
as
$x\to x'$.
This implies
$\eta\to\id$
as
$x\to x'$.
\end{proof}

\subsection{Two point homogeneity}
\label{subsect:enhancing_E}

By property (E) formulated in sect.~\ref{sect:introduction}
we know that the space
$X$
contains at least one Ptolemy circle.

\begin{pro}\label{pro:two_point_homogeneous} Any compact Ptolemy
space with the inversion property (I) is two point homogeneous,
that is, for each (ordered) pairs
$(x,y)$, $(x',y')$
of distinct points in
$X$
there is a M\"obius automorphism
$f:X\to X$
with
$f(x)=x'$, $f(y)=y'$. 
\end{pro}

\begin{proof} Applying an inversion, we can map 
$x$
to
$x'$.
Let
$y''$
be the image of
$y$
under the inversion. Then
$y''\neq x'$
by the assumption. We consider a metric of the M\"obius
structure with infinitely remote point 
$x'$.
By discussion above, there is a shift w.r.t. that metric which maps 
$y''$
to
$y'$.
The resulting composition gives a required M\"obius
automorphism.
\end{proof}

\begin{rem}\label{rem:2-transitive} In another terminology, under conditions of
Proposition~\ref{pro:two_point_homogeneous}, the action of
the group
$\aut X$ 
of M\"obius transformations
on
$X$
is {\em 2-transitive}, i.e.
$\aut X$
acts transitively on all pairs
$(x,x')\in X\times X$
with distinct entries, see \cite{Kr}.
\end{rem}

It immediately follows from Proposition~\ref{pro:two_point_homogeneous}
that the property (E) in any compact Ptolemy space with (I) 
is promoted to

\noindent
(E) Enhanced existence: through any two points in
$X$
there is a Ptolemy circle.

In what follow, we use this property under the name (E).

\subsection{Busemann functions and Busemann flatness}
\label{subsect:busemann_functions}

The proof of Theorem~\ref{thm:basic_ptolemy} is based on
study of Busemann functions on
$X_\om$.
In this section we assume that a compact Ptolemy space
$X$
possesses the properties (E) and (I).

\begin{lem}\label{lem:limit_circle} Assume that
$x_i\to x$
in
$X$,
and a point
$\om\in X$
distinct from
$x$
is fixed. Then any Ptolemy circle
$l\sub X$
through
$\om$, $x$
is the (pointwise) limit of a sequence of Ptolemy circles
$l_i\sub X$
through
$\om$, $x_i$.
\end{lem}

\begin{proof} In the space
$X_\om$
the circle
$l$
is a Ptolemy line (with infinitely remote point 
$\om$)
through
$x$.
Then the sequence
$l_i=\eta_i(l)$
of Ptolemy lines with
$x_i\in l_i$
converges to
$l$,
where
$\eta_i:X_\om\to X_\om$
is a shift with 
$\eta_i(x)=x_i$,
because the lines
$l_i$
are Busemann parallel to 
$l$
by Lemma~\ref{lem:pure_homothety}, and any sublimit
of the sequence
$\{l_i\}$
coincides with 
$l$
by Lemma~\ref{lem:unique_line}.
\end{proof}

We fix
$\om\in X$
and a metric 
$d$
of the M\"obius structure such that
$\om$
is the infinitely remote point. It immediately 
follows from the Ptolemy inequality that the distance function 
$d(z,\cdot)$
to a point 
$z\in X_\om$
is convex along any Ptolemy line in
$X_\om$, 
see \cite{FS2}. Under the homothety property (H) we prove that in fact 
$d(z,\cdot)$
is $C^1$-smooth.

\begin{pro}\label{pro:smooth_convex} 
A compact Ptolemy space with properties
(E) and (I)
is Busemann flat.
\end{pro}

\begin{proof} By Proposition~\ref{pro:busemann_flat} 
it suffices to prove that the distance function
$d_z=d(z,\cdot):X_\om\to\R$
is $C^1$-smooth along any Ptolemy line
$l\sub X_\om$
for any
$\om\in X$, $z\in X_\om\sm l$.

Assume that 
$d_z$
is not $C^1$-smooth at some point
$x\in l$.
We fix an arclength parameterization 
$c:\R\to X_\om$
of
$l$
such that
$x=c(0)$.
Since
$f=d_z\circ c$
is convex, it has the left and the right derivatives at every point.
By assumption, these derivatives are different at
$t=0$.
It follows that
$$\liminf_{t\to 0}\frac{f(t)-2f(0)+f(-t)}{t}>0.$$
By Proposition~\ref{pro:homothety_property}, property~(H)
holds for 
$X$.
Now, using (H), we find for every
$\la>0$
a homothety
$h_\la:X_\om\to X_\om$
with coefficient
$\la$
that preserves the point
$x$
and the Ptolemy line
$l$, $h_\la(x)=x$, $h_\la(l)=l$.
Then
$d(x,h_\la(z))\to\infty$
as
$\la\to\infty$
and thus
$\om_\la=h_\la(z)\to\om$.
By Lemma~\ref{lem:limit_circle}, there is a Ptolemy circle
$l_\la$
through
$x$, $\om_\la$
such that
$l_\la\to l$
as
$\la\to\infty$. 
We put
$\la=1/t$
and consider points
$x_t^+$, $x_t^-\in l_\la$
separated by
$x$
with
$d(x,x_t^\pm)=1$.
Then, W.L.G.,
$x_t^\pm\to c(\pm 1)$
as
$t\to 0$.
The points
$x_t^+$, $x$, $x_t^-$, $\om_\la$
lie on the Ptolemy circle
$l_\la$
(in this order), thus
$$2d(x,\om_\la)\ge d(x,\om_\la)d(x_t^+,x_t^-)=d(x_t^+,\om_\la)+d(x_t^-,\om_\la)$$
by the Ptolemy equality. On the other hand,
$f(0)/t=d(x,\om_\la)$
and
$f(\pm t)/t=d(c(\pm 1),\om_\la)$.
Thus
$|d(x_t^\pm,\om_\la)-f(\pm t)/t|\le d(x_t^\pm,c(\pm 1))\to 0$
as
$t\to 0$.
Therefore,
$(f(t)-2f(0)+f(-t))/t\to 0$
as
$t\to 0$
in contradiction with our assumption. 
\end{proof}

Using Proposition~\ref{pro:busemann_affine}, we immediately obtain 

\begin{cor}\label{cor:busemann_affine} Given two Ptolemy lines
$l$, $l'\in X_\om$,
the Busemann functions of
$l$
are affine functions on
$l'$.
\qed
\end{cor}

\begin{lem}\label{lem:geoconvex_horosphere} For any Busemann function
$b:X_\om\to\R$
of any Ptolemy line
$l\sub X_\om$,
every horosphere
$H_t=b^{-1}(t)$, $t\in\R$,
is geodesically convex, that is, any Ptolemy line
$l'\sub X_\om$
having two distinct points 
$z$, $z'$ 
in common with
$H_t$
is contained in
$H_t$, $l'\sub H_t$. 
\end{lem}

\begin{proof} We put
$b^+=b$
and assume W.L.G. that
$b^+(z)=0=b^-(z)$,
where the Busemann function 
$b^-$
of
$l$
is opposite to
$b^+$. 
Then
$b^++b^-\equiv 0$
because
$X$
is Busemann flat, see Proposition~\ref{pro:smooth_convex}. Thus
$H_0=(b^+)^{-1}(0)=(b^-)^{-1}(0)$
is a common horosphere for
$b^+$, $b^-$.
Since horoballs, i.e. sublevel sets of Busemann functions, 
are convex, the geodesic segment
$zz'\sub l$
lies in
$H_0$.
By Corollary~\ref{cor:busemann_affine}, the function
$b$
is affine along
$l'$,
that is, 
$b\circ c(t)=\al t+\be$
for any arclength parameterization 
$c:\R\to l'$
of
$l'$
and some
$\al$, $\be\in\R$, $|\al|\le 1$.
We choose
$c$
so that
$c(0)=z$, $c(|zz'|)=z'$.
Then
$\be=0$
by the assumption
$b(z)=0$,
and
$0=b(z')=b\circ c(|zz'|)=\al|zz'|$.
Hence
$\al=0$
and
$b|l'\equiv 0$.
This shows that
$l'\sub H_0$.
\end{proof}

\section{Slope of Ptolemy lines and circles}
\label{sect:slope}

\subsection{Definition and properties}
\label{subsect:slope}

By Corollary~\ref{cor:busemann_affine}, a Busemann
function associated with a Ptolemy line is affine along
any other Ptolemy line. We introduce a quantity which measures
a mutual position of Ptolemy lines in the space.

Let
$l$, $l'\sub X_\om$
be oriented Ptolemy lines.
We define the {\em slope} of
$l'$
w.r.t.
$l$
as the coefficient of a Busemann function
$b$
associated with
$l$
when restricted to
$l'$, $\slope(l';l)=\al$
if and only if
$b\circ c'(t)=\al t+\be$
for some
$\be\in\R$
and all
$t\in\R$,
where
$c':\R\to X_\om$
is a unit speed parameterization of
$l'$
compatible with its orientation. The quantity
$\slope(l';l)\in[-1,1]$
is well defined, i.e. it depends of the choice neither the Busemann
function
$b$ 
nor the parameterization
$c'$
(we assume that
$b$
is defined via a parameterization of
$l$
compatible with its orientation).
Note that the slope changes the sign when the orientation of
$l$
or
$l'$
is changed,
$$\slope(-l';l)=-\slope(l';l)=\slope(l';-l).$$
The first equality is obvious, while the second one holds because
$X$
is Busemann flat by Proposition~\ref{pro:smooth_convex}. 

By definition, we have
$\slope(l;l)=-1$
for any oriented Ptolemy line
$l\sub X_\om$.
More generally, let
$l$, $l'\sub X_\om$
be Busemann parallel Ptolemy lines. If an orientation of
$l$
is fixed, then a {\em compatible} orientation of
$l'$
is well defined. Indeed, we take a Busemann function 
$b$
of
$l$
such that
$b\to -\infty$
along
$l$
in the chosen direction. Since
$b$
is also a Busemann function of
$l'$,
the respective direction of
$l'$
such that
$b\to-\infty$
along
$l'$
is well defined, and it is independent of the choice of
$b$.

Now, if orientations of Busemann parallel
$l$, $l'$
are compatible, then 
$\slope(l';l)=-1=\slope(l;l')$.

\begin{lem}\label{lem:paraline_busemann} Let
$l$, $l'\sub X_\om$
be Busemann parallel Ptolemy lines with compatible orientations. Then 
for any oriented Ptolemy line
$l''\sub X_\om$
we have
$\slope(l;l'')=\slope(l';l'')$.
\end{lem}

\begin{proof}
Let
$b:X_\om\to\R$
be a Busemann function associated with
$l''$.
There are
unit speed parameterizations
$c:\R\to l$, $c':\R\to l'$
compatible with the orientations of
$l$, $l'$
such that
$b\circ c(0)=b\circ c'(0)=:\be$.
Then, since
$b$
is affine along Ptolemy lines, 
$b\circ c(t)=\al t+\be$, $b\circ c'(t)=\al' t+\be$
for some
$|\al|$, $|\al'|\le 1$
and all
$t\in\R$.
We show that
$\al=\al'$.

Since the orientations of
$l$, $l'$
are compatible, we have
$|c(t)c'(t)|=o(1)|t|$
as
$|t|\to\infty$
by Lemma~\ref{lem:busparallel_sublinear}. Let
$c'':\R\to l''$
be a unit speed parameterization such that
$b(x)=\lim_{s\to\infty}|c''(s)x|-s$, $x\in X_\om$.
Since
$||c''(s)c(t)|-|c''(s)c'(t)||\le|c(t)c'(t)|=o(1)t$
as
$t\to\infty$,
we have
$$|(\al-\al')t|=|b\circ c(t)-b\circ c'(t)|=o(1)t$$
and hence
$\al=\al'$.
\end{proof}

Proposition~\ref{pro:busemann_affine} combined with duality gives
rise to a first variation formula to describe which we use
the following agreement. Let
$\si$, $\si'\sub X$
be Ptolemy circles meeting each other at two distinct points
$\om$
and
$\om'$, $\si\cap\si'=\{\om,\om'\}$,
which decompose the circles into (closed) arcs
$\si=\si_+\cup\si_-$
and
$\si'=\si_+'\cup\si_-'$.
The choice of
$\om$
as an infinitely remote point automatically introduces the orientation
of 
$\si$
as well as of
$\si'$
such that 
$\om'$
is the initial point of the arcs
$\si_+$, $\si_+'$,
while
$\om$
the final point of
$\si_+$, $\si_+'$,
and the similar agreement holds true for the choice of
$\om'$
as an infinitely remote point. Then the 
$\slope(\si_\om';\si_\om)$
of the Ptolemy line
$\si_\om'\sub X_\om$
w.r.t. the Ptolemy line
$\si_\om\sub X_\om$
is well defined.

\begin{lem}\label{lem:first_variation} Let
$\si$, $\si'\sub X$
be Ptolemy circles meeting each other at two distinct points
$\om$
and
$\om'$, $\si\cap\si'=\{\om,\om'\}$,
which decompose the circles into (closed) arcs
$\si=\si_+\cup\si_-$
and
$\si'=\si_+'\cup\si_-'$.
Let
$c:\R\to X_{\om'}$, $c':\R\to X_\om$
be the unit speed parameterizations of the oriented Ptolemy lines
$\si_{\om'}\sub X_{\om'}$, $\si_\om'\sub X_\om$
respectively compatible with the orientations such that
$c(0)=\om$, $c'(0)=\om'$.
Then
\begin{equation}\label{eq:first_variation}
\frac{d}{dt}d'(c'(s),c(t))|_{t=0}=\al\sign s
\end{equation}
for all
$s\neq 0$,
where
$d'$
is the metric of
$X_{\om'}$,
and 
$\al=\slope(\si_\om';\si_\om)$.
\end{lem}

\begin{rem}\label{rem:first_variation} We emphasize that
(\ref{eq:first_variation}) is a typical duality equality
where the left hand side is computed in the space
$X_{\om'}$,
while the right hand side is computed in the inverted space
$X_\om$.
\end{rem}

\begin{proof} By Corollary~\ref{cor:busemann_affine}, 
the Busemann function
$b$
of
$\si_\om$
with
$b(\om')=0$
is affine along the Ptolemy line
$\si_\om'\sub X_\om$.
Thus
$b\circ c'(s)=\al s$
for 
$\al=\slope(\si_\om';\si_\om)$
and all
$s\in\R$
because
$b\circ c'(0)=b(\om')=0$.
Since
$X$
is Busemann flat by Proposition~\ref{pro:smooth_convex},
Equation~(\ref{eq:smooth_duality}) applied to
$b^+=b$
gives
$$\al s=b\circ c'(s)=\frac{d}{dt}\ln d'(c'(s),c(t))|_{t=0}
 =\frac{1}{d'(c'(s),\om)}\frac{d}{dt}d'(c'(s),c(t))|_{t=0}$$
for all
$s\neq 0$. 
Using that
$d'(c'(s),\om)=1/|s|$,
we obtain the required equality.
\end{proof}

In the situation with two Ptolemy circles intersecting at two
distinct points as in Lemma~\ref{lem:first_variation}, we have
four a priori different slopes. The duality and existence of 
s-inversions allows to reduce this number to one.

\begin{lem}\label{lem:opposite_slopes} Let
$\si$, $\si'\sub X$
be Ptolemy circles meeting each other at two distinct points
$\om$
and
$\om'$, $\si\cap\si'=\{\om,\om'\}$,
which decompose the circles into (closed) arcs
$\si=\si_+\cup\si_-$
and
$\si'=\si_+'\cup\si_-'$.
Then
$$\slope(\si_\om';\si_\om)=\slope(\si_{\om'};\si_{\om'}').$$
\end{lem}

\begin{proof} Denote
$\al=\slope(\si_\om';\si_\om)$
and
$\al'=\slope(\si_{\om'};\si_{\om'}')$. 
As in Lemma~\ref{lem:first_variation} let
$c:\R\to X_{\om'}$, $c':\R\to X_\om$
be the unit speed parameterizations of the oriented Ptolemy lines
$\si_{\om'}\sub X_{\om'}$, $\si_\om'\sub X_\om$
respectively compatible with the orientations such that
$c(0)=\om$, $c'(0)=\om'$.
Let
$b':X_{\om'}\to\R$
be the Busemann function associated with 
$\si_{\om'}'$
such that
$b'(\om)=0$
(according to our agreement,
$b'$
is computed via a parameterization of
$\si_{\om'}'$
which is opposite in orientation to that of
the parameterization 
$c'$). Then
$b'\circ c(t)=\al' t$
for all 
$t\in\R$
by the definition of
$\al'=\slope(\si_{\om'};\si_{\om'}')$.
Thus
$$\al' t=b'\circ c(t)\le d'(c'(s),c(t))-d'(c'(s),\om)$$
for all
$s>0$
and all (sufficiently small)
$t\in\R$,
where
$d'$
is the metric of
$X_{\om'}$.
The last inequality holds because the right hand side 
decreases monotonically to
$b'\circ c(t)$
for every fixed
$t$
as
$s\to 0$.
Applying Equality~(\ref{eq:first_variation}), we obtain
$\al'\le\al$.
Interchanging
$\om$
with
$\om'$
and
$\si$
with
$\si'$,
we have
$\al\le\al'$
by the same argument. Hence, the claim.
\end{proof}

Using Lemma~\ref{lem:opposite_slopes}, the first variation
formula~(\ref{eq:first_variation}) can be rewritten as follows
\begin{equation}\label{eq:re_first_variation}
\frac{d}{dt}d'(c'(s),c(t))|_{t=0}=\al'\sign s
\end{equation}
for all
$s\neq 0$,
where
$\al'=\slope(\si_{\om'};\si_{\om'}')$.
Now, the both sides of (\ref{eq:re_first_variation})
are computed in the same space
$X_{\om'}$.

Lemma~\ref{lem:opposite_slopes} implies the symmetry of the slope
w.r.t. the arguments. 

\begin{lem}\label{lem:slope_symmetry} For any oriented Ptolemy lines
$l$, $l'\sub X_\om$
we have
$\slope(l';l)=\slope(l;l')$.
\end{lem}

\begin{proof} We assume W.L.G. that
$l\cap l'=\om'$
and represent
$l=\si_\om$, $l'=\si_\om'$
for Ptolemy circles
$\si=l\cup\om$, $\si'=l'\cup\om$.
Then
$\si\cap\si'=\{\om,\om'\}$.
Let
$S\sub X$
be a sphere between
$\om$, $\om'$, $\phi=\phi_{\om,\om',S}:X\to X$
the s-inversion w.r.t.
$\om$, $\om'$, $S$.
Then
$\phi$
preserves any Ptolemy circle though
$\om$, $\om'$
and its orientations. In particular,
$\phi(\si_\om)=\si_{\om'}$
and
$\phi(\si_\om')=\si_{\om'}'$.
We assume that
$S=S_1^d(\om')$,
where
$d\in\cM$
is the metric of
$X_\om$.
Then the metric 
$d'=\phi^\ast d$
is the m-inversion of
$d$,
and vice versa, see Lemma~\ref{lem:sinversion_minversion}.
It follows that the map
$\phi:(X_\om,d)\to (X_{\om'},d')$
is an isometry.

Now, we have
$$\slope(\si_{\om'}';\si_{\om'})=
  \slope(\phi(\si_\om');\phi(\si_\om))=\slope(\si_\om';\si_\om)
  =\slope(l';l).$$
Using Lemma~\ref{lem:opposite_slopes} we obtain
$\slope(l;l')=\slope(\si_\om;\si_\om')
=\slope(\si_{\om'}';\si_{\om'})=\slope(l';l)$. 
\end{proof}

From now on, we use notation
$\slope(l,l')$
for the slope instead of
$\slope(l;l')$.
We say that Ptolemy lines
$l$, $l'\sub X_\om$
are {\em orthogonal} if
$\slope(l',l)=0$.
By Lemma~\ref{lem:slope_symmetry} this is a symmetric
relation. For orthogonal lines we also use notation
$l\bot l'$.

\subsection{Tangent lines and slope of Ptolemy circles}
\label{subsect:tangent_lines}

A Ptolemy line 
$l\sub X_\om$
is {\em tangent} to a Ptolemy circle 
$\si\sub X_\om$
at a point
$x\in\si$
if for every
$y\in\si$
sufficiently close to
$x$
we have
$\dist(y,l)=o(|xy|)$.

\begin{pro}\label{pro:tangent_rcircle} Every Ptolemy circle 
$\si\sub X_\om$
possesses a unique tangent Ptolemy line
$l$ 
at every point 
$x\in\si$. 
\end{pro}

\begin{proof} Let
$d$
be the metric of
$X_\om$, $d'$
the metric inversion of
$d$
w.r.t.
$x$, $d'(y,z)=\frac{d(y,z)}{d(y,x)d(z,x)}$.
Then
$x$
is infinitely remote for
$d'$,
and
$\si\sm x\sub X_x$
is a Ptolemy line w.r.t. the metric
$d'$
on
$X_x$.
By Corollary~\ref{cor:busparallel_foliation}, there is
a unique Ptolemy line 
$\wt l\sub X_x$
through
$\om$
which is Busemann parallel to the line
$\si\sm x$.
Then
$l=(\wt l\cup x)\sm\om\sub X_\om$
is a Ptolemy line through
$x$.
We show that
$l$
is tangent to
$\si$
at
$x$.
 
We fix on
$\si\sm x$
and
$\wt l$
compatible orientations, see sect.~\ref{subsect:slope},
and choose
$y\in\si$, $y'\in l$
with sufficiently small positive
$t=d(x,y)=d(x,y')$
according to the orientations. Recall that
$d$
is also the metric inversion of
$d'$
w.r.t.
$\om$,
and that
$d(x,z)=1/d'(\om,z)$
for every
$z\in X\sm\{x,\om\}$.
Then
$$d(y,y')=\frac{d'(y,y')}{d'(\om,y)d'(\om,y')}
         =t\frac{d'(y,y')}{1/t}.$$
By Lemma~\ref{lem:busparallel_sublinear},
$\frac{d'(y,y')}{1/t}\to 0$
as
$t\to 0$,
hence
$d(y,y')=o(t)$,
and thus
$l$
is tangent to
$\si$
at
$x$.

If 
$l'\sub X_\om$
is another tangent line to
$\si$
at
$x$,
then reversing the argument above we observe that
the Ptolemy lines
$\wt l$, $\wt l'=(l'\cup\om)\sm x\sub X_x$
through 
$\om$
diverge sublinearly and thus they are Busemann parallel
again by Lemma~\ref{lem:busparallel_sublinear}. It follows
that
$\wt l=\wt l'$
and
$l=l'$. 
\end{proof}

Now, we reformulate Corollary~\ref{cor:busparallel_foliation}
as follows.

\begin{cor}\label{cor:tangent_unique} Given a Ptolemy line 
$l\sub X_\om$
and a point 
$x\in l$,
for any other point
$y\in X_\om$
there exists a unique Ptolemy circle
$\si\sub X_\om$
through
$y$
tangent to
$l$
at 
$x$.
In particular, if
$y\in l$,
then
$\si=l$. 
\end{cor}

\begin{proof} Consider a metric of the M\"obius structure on
$X$
with the infinitely remote point 
$x$
and apply Corollary~\ref{cor:busparallel_foliation}. 
\end{proof}

If oriented Ptolemy circles
$\si$, $\si'\sub X$
are disjoint, then their slope is not determined. Assume now that
$\om\in\si\cap\si'$.
Then 
$\slope_\om(\si,\si)\in[-1,1]$
is defined as the slope of oriented Ptolemy lines
$\si\sm\om$, $\si'\sm\om\sub X_\om$.
This is well defined and symmetric,
$\slope_\om(\si,\si')=\slope_\om(\si',\si)$,
by Lemma~\ref{lem:slope_symmetry}. In the case
$\slope_\om(\si,\si')=\pm 1$,
the Ptolemy circles are tangent to each other at
$\om$,
having compatible ($-1$) or opposite ($+1$) orientations.
This means that in any space 
$X_{\om'}$
with 
$\om'\neq\om$,
the Ptolemy circles
$\si\sm\om'$, $\si'\sm\om'\sub X_{\om}$
have a common tangent line at
$\om$.

More generally, let
$l$, $l'\sub X_{\om'}$
be the tangent lines to
$\si$, $\si'$
respectively at
$\om$, 
oriented according to the orientations of
$\si$, $\si'$.
Then
$\slope_\om(\si,\si')=\slope(l,l')$
because in the space 
$X_\om$
the Ptolemy line
$\si\sm\om$
is Busemann parallel to
$l\sm\om$,
and
$\si'\sm\om$
is Busemann parallel to
$l'\sm\om$,
and we can apply Lemma~\ref{lem:paraline_busemann}.

Furthermore, if
$\si$, $\si'$
have two distinct common points,
$\om$, $\om'\in\si\cap\si'$,
then
$\slope_\om(\si,\si')=\slope_{\om'}(\si,\si')$.
This follows from Lemma~\ref{lem:opposite_slopes}
and Lemma~\ref{lem:slope_symmetry}. Thus we use notation
$\slope(\si,\si')$
for the slope of intersecting oriented Ptolemy circles
$\si$, $\si'$.

\section{Fibration 
$\pi_\om:X_\om\to B_\om$}
\label{sect:fibration}

\subsection{Definition and properties}
\label{subsect:def_fibration}

As usual we fix
$\om\in X$
and consider a metric 
$d$
of the M\"obius structure with 
infinitely remote point
$\om$.

Given
$x\in X_\om$,
we define
$$F_x=\bigcap_{l\ni x}H_l,$$
where the intersection is taken over all the Ptolemy lines
$l\sub X_\om$
through
$x$,
$H_l$
is the horosphere through 
$x$
of a Busemann function associated with
$l$
(since
$X$
is Busemann flat,
$H_l$
is independent of choice of a Busemann function).

\begin{lem}\label{lem:fiber_def} For any
$y\in F_x$
we have
$F_y=F_x$.
\end{lem}

\begin{proof}
By Corollary~\ref{cor:busparallel_foliation}, for every Ptolemy line
$l$
through
$x$
there is a unique Ptolemy line
$l'$
through
$y$
such that
$l$, $l'$
are Busemann parallel. Let
$b$
be a Busemann function of
$l$
such that
$b(x)=0$.
Then
$b(y)=0$
because
$y$
lies in the horosphere through
$x$
of
$b$.
Hence
$H_l=H_{l'}$
because
$b$
is a Busemann function also of
$l'$,
and thus
$F_y=F_x$. 
\end{proof}

By Lemma~\ref{lem:fiber_def}, the sets
$F_x$, $F_{x'}$
coincide or are disjoint for any
$x$, $x'\in X_\om$.
We let
$B_\om=\set{F_x}{$x\in X_\om$}$
and define
$\pi_\om:X_\om\to B_\om$
by
$\pi_\om(x)=F_x$.
Therefore, the fibers
$F_b=\pi_\om^{-1}(b)$, $b\in B_\om$,
form a partition of
$X_\om$, $B_\om$
is the factor-space of this partition, and
$\pi_\om$
is the respective factor-map. A fiber
$F$
of
$\pi_\om$
is also called a $\K$-{\em line}. 

\begin{lem}\label{lem:induced_base_map}
For any
$\om$, $\om'\in X$,
any M\"obius automorphism
$\phi:X\to X$
with
$\phi(\om)=\om'$
induces a bijection
$\ov\phi:B_\om\to B_{\om'}$
such that
$\pi_{\om'}\circ\phi=\ov\phi\circ\pi_\om$.
\end{lem}

\begin{proof}
It follows from Lemma~\ref{lem:homothety_infinite} that
for any metrics
$d$, $d'$
of the M\"obius structure with infinitely remote points
$\om$, $\om'$
respectively, the map
$\phi:(X_\om,d)\to(X_{\om'},d')$
is a homothety. Thus
$\phi$
maps any Ptolemy line 
$l\sub X_\om$
to the Ptolemy line
$l'=\phi(l)\sub X_{\om'}$,
and
$b\circ\phi^{-1}$
is proportional to a Busemann function of
$l'$
for any Busemann function
$b$
of
$l$.
It follows that
$\phi$
induces a bijection
$\ov\phi:B_\om\to B_\om$
such that
$\ov\phi\circ\pi_\om=\pi_\om\circ\phi$. 
\end{proof}

\begin{proof}[Proof of property ($\K$): uniqueness]
Let
$F\sub X_\om$
be a $\K$-line, and let
$x\in X\sm F$.
We show that there is at most one Ptolemy line in
$X_\om$
through
$x$
that meets
$F$.
Assume that there are Ptolemy lines
$l$, $l'\in X_\om$
through
$x$
that intersect
$F$.
Let
$c:\R\to l$, $c':\R\to l'$
be unit speed parameterizations such that
$c(0)=x=c'(0)$,
and
$c(s)\in F$, $c'(s')\in F$
for some
$s$, $s'>0$.
For the Busemann function
$b:X_\om\to\R$, $b(y)=\lim_{t\to-\infty}|yc(t)|-|t|$,
of
$l$
we have
$b(c(s))=s=b(c'(s'))$
because
$c(s)$, $c'(s')\in F$
and 
$F$
is a fiber of the fibration
$\pi_\om:X_\om\to B_\om$.
The function
$t\mapsto |c'(s')c(t)|-(|t|+s)$
is nonincreasing and it converges to
$b(c'(s'))-s=0$
as
$t\to-\infty$,
thus
$s'=|c'(s')c(0)|\ge s=|c(s)x|$.
Interchanging 
$l$ 
and
$l'$
we obtain 
$s\ge s'$
by the same reason. Hence
$s=s'$.
Since the Busemann function
$b$
is affine along
$l'$
by Corollary~\ref{cor:busemann_affine}
and it takes the equal values
$0=b\circ c(0)$
and
$b\circ c(s)=s=b\circ c'(s)$
along
$l$, $l'$
at two different parameter points, we have
$b\circ c(t)=b\circ c'(t)$
for every
$t\in\R$.
By Lemma~\ref{lem:unique_line},
$l=l'$.
\end{proof}

\begin{lem}\label{lem:rfoliation_semik} Given a Ptolemy line
$l\sub X_\om$
and distinct points
$x$, $y\in l$,
a Ptolemy line
$l'\sub X_\om$
through
$x'\in F_x$
meets the fiber
$F_y$
if and only if it is Busemann parallel to 
$l$.
In this case
$|x'y'|=|xy|$
for 
$y'=l'\cap F_y$. 
\end{lem}

\begin{proof}
By Corollary~\ref{cor:busparallel_foliation}, 
there is a unique Ptolemy line
$\wt l$
through
$x'$
which is Busemann parallel to
$l$.
Consider compatible unit speed parameterizations
$c:\R\to l$, $\wt c:\R\to\wt l$
such that
$c(0)=x$, $\wt c(0)=x'$.
Then
$y=c(t)$
for some
$t\in\R$.
We show that
$\wt c(t)\in F_y$.
Let
$l''$
be a Ptolemy line through
$c(t)$, $b''$
a Busemann function of
$l''$
with
$b''(c(t))=0$.
We show that
$\wt c(t)$
lies in the zero level set of
$b''$, $b''(\wt c(t))=0$.

By Corollary~\ref{cor:busparallel_foliation}, there is
a Ptolemy line through
$c(0)$
for which
$b''$
is a Busemann function. Then the $\K$-line
$F_x$
lies in a level set of
$b''$,
in particular,
$b''(c(0))=b''(\wt c(0))=:\be$.
By Lemma~\ref{lem:paraline_busemann} we have
$b''\circ c(s)=\al s+\be=b''\circ\wt c(s)$
for all
$s\in\R$.
In particular,
$b''(\wt c(t))=b''(c(t))=0$,
hence
$\wt c(t)\in F_y$
and
$\wt l$
meets
$F_y$.
By the uniqueness part of property~($\K$),
there is at most one Ptolemy line through
$x'$
that hits
$F_y$.
Hence, if
$l'$
meets
$F_y$,
then
$l'=\wt l$
is Busemann parallel to
$l$.
Moreover, the argument above also shows that the $\K$-lines
$F_x$, $F_y$
are equidistant,
$|x'y'|=|xy|$.
\end{proof}

\subsection{Zigzag curves}
\label{subsect:zigzag}

We fix
$\om\in X$
and consider a metric on
$X_\om$
with infinitely remote point
$\om$.
Let
$l\sub X_\om$
be an oriented Ptolemy line. By Corollary~\ref{cor:busparallel_foliation},
there is a foliation
$l(x)$, $x\in X_\om$
of the space
$X_\om$
by Ptolemy lines, which are Busemann parallel to
$l$.
Moreover, every member
$l(x)$
of the foliation has a well defined orientation compatible
with that of
$l$,
see sect.~\ref{subsect:slope}.

\begin{lem}\label{lem:parallel_couple} Let
$l_1$, $l_2\sub X_\om$
be oriented Ptolemy lines which induce respective
foliations of
$X_\om$.
We start moving from
$x\in X_\om$
along
$l_1(x)$
by some distance
$s_1\ge 0$
up to a point
$y$,
and then switch to
$l_2(y)$
and move along it by some distance
$s_2\ge 0$
up to a point
$z$.
Next, we move from
$x'\in X_\om$
along
$l_2(x')$
by the distance
$s_2$
up to a point
$y'$,
and then switch to
$l_1(y')$
and move along it by the distance
$s_1$
up to a point
$z'$,
where we always move in the directions prescribed by 
the orientations. If
$x$, $x'$
lie in a $\K$-line
$F\sub X_\om$,
then
$z$, $z'$
also lie in one and the same $\K$-line
$F'\sub X_\om$.
\end{lem}

\begin{proof} Let
$c_1$, $c_2:\R\to X_\om$
be the unit speed parameterizations of
$l_1(x)$, $l_2(x')$
respectively compatible with the orientations such that
$c_1(0)=x$, $c_2(0)=x'$.
We also consider the unit speed parameterizations
$c_1'$, $c_2':\R\to X_\om$
of
$l_1(y')$, $l_2(y)$
respectively compatible with the orientations such that
$c_1'(0)=y'$, $c_2'(0)=y$.
Then
$c_1(s_1)=y=c_2'(0)$, $c_2(s_2)=y'=c_1'(0)$
and
$c_2'(s_2)=z$, $c_1'(s_1)=z'$.

Let
$b$
be a Busemann function of a Ptolemy line
$l\sub X_\om$
which vanishes along
$F$,
in particular,
$b(x)=0=b(x')$.
By Corollary~\ref{cor:busemann_affine},
$b$
is affine along any Ptolemy line in
$X_\om$,
in particular,
$b\circ c_1(t)=\al_1t$, $b\circ c_2(t)=\al_2t$
for some
$\al_i$
which by Lemma~\ref{lem:paraline_busemann}
only depends on
$l_i$, $i=1,2$,
and for all
$t\in\R$. 
Thus we have
$b(z')=b\circ c_1'(s_1)=\al_1s_1+\al_2s_2$
and similarly
$b(z)=b\circ c_2'(s_2)=\al_2s_2+\al_1s_1$.
Hence any Busemann function on
$X_\om$
takes the same value at the points
$z$
and 
$z'$,
i.e. these points lie in a common $\K$-line
$F'$.
\end{proof}

Given a base point
$o\in X_\om$,
a finite ordered collection
$\cL=\{l_1,\dots,l_k\}$
of oriented Ptolemy lines in
$X_\om$,
and a collection
$S=\{s_1,\dots,s_k\}$
of nonnegative numbers with
$s_1+\dots+s_k>0$,
we construct a sequence
$\ga_p=\ga_p(o,\cL,S)\sub X_\om$, $p\ge 1$,
of piecewise geodesic curves through
$o$
as follows. Recall that we have 
$k$ 
foliations of
$X_\om$
by oriented Ptolemy lines 
$l_1(x),\dots,l_k(x)$, $x\in X_\om$,
which are Busemann parallel with compatible orientations to
$l_1,\dots,l_k$
respectively.

The curve
$\ga_p$
starts at
$o=v_p^0$
for every
$p\ge 1$.
We move along
$l_1(o)$
by the distance
$s_1/2^{p-1}$
up to the point
$v_p^1\in l_1(v_p^0)$,
then switch to the line
$l_2(v_p^1)$
and move along it by the distance
$s_2/2^{p-1}$
up to the point
$v_p^2$
etc. On the
$i$th 
step, for
$1\le i\le k$, 
we move along the line
$l_i(v_p^{i-1})$
by the distance
$s_i/2^{p-1}$
in the direction prescribed by the orientation of the line
up to the point
$v_p^i\in l_i(v_p^{i-1})$.
Starting with the point
$v_p^k$
we then repeat this procedure only taking the subindices for
$l_i$, $s_i$
modulo
$k$
for all integer
$i\ge k+1$.

This produces the sequence
$v_p^n$
of vertices of
$\ga_p$
for all
$n\ge 0$.
For integer
$n<0$
the vertices
$v_p^n$
are determined in the same way with all the orientations
of the lines
$l_1,\dots,l_k$
reversed, with the starting line
$l_k(o)$,
and with the ordered collections
$\ov{\cL}=\{l_k,\dots,l_1\}$
of lines, and
$\ov S=\{s_k,\dots,s_1\}$
of numbers.

Every curve
$\ga_p$
receives the arclength parameterization, for which
we use the same notation
$\ga_p:\R\to X_\om$,
with
$\ga_p(0)=o$.
Then for every
$m\in\Z$, $1\le i\le k$,
we have
$\ga_p(t_p^n)=v_p^n$
is a vertex of
$\ga_p$,
where
$n=k(m-1)+i$, 
$t_p^n=[(s_1+\dots+s_i)m+(s_{i+1}+\dots+s_k)(m-1)]/2^{p-1}$
(the sum
$(s_{i+1}+\dots+s_k)$
is assumed to be zero for
$i=k$).

It follows from Lemma~\ref{lem:parallel_couple}
by induction that for every 
$n=km\in\Z$,
the vertices
$v_p^n=\ga_p(t_p^n)$
of
$\ga_p$
and
$v_{p+1}^{2n}=\ga_{p+1}(t_{p+1}^{2n})$
of
$\ga_{p+1}$
lie in a common $\K$-line in
$X_\om$
for every
$p\ge 1$.
From this one easily concludes that the sequence of the projected curves
$\pi_\om(\ga_p)\sub B_\om$
converges (pointwise in the induced topology). At this stage,
we do not have tools to prove that the sequence
$\ga_p$
itself converges in
$X_\om$.
However, we need a limiting object of
$\ga_p$.
Thus, for instance, we fix a nonprincipal ultra-filter on
$\Z$
and say that
$\ga=\lim\ga_p$
w.r.t. that ultra-filter. By this we mean that
$\ga(t)=\lim\ga_p(t)$
for every 
$t\in\R$.
The curve
$\ga=\ga(o,\cL,S)$
is called a {\em zigzag} curve, and it is obtained together with 
the limiting parameterization
$\ga:\R\to X_\om$, $\ga(0)=o$,
which in general is not an arclength parameterization.

\begin{lem}\label{lem:busemann_affine_zigzag} Every Busemann function
$b:X_\om\to\R$
is affine along any zigzag curve
$\ga$,
that is, the function
$b\circ\ga:\R\to\R$
is affine. Moreover, if
$\ga=\ga(o,\cL,S)$
for a base point
$o\in X_\om$,
some ordered collection
$\cL=\{l_1,\dots,l_k\}$
of oriented Ptolemy lines in
$X_\om$, 
and a collection
$S=\{s_1,\dots,s_k\}$
of nonnegative numbers with
$s_1+\dots+s_k>0$,
and
$b(o)=0$,
then
$b\circ\ga(t)=\be t$
for all
$t\in\R$,
where
$\be=\sum_i\al_is_i/\sum_is_i$,
$\al_i=\slope(l_i,l)$, $i=1,\dots,k$,
and
$l\sub X_\om$
is the oriented Ptolemy line for which the function
$b$
is associated.
\end{lem}

\begin{proof} We assume that
$\ga=\lim\ga_p$.
Since
$\ga_p$
is piecewise geodesic for every
$p\ge 1$,
the function
$b\circ\ga_p:\R\to\R$
is piecewise affine. Recall that the points
$v_p^n=\ga_p(t_p^n)$
are vertices of
$\ga_p$,
where
$t_p^n=[(s_1+\dots+s_i)m+(s_{i+1}+\dots+s_k)(m-1)]/2^{p-1}$
for 
$n=k(m-1)+i\in\Z$.
Thus we have by induction
$$b\circ\ga_p(t_p^n)=
  [(\al_1s_1+\dots+\al_is_i)m
  +(\al_{i+1}s_{i+1}+\dots+\al_ks_k)(m-1)]/2^{p-1}$$
for
$n=k(m-1)+i\in\Z$.
Hence,
$b\circ\ga_p(t_p^n)=\be t_p^n+o(1)$
as
$p\to\infty$.
Since the step
$t_p^{n+1}-t_p^n\le\max_is_i/2^{p-1}\to 0$
as
$p\to\infty$,
we conclude that
$b\circ\ga_p\to b\circ\ga$
pointwise as
$p\to\infty$,
and
$b\circ\ga(t)=\be t$
for all
$t\in\R$.
\end{proof}

Lemma~\ref{lem:busemann_affine_zigzag} gives a strong evidence
in support of the expectation that a zigzag curve under natural
assumptions actually is a Ptolemy line. However we need additional 
arguments for the proof of this.

For
$\om$, $o\in X$,
the group
$\Ga_{\om,o}$
consists of homotheties
$\phi:X_\om\to X_\om$
with
$\phi(o)=o$
such that
$\phi(l)=l$
for every Ptolemy line
$l\sub X_\om$
through
$o$
preserving an orientation of
$l$,
and moreover by property (H),  
$\Ga_{\om,o}$
acts transitively on the open rays of
$l$
with the vertex
$o$,
see Proposition~\ref{pro:homothety_property}.

\begin{lem}\label{lem:zigzag_preserved} For any base point 
$o\in X_\om$,
the homothety
$\phi\in\Ga_{\om,o}$
with the coefficient
$\la=1/2$
leaves invariant a zigzag curve
$\ga=\ga(o,\cL,S)$
for any ordered collection of oriented Ptolemy lines
$\cL=\{l_1,\dots,l_k\}$
in
$X_\om$,
and any collection
$S=\{s_1,\dots,s_k\}$
of nonnegative numbers with
$s_1+\cdots+s_k>0$.
\end{lem}

\begin{proof} Let 
$\ga_p=\ga_p(o,\cL,S)$, $p\ge 1$,
be the sequence of piecewise geodesic curves in
$X_\om$,
so
$\ga=\lim\ga_p$.
For
$p\ge 1$,
we let
$v_p^{km}$, $m\in\Z$,
be the sequence of vertices of
$\ga_p$, $\ov v_p^{km}=\pi_\om(v_p^{km})$ 
the sequence of respective fibers of the fibration
$\pi_\om:X_\om\to B_\om$.
Recall that the sequences
$\set{\ov v_p^{km}}{$m\in\Z$}$, $p\ge 1$,
approximate the projection
$\pi_\om(\ga)$
of 
$\ga$,
that is,
$\pi_\om(\ga)$
coincides with the closure of the union
$\cup_p\set{\ov v_p^{km}}{$m\in\Z$}$.

We have
$\phi(\ga_p)=\ga_{p+1}$
and
$\phi(v_p^{km})=v_{p+1}^{km}$
by the construction of
$\ga_p$
and Lemma~\ref{lem:pure_homothety}. For every dyadic number
$q=m/2^r$, $m\in\Z$, $r\ge 0$,
the sequence
$v(q)=\set{v_p^{q_p}}{$p\ge r+1$}$,
where
$q_p=2^{p-(r+1)}\cdot km$, 
lies in a common fiber
$F=F(q)$
of
$\pi_\om$.
Thus 
$\phi$
maps this sequence into the sequence 
$v'(q)=\set{v_{p+1}^{q_p'}}{$p\ge r+1$}\sub\phi(F)=F(q/2)$,
where
$q_p'=2^{p-r}(km/2)$,
shrinking the mutual distances by the factor
$1/2$.
Hence for the limit point
$x=\lim v(q)\in\ga$
of any limiting procedure giving
$\ga=\lim\ga_p$
we have
$\phi(x)=\lim v'(q)\in\ga$.
The points of type
$x=\lim v(q)$
with dyadic
$q$
are dense in
$\ga$,
thus 
$\phi$
preserves
$\ga$, $\phi(\ga)=\ga$.
\end{proof}

\begin{lem}\label{lem:zigzag_change_basepoint} Let
$\ga=\ga(o,\cL,\cS)\sub X_\om$
be a zigzag curve with base point
$o\in X_\om$,
where
$\cL=\{l_1,\dots,l_k\}$, $\cS=\{s_1,\dots,s_k\}$, 
$s_1+\dots+s_k>0$.
Then for any
$o'\in\ga$
we have
$\ga(o',\cL,\cS)=\ga$,
i.e. any zigzag curve 
$\ga$
is independent of a choice of
its base point
$o$.
\end{lem}

\begin{proof} We first consider the case
$o'=\ga(t_q)$
is a {\em dyadic} point with dyadic
$q=m/2^r$, $m\in\Z$, $r\ge 0$,
and 
$t_q=(s_1+\dots+s_k)q$,
for the canonical parameterization
$t\mapsto\ga(t)$
of
$\ga$.
Then 
$o'$
is an accumulation point of the vertices
$v_p=v_p^{q_p}=\ga_p(t_p^{q_p})$, $p\ge r+1$,
where
$q_p=2^{p-(r+1)}\cdot km$
and
$t_p^{q_p}=(s_1+\dots+s_k)(q_p/k)/2^{p-1}=t_q$,
of approximating piecewise geodesic curves
$\ga_p$, $\ga=\lim\ga_p$
(recall that the sequence
$v(q)=\set{v_p^{q_p}}{$p\ge r+1$}$
lies in a fiber
$F(q)\sub X_\om$
of the projection
$\pi_\om$,
see the proof of Lemma~\ref{lem:zigzag_preserved}).
That is,
$o'=\lim v_p$
for our limiting procedure. 

By Lemma~\ref{lem:shift_identity}, there is a shift
$\eta_p=\eta_{v_po'}:X_\om\to X_\om$
with
$\eta_p(v_p)=o'$
and
$\lim\eta_p=\id$.
Then
$\eta_p(\ga_p)=\ga_p'$,
where
$\ga_p'=\ga_p(o',\cL,\cS)$
is the piecewise geodesic curve with the base point 
$o'$
approximating the zigzag curve
$\ga'=\ga(o',\cL,\cS)$, $\ga'=\lim\ga_p'$.
Now for an arbitrary point
$x\in\ga$, $x=\ga(t)$,
we have
$x=\lim\ga_p(t)$.
We put
$x'=\ga'(t)=\lim\ga_p'(t)$.
Then for an arbitrary
$\ep>0$
we have
$|x\ga_p(t)|<\ep$, $|x'\ga_p'(t)|<\ep$,
and
$|x\eta_p(x)|<\ep$
for sufficiently large
$p$.
The last estimate holds since
$\lim\eta_p=\id$.
Using
$|\eta_p(x)\ga_p'(t)|=|\eta_p(x)\eta_p\circ\ga_p(t)|
 =|x\ga_p(t)|$,
we obtain 
$$|xx'|\le|x\eta_p(x)|+|\eta_p(x)\ga_p'(t)|+|\ga_p'(t)x'|\le 3\ep,$$
thus
$x=x'$,
that is,
$\ga=\ga'$.

For a general case, the point
$o'=\ga(t)$
can be approximated by dyadic ones,
$t_q\to t$.
Then respective piecewise geodesic curves
$\ga_{p,q}'$
with dyadic base points
$\ga(t_q)$
approximate pointwise the curve
$\ga_p'$
with the base point 
$o'$
for every 
$p\ge 1$.
Thus
$\ga'=\ga$
also in that case.
\end{proof}

\begin{pro}\label{pro:zigzag_geodesic} Every zigzag curve
$\ga\sub X_\om$
is either a geodesic and hence a Ptolemy line, or it degenerates
to a point. More precisely, if
$\ga=\ga(o,\cL,\cS)$
for some base point 
$o\in X_\om$
and collections
$\cL=\{l_1,\dots,l_k\}$
of oriented Ptolemy lines in
$X_\om$, $\cS=\{s_1,\dots,s_k\}$
of nonnegative numbers with
$s_1+\dots+s_k>0$,
then
$\ga$
is degenerate if and only if
$\sum_i\al_is_i=0$
for every oriented Ptolemy line
$l\sub X_\om$,
where
$\al_i=\slope(l_i,l)$, $i=1,\dots,k$.
\end{pro}

\begin{proof} We first show that for each
$o$, $o'\in\ga$
there is a midpoint
$x\in\ga$.
By Lemma~\ref{lem:zigzag_change_basepoint} and
Lemma~\ref{lem:zigzag_preserved}, homotheties
$\phi\in\Ga_{\om,o}$, $\phi'\in\Ga_{\om,o'}$
with coefficient
$\la=1/2$
both preserve
$\ga$, $\phi(\ga)=\ga=\phi'(\ga)$.
Then for  
$x=\phi(o')\in\ga$
we have
$|ox|=|oo'|/2$
and similarly
for  
$x'=\phi'(o)\in\ga$
we have
$|x'o'|=|oo'|/2$.
Furthermore, the length of the segment of
$\ga$
between
$o$, $x$
is half of the length of the segment between
$o$, $o'$, $L([ox]_\ga)=L([oo']_\ga)/2$.
Thus
$L([xo']_\ga)=L([oo']_\ga)/2$
by additivity of the length. Then
$L([x'o']_\ga)=L([oo']_\ga)/2=L([xo']_\ga)$
and hence
$x'=x$
by monotonicity of the length, and
$x$
is the required midpoint.

It follows that the segment of
$\ga$
between its any two points is geodesic. Since 
$\ga$
is invariant under the nontrivial homothety
$\phi\in\Ga_{\om,o}$,
we see that 
$\ga$
is a Ptolemy line unless it degenerates to a point.

If 
$\ga$
is degenerate, then any Busemann function
$b:X_\om\to\R$
is constant along 
$\ga$.
By Lemma~\ref{lem:busemann_affine_zigzag}, we have
$\sum_i\al_is_i=0$
for every oriented Ptolemy line 
$l\sub X_\om$, 
where
$\al_i=\slope(l_i,l)$, $i=1,\dots,k$.
Conversely, if
$\ga$
is nondegenerate, then 
$l=\ga(\R)$
is a Ptolemy line in
$X_\om$
by the first part of the proof,
and the associated Busemann function 
$b:X_\om\to\R$
is nonconstant along
$l$.
By Lemma~\ref{lem:busemann_affine_zigzag}, we have
$b\circ\ga(t)=\be t$
for the canonical parameterization of
$\ga$
with
$\be=\sum_i\al_is_i/\sum_is_i$,
$\al_i=\slope(l_i,l)$, $i=1,\dots,k$.
Thus
$\sum_i\al_is_i\neq 0$.
\end{proof}

\begin{rem}\label{rem:distinct_fiber_zigzag} Assume
$\ga_1$
is a piecewise geodesic curve (with finite number
of edges) between different fibers in
$X_\om$, $\ga_1(0)\in F$, $\ga_1(s_1+\dots+s_k)\in F'$,
$F\neq F'$,
where
$s_1,\dots,s_k$
are the lengths of its edges. Then the respective zigzag curve
$\ga=\lim\ga_p$
is not degenerate. This follows from
$\ga(0)\in F$, $\ga(s_1+\dots+s_k)\in F'$
by construction of the approximating sequence
$\ga_1,\dots,\ga_p,\dots\to\ga$. 
\end{rem}

Now, we compute a unit speed parameterization a zigzag
curve
$\ga=\ga(o,\cL,S)$
assuming for simplicity that the collection
$\cL$
consists of mutually orthogonal Ptolemy lines.

\begin{lem}\label{lem:uspeed_parameter_zigzag} Let 
$\cL=\{l_1,\dots,l_k\}$
be a collection of mutually orthogonal oriented Ptolemy lines in
$X_\om$, $l_i\bot l_j$
for 
$i\neq j$, $S=\{s_1,\dots,s_k\}$
a collection of nonnegative numbers with
$\sum_is_i=1$.
Then for every
$o\in X_\om$
the zigzag curve 
$\ga=\ga(o,\cL,S)$
is nondegenerate, for the Ptolemy line
$l=\ga(\R)$
we have
$$\slope(l,l_i)=-s_i/\la$$
for every
$i=1,\dots,k$,
where
$\la=\sqrt{\sum_is_i^2}$,
and 
$|o\ga(t)|=\la|t|,\ t\in\R$,
for the canonical parameterization 
$t\mapsto\ga(t)$
of
$\ga$.
\end{lem}

\begin{proof} By Proposition~\ref{pro:zigzag_geodesic},
$l=\ga(\R)$
is a Ptolemy line or it degenerates to a point. We put 
$\la=|o\ga(1)|\ge 0$.
Let
$b_i$, $i=1,\dots,k$,
be the Busemann function of
$l_i$
normalized by
$b_i(o)=0$.
By Lemma~\ref{lem:busemann_affine_zigzag},
$$b_i\circ\ga(1)=\sum_j\slope(l_j,l_i)s_j=-s_i$$
because the lines of
$\cL$
are mutually orthogonal. Since
$\sum_is_i=1$
there is
$i$
with
$s_i\neq 0$.
Hence
$\ga$
is nondegenerate and
$\la>0$.

On the other hand, by definition of
$\al_i=\slope(l,l_i)$
we have
$b_i\circ\ga(1)=\al_i|o\ga(1)|=\al_i\la$.
Thus
$\al_i\la=-s_i$
for every 
$i=1,\dots,k$.

Let
$b$
be the Busemann function of
$l$
normalized by
$b(o)=0$
and
$b\circ\ga(t)<0$
for
$t>0$.
Then using again
Lemma~\ref{lem:busemann_affine_zigzag} and the symmetry of the slope,
$\al_i=\slope(l_i,l)$,
we obtain
$$\la^2=\la|o\ga(1)|=-\la b\circ\ga(1)=-\la\sum_i\al_is_i=\sum_is_i^2.$$
Thus
$\slope(l,l_i)=-s_i/\la$
with 
$\la=\sqrt{\sum_is_i^2}$. 

The shift
$\eta=\eta_{o\ga(1)}:X_\om\to X_\om$
leaves
$l$
invariant and
$\eta\circ\ga(1)=\ga(2)$
by construction of
$\ga=\ga(o,\cL,S)$.
Thus
$|o\ga(2)|=2|o\ga(1)|=2\la$.
By Lemma~\ref{lem:zigzag_preserved},
$\phi\circ\ga(1)=\ga(1/2)$,
where
$\phi\in\Ga_{\om,o}$
is the homothety with coefficient
$1/2$.
Thus
$|o\ga(1/2)|=\la/2$.
From this one easily sees that
$$|o\ga(t)|=\la|t|$$
first for dyadic and then by continuity for all 
$t\in\R$.
\end{proof}

\subsection{Orthogonalization procedure}
\label{subsect:existence}

As usual, we fix
$\om\in X$
and a metric 
$d$
of the M\"obius structure with infinitely
remote point 
$\om$.

\begin{pro}\label{pro:orthogonal_base} There is a finite
collection
$\cL_\bot$
of mutually orthogonal Ptolemy lines such that for every
$x\in X_\om$
the fiber
$F\sub X_\om$
through
$x$
of the fibration
$\pi_\om:X_\om\to B_\om$
is represented as
$F=\cap_{l\in\cL_\bot}H_l$,
where
$H_l$
is the horosphere of
$l$
through
$x$. 
\end{pro}

We first note that the cardinality of any collection
of mutually orthogonal Ptolemy lines in
$X_\om$
is uniformly bounded above.

\begin{lem}\label{lem:orthogonal_bound} There is
$N\in\N$
such that the cardinality of any collection
$\cL$
of mutually orthogonal Ptolemy lines in
$X_\om$
is bounded by
$N$, $|\cL|\le N$.
\end{lem}

\begin{proof} We fix 
$x\in X_\om$
and assume W.L.G. that all the lines of
$\cL$
pass through
$x$.
On every line
$l\in\cL$,
we fix a point at the distance 1 from
$x$. 
Let
$A\sub X_\om$
be the set of obtained points. By compactness of
$X$
and homogeneity of
$X_\om$
it suffices to show that the distance
$|aa'|\ge 1$
for each distinct
$a$, $a'\in A$.
We have
$a\in l$, $a'\in l'$
for some distinct lines
$l$, $l'\in\cL$.
Since
$\slope(l',l)=0$,
the lines
$l$, $l'$
are also orthogonal at the infinite remote point
$\om$
according Lemma~\ref{lem:opposite_slopes} and 
Lemma~\ref{lem:slope_symmetry}. Thus in the space
$X_x$
with infinitely remote point 
$x$,
the Ptolemy line 
$(l'\cup\om)\sm x\sub X_x$
lies in the horosphere 
$H$
of the line
$(l\cup\om)\sm x\sub X_x$
through
$\om$.
By duality, see Lemma~\ref{lem:flat_duality}, the point 
$x\in l$
is closest on the line 
$l$
to any fixed point of 
$l'\sub X_\om$, 
in particular
$|a'a|\ge|a'x|=1$. 
\end{proof}

Next, we describe an orthogonalization procedure.

\begin{lem}\label{lem:orthogonalization_procedure} Let
$l_1,\dots,l_k$
be a collection of mutually orthogonal Ptolemy lines in
$X_\om$, $l_i\bot l_j$
for
$i\neq j$.
Given a Ptolemy line 
$l\sub X_\om$,
through any
$o\in X_\om$
there is a zigzag curve
$\ga=\ga(o,\cL,S)$,
where
$\cL=\{l_1,\dots,l_k,l\}$
is an ordered collection of oriented Ptolemy lines,
$S=\{s_1,\dots,s_{k+1}\}$
a collection of nonnegative numbers with
$s_1+\dots+s_{k+1}>0$,
which is orthogonal to
$l_1,\dots,l_k$, $\ga(\R)=l_{k+1}\bot l_i$
for
$i=1,\dots,k$.
Furthermore, if
$\sum_1^k\al_i^2\neq 1$,
where
$\al_i=\slope(l,l_i)$,
then
$\ga$
is nondegenerate, and
$l_{k+1}$
is a Ptolemy line. 
\end{lem}

\begin{proof} We fix an orientation of
$l$
and for every
$i=1,\dots,k$
we choose an orientation of
$l_i$
so that
$\al_i=\slope(l,l_i)\ge 0$,
and put
$\al:=\sum_i\al_i\ge 0$.
For any zigzag curve
$\ga=\ga(o,\cL,S)$
in
$X_\om$,
where
$\cL=\{l_1,\dots,l_k,l\}$, $S=\{s_1,\dots,s_{k+1}\}$,
for 
$i=1,\dots,k$
and for the Busemann function
$b_i$
of
$l_i$
with
$b_i(o)=0$,
by Lemma~\ref{lem:busemann_affine_zigzag} we have 
$b_i\circ\ga(t)=\be_it$
for all 
$t\in\R$,
where
$\be_i=(-s_i+\al_is_{k+1})/(s_1+\dots+s_{k+1})$.
Then putting
$s_i=\frac{\al_i}{1+\al}$, $i=1,\dots,k$,
$s_{k+1}=\frac{1}{1+\al}$,
we have
$s_1+\dots s_{k+1}=1$
and
$\be_i=0$
for every
$i=1,\dots,k$.
Thus 
$\ga$
is orthogonal to
$l_1,\dots,l_k$,
and it gives us a required Ptolemy line
$l_{k+1}=\ga(\R)$
unless
$\ga$
degenerates.
 
Let
$b$
be the Busemann function of
$l$
with 
$b(o)=0$. 
By Lemma~\ref{lem:busemann_affine_zigzag} we have
$b\circ\ga(t)=\be t$
for all 
$t\in\R$
with
$\be=\sum_1^k\slope(l_i,l)s_i-s_{k+1}$.
Using the symmetry of the slope, we see that
$\slope(l_i,l)=\slope(l,l_i)=\al_i$
and thus
$\be=(\sum_1^k\al_i^2-1)/(1+\al)$.
If
$\sum_1^k\al_i^2\neq 1$,
then this shows that
$\ga$
is nondegenerate.
\end{proof}

We say that a collection
$\{l_1,\dots,l_k\}$
mutually orthogonal Ptolemy lines in
$X_\om$
is {\em maximal} if there is no Ptolemy line in
$X_\om$
which is orthogonal to every
$l_1,\dots,l_k$.
By Lemma~\ref{lem:orthogonal_bound}, such a collection exists.

\begin{lem}\label{lem:max_orthogonal} Let
$\{l_1,\dots,l_k\}$
be a maximal collection of mutually orthogonal Ptolemy lines in
$X_\om$.
Then every Ptolemy line 
$l\sub X_\om$
can be represented as a zigzag curve
$\ga(o,\cL,S)$
with
$o\in l$
for a collection
$\cL=\{l_1,\dots,l_k\}$
of oriented Ptolemy lines and a collection 
$S=\{s_1,\dots,s_k\}$
of nonnegative numbers with
$s_1+\dots+s_k>0$.
Furthermore, we have
$\sum_i\al_i^2=1$,
where
$\al_i=\slope(l,l_i)$, $i=1,\dots,k$.
\end{lem}

\begin{proof} We apply the orthogonalization procedure described 
in Lemma~\ref{lem:orthogonalization_procedure} to the collection
$\cL_l=\{l_1,\dots,l_k,l\}$
and construct a zigzag curve
$\ga_l=\ga_l(o,\cL_l,S_l)$,
where
$S_l=\{s_1,\dots,s_{k+1}\}$
is the collection of nonnegative numbers described there.
Since
$\ga_l$
of orthogonal to
$l_1,\dots,l_k$,
we conclude from maximality of
$\{l_1,\dots,l_k\}$
that 
$\ga_l$
degenerates and moreover
$\sum_i\al_i^2=1$
by Lemma~\ref{lem:orthogonalization_procedure}.
 
According to Remark~\ref{rem:distinct_fiber_zigzag},
the ends of the piecewise geodesic curve
$\ga_{l,1}$
with 
$k+1$
edges
$\si_1,\dots,\si_k,\si$
on Ptolemy lines Busemann parallel to
$l_1,\dots,l_k,l$
respectively with
$|\si_i|=s_i$, $i=1,\dots,k$, $|\si|=s_{k+1}$,
lie in one and the same fiber ($\K$-line)
$F\sub X_\om$, 
that is,
$o=\ga_{l,1}(0)$
and 
$x=\ga_{l,1}(s+s_{k+1})\in F$,
where
$s=s_1+\dots+s_k$.
Thus the reduced piecewise geodesic curve
$\ga_1=\si_1\cup\dots\cup\si_k$
and the last edge
$\si$
of
$\ga_{l,1}$
have the ends in the same fibers
$F$, $F'$,
where
$F'$
is the fiber through
$\ga_{l,1}(s)=\ga_1(s)$.

Then the zigzag curve
$\ga=\ga(o,\cL,S)$
is nondegenerate, where 
$\cL=\{l_1,\dots,l_k\}$, $S=\{s_1,\dots,s_k\}$,
and it gives a Ptolemy line 
$\ga(\R)\sub X_\om$
though
$o$
which hits the fiber
$F'$.
Since the Ptolemy line
$l'$
containing the segment
$\si$
is Busemann parallel to
$l$
and intersects the fibers
$F$, $F'$,
the line 
$\ga(\R)$
is Busemann parallel to
$l$,
see Lemma~\ref{lem:rfoliation_semik}. Hence
$\ga(\R)=l$
by uniqueness, see Lemma~\ref{lem:unique_line}. 
\end{proof}

\begin{lem}\label{lem:linear_combination_busemann} Let
$\cL=\{l_1,\dots,l_k\}$
be a maximal collection of mutually orthogonal oriented Ptolemy lines in
$X_\om$, $S=\{s_1,\dots,s_k\}$
a collection of nonnegative numbers with 
$s_1+\dots+s_k=1$, $b$, $b_i:X_\om\to\R$
the Busemann functions of the zigzag curve
$\ga=\ga(o,\cL,S)$,
Ptolemy line 
$l_i$
with
$b(o)=0=b_i(o)$
respectively,
$i=1,\dots,k$.
Then
$\la b=\sum_is_ib_i$,
where
$\la=\sqrt{\sum_is_i^2}$. 
\end{lem}

\begin{proof} We denote by
$h=\la b-\sum_is_ib_i$
the function
$X_\om\to\R$
with
$h(o)=0$,
which is an affine function on every Ptolemy line on
$X_\om$.
First, we check that 
$h$
vanishes along 
$l_1,\dots,l_k$
(assuming that these lines pass through
$o$).
Indeed,
$\sum_is_ib_i(z)=s_jb_j(z)$
for every
$z\in l_j$
because
$l_i\bot l_j$
for
$i\neq j$.
Since
$b$
is a Busemann function of
$l$,
it is affine on
$l_j$
with the coefficient
$\slope(l_j,l)$,
$b(z)=-\slope(l_j,l)b_j(z)$.
Using symmetry of the slope, we obtain
$$\slope(l_j,l)=\slope(l,l_j)=-s_j/\la,$$
see Lemma~\ref{lem:uspeed_parameter_zigzag}. Thus
$\la b(z)=s_jb_j(z)$,
and
$h(z)=0$.

Next, we show that if
$h$
is constant on a Ptolemy line
$l$, 
then it is constant on every Ptolemy line
$l'$
that is Busemann parallel to
$l$
(with maybe a different value). By Lemma~\ref{lem:busparallel_sublinear} 
we know that
$l$, $l'$
diverge at most sublinearly,  
and also that
$h$
is affine on
$l'$.
Thus 
$h|l'$
cannot be nonconstant because 
$h$
is a Lipschitz function on
$X_\om$.

It follows that 
$h$
vanishes on every piecewise geodesic curve with origin
$o$
and with edges Busemann parallel to the lines
$l_1,\dots,l_k$.
Hence,
$h$
vanishes along any zigzag curve of type
$\ga(o,\cL,S)$.
Using Lemma~\ref{lem:max_orthogonal}, we conclude that
$h$
is constant along any Ptolemy line in
$X_\om$.

By Proposition~\ref{pro:tangent_rcircle}, every Ptolemy
circle possesses a unique tangent line, which is certainly
a Ptolemy line, at every point. Using standard approximation
arguments, we see that
$h$
is constant along any Ptolemy circle in
$X_\om$.
By the existence property (E), every
$x\in X_\om$
is connected with 
$o$
by a Ptolemy circle. Thus
$h(x)=0$
and
$\la b=\sum_ib_i$. 
\end{proof}

\begin{proof}[Proof of Proposition~\ref{pro:orthogonal_base}]
Let
$\cL_\bot=\{l_1,\dots,l_k\}$
be a maximal collection of mutually orthogonal Ptolemy lines in
$X_\om$.
This means that we actually consider respective foliations of
$X_\om$
by Busemann parallel Ptolemy lines.
For any
$x\in X_\om$,
for the fiber
$F$
and for the respective lines from
$\cL_\bot$
through 
$x$,
we have by definition
$F\sub\cap_jH_j$,
where
$H_j\sub X_\om$
is the horosphere of
$l_j$
through
$x$.
It follows from Lemma~\ref{lem:max_orthogonal} and
Lemma~\ref{lem:linear_combination_busemann} that any
Busemann function
$b:X_\om\to\R$
with 
$b(x)=0$
is a linear combination of the Busemann functions
$b_1,\dots,b_k$
of the lines
$l_1,\dots,l_k$
which vanish at
$x$.
Thus 
$b$
vanishes on
$\cap_jH_j$,
and therefore
$F=\cap_jH_j$.
\end{proof}

\begin{proof}[Proof the property ($\K$)] Given a fiber 
($\K$-line) $F\sub X_\om$
and
$x\in X_\om\sm F$,
we show that there is a Ptolemy line 
$l\sub X_\om$
through
$x$
that hits
$F$.
Uniqueness of
$l$
is proved above just after Lemma~\ref{lem:induced_base_map}.

Using Proposition~\ref{pro:orthogonal_base}, we represent
$F=\cap_{l\in\cL_\bot}H_l$,
where
$\cL_\bot$
is a finite collection of mutually orthogonal Ptolemy lines,
$H_l$
a horosphere of
$l$.
Choosing appropriate orientations of the members of
$\cL_\bot$,
we can assume that
$H_l=b_l^{-1}(0)$
and
$b_l(x)\ge 0$
for every
$l\in\cL_\bot$,
where
$b_l:X_\om\to\R$
is a Busemann function of
$l$.
Moving from
$x$
in an appropriate direction along a Ptolemy line,
which is Busemann parallel to 
$l\in\cL_\bot$
with
$b_l(x)>0$,
we reduce the value of
$b_l$
to zero keeping up every other Busemann function
$b_{l'}$, $l'\in\cL_\bot$,
constant. Repeating this procedure, we connect
$x$
with 
$F$
by a piecewise geodesic curve with at most
$|\cL_\bot|$
edges. Now, the zigzag construction produces a required
Ptolemy line through
$x$
that hits
$F$.
\end{proof}

\subsection{Properties of the base $B_\om$}
\label{subsect:base_properties}

We fix
$\om\in X$
and a metric 
$d$
from the M\"obius structure for which
$\om$
is infinitely remote. We also use notation
$|xy|=d(x,y)$
for the distance between
$x$, $y\in X_\om$.

\begin{lem}\label{lem:homothety_vert_shift} Let
$\phi:X_\om\to X_\om$
be a pure homothety with
$\phi(o)=o$, $o\in X_\om$.
Then
$\phi$
preserves the fiber ($\K$-line)
$F$
through
$o$, $\phi(F)=F$.
In particular, every shift
$\eta_{xx'}:X_\om\to X_\om$
with
$x$, $x'\in F$
preserves
$F$.
\end{lem}

\begin{proof} As in Lemma~\ref{lem:pure_homothety}, we have
$\la b\circ\phi=b$
for any Busemann function
$b:X_\om\to\R$
with 
$b(o)=0$,
where
$\la$
is the coefficient of the homothety
$\phi$.
Since
$b(x)=0$
for every
$x\in F$,
we see that
$b\circ\phi(x)=0$,
and thus
$\phi(x)\in F$,
that is,
$\phi(F)=F$.
The assertion about a shift
$\eta_{xx'}$
follows now from the definition of
$\eta_{xx'}$.
\end{proof}

\begin{lem}\label{lem:vert_shift} Let
$\eta:X_\om\to X_\om$
be a shift that preserves a $\K$-line
$F\sub X_\om$.
Then 
$\eta$
preserves any other $\K$-line 
$F'\sub X_\om$.
\end{lem}

\begin{proof} Let
$b:X_\om\to\R$
be a Busemann function associated with an (oriented)
Ptolemy line
$l\sub X_\om$.
Then for any isometry
$\eta:X_\om\to X_\om$
the function
$b\circ\eta$
is a Busemann function associated with 
Ptolemy line 
$\eta^{-1}(l)$.
Thus for an arbitrary shift
$\eta:X_\om\to X_\om$,
we have
$b\circ\eta=b+c_b$,
where
$c_b\in\R$
is a constant depending on
$b$,
because the line 
$\eta^{-1}(l)$
is Busemann parallel to
$l$,
hence the function
$b\circ\eta$
is also a Busemann function of
$l$,
and thus the functions
$b$, $b\circ\eta$
differ by a constant. 

In our case, when
$\eta$
preserves a $\K$-line, this constant is zero,
$c_b=0$,
thus
$b\circ\eta=b$
for any Busemann function
$b:X_\om\to X_\om$.
Therefore,
$\eta$
preserves any $\K$-line.
\end{proof}

We define
$$d(F,F')=\inf\set{|xx'|}{$x\in F,\ x'\in F'$}$$
for $\K$-lines
$F$, $F'\sub X_\om$.

\begin{lem}\label{lem:klines_distance} Given $\K$-lines
$F$, $F'\sub X_\om$,
and
$x\in F$,
there is
$x'\in F'$
such that
$d(F,F')=|xx'|$.
\end{lem}

\begin{proof} Let
$x_i\in F$, $x_i'\in F'$
be sequences with 
$|x_ix_i'|\to d(F,F')$.
Using Lemma~\ref{lem:vert_shift}, we can assume that
$x_i=x$
for all
$i$.
Then the sequence
$x_i'$
is bounded, and by compactness of
$X$
it subconverges to
$x'\in F'$
with
$|xx'|=d(F,F')$.
\end{proof}

\begin{lem}\label{lem:line_dist_klines} For any $\K$-lines
$F$, $F'\sub X_\om$,
we have
$d(F,F')=|xx'|$,
where
$x=l\cap F$, $x'=l\cap F'$,
and 
$l$
is any Ptolemy line in
$X_\om$
that meets both
$F$, $F'$.
\end{lem}

\begin{proof} By Lemma~\ref{lem:rfoliation_semik}, the distance
$|xx'|$
is independent of the choice of
$l$.
By definition
$|xx'|\ge d(F,F')$.
By Lemma~\ref{lem:klines_distance}, there is
$x''\in F'$
with
$|xx''|=d(F,F')$.
The horosphere 
$H$
of (a Busemann function associated with)
$l$
through
$x'$
contains
$F'$,
in particular,
$x''\in H$.
Then
$|xx''|\ge|xx'|$
and hence,
$d(F,F')=|xx'|$.
\end{proof}

Let
$\pi_\om:X_\om\to B_\om$
be the canonical fibration, see sect.~\ref{sect:fibration}.
For
$b\in B_\om$,
we denote by
$F_b=\pi_\om^{-1}(b)$
the $\K$-line over 
$b$.
For
$b$, $b'\in B_\om$
we put
$|bb'|:=|xx'|$,
where
$x=l\cap F_b$, $x'=l\cap F_{b'}$,
and
$l\sub X_\om$
is any Ptolemy line that meets both
$F_b$
and
$F_{b'}$.
By property ($\K$),
such a line 
$l$
exists, by Lemma~\ref{lem:line_dist_klines}, the number
$|bb'|$
is well defined, and the function
$(b,b')\mapsto|bb'|$
is a metric on
$B_\om$.
This metric is said to be {\em canonical}.

\begin{pro}\label{pro:base_metric} The canonical projection
$\pi_\om:X_\om\to B_\om$
is a 1-Lipschitz submetry with respect to the canonical
metric on
$B_\om$.
Furthermore,
$B_\om$
is a geodesic metric space with the property
that through any two distinct points
$b$, $b'\in B_\om$
there is a unique geodesic line in
$B_\om$.
\end{pro}

\begin{proof}
It follows from Lemma~\ref{lem:line_dist_klines} that the map
$\pi_\om$
is 1-Lipschitz. Let
$D=D_r(o)$
be the metric ball in
$X_\om$
of radius
$r>0$
centered at a point
$o\in X_\om$, $D'\sub B_\om$
the metric ball of the same radius
$r$
centered at
$\pi_\om(o)$.
The inclusion
$D'\sub\pi_\om(D)$
follows from the definition of the metric of
$B_\om$.
The opposite inclusion
$D'\sups\pi_\om(D)$
holds because
$\pi_\om$
is 1-Lipschitz. Thus
$\pi_\om:X_\om\to B_\om$
is a 1-Lipschitz submetry.

Furthermore, by Lemma~\ref{lem:line_dist_klines},
the projection
$\pi_\om$
restricted to every Ptolemy line in
$X_\om$
is isometric, and thus by property~($\K$), the base
$B_\om$
is a geodesic metric space. Moreover, it follows
from ($\K$) that through any two distinct
points
$b$, $b'\in B_\om$
there is a unique geodesic line in
$B_\om$.
\end{proof}

\begin{cor}\label{cor:pi_submetry} For any homothety
$\phi:X_\om\to X_\om$,
the induced map
$\pi_\ast(\phi):B_\om\to B_\om$
is a homothety with the same dilatation coefficient.
\qed
\end{cor}

\begin{pro}\label{pro:base_euclidean} The base
$B_\om$
is isometric to an Euclidean
$\R^k$
for some 
$k>0$. 
\end{pro}

\begin{proof} Any Busemann function
$b:X_\om\to\R$
is affine on Ptolemy lines by Corollary~\ref{cor:busemann_affine}.
By definition,
$b$
is constant on the fibers of
$\pi_\om$,
thus it determines a function
$\ov b:B_\om\to\R$
such that
$\ov b\circ\pi_\om=b$.
This function is affine on geodesic lines in
$B_\om$
because every geodesic line
$\ov l\sub B_\om$
is of the form
$\ov l=\pi_\om(l)$
for some Ptolemy line
$l\sub X_\om$,
and each unit speed parameterization
$c:\R\to X_\om$
of
$l$
induces the unit speed parameterization
$\ov c=\pi_\om\circ c$
of
$\ov l$.
Then
$\ov b\circ\ov c=\ov b\circ\pi_\om\circ c=b\circ c$
is an affine function on
$\R$.

We fix a base point 
$o\in X_\om$
and a maximal collection
$\cL=\{l_1,\dots,l_k\}$
of mutually orthogonal oriented Ptolemy lines of
$X_\om$
through
$o$.
Let
$b_1,\dots,b_k$
be Busemann functions of the lines
$l_1,\dots,l_k$
respectively that vanish at
$o$.
We denote by
$\ov l_i$
the projection of
$l_i$
to
$B_\om$,
and by
$\ov b_i:B_\om\to\R$
the function corresponding to
$b_i$, $i=1,\dots,k$.
By Proposition~\ref{pro:orthogonal_base}, the functions
$b_1,\dots,b_k$
separates fibers in
$X_\om$.
Thus the functions
$\ov b_1,\dots,\ov b_k$
separates points of 
$B_\om$,
that is, for each
$z$, $z'\in B_\om$
there is 
$i$
with
$\ov b_i(z)\neq\ov b_i(z')$.
Therefore, the continuous map
$h:B_\om\to\R^k$, $h(z)=(\ov b_1(z),\dots,\ov b_k(z))$
is injective. This map is surjective by the same argument
as in the proof of the property ($\K$), and it introduces
coordinates on
$B_\om$.
We compute the distance on
$B_\om$
in these coordinates. Applying a shift if necessary,
see Corollary~\ref{cor:pi_submetry}, we consider W.L.G. the distance
$|\ov o\,\ov z|$
for every
$z\in X_\om$,
where
$\ov z=\pi_\om(z)$.
By ($\K$) there is a unique Ptolemy line 
$l\sub X_\om$
through
$o$
that hits the fiber
$F_z$
through
$z$.
It follows from our definitions that for
$\al_i=\slope(l,l_i)$
we have
$\ov b_i(\ov z)=\al_i|\ov o\,\ov z|$, $i=1,\dots,k$.
By Lemma~\ref{lem:max_orthogonal},
$\sum_i\al_i^2=1$,
thus
$|\ov o\,\ov z|^2=\sum_i\ov b_i^2(\ov z)$.
This shows that 
$B_\om$
is isometric to an Euclidean
$\R^k$.
We have
$k>0$,
because there is a Ptolemy line in
$X_\om$.
\end{proof}

This completes the proof of Theorem~\ref{thm:basic_ptolemy}.

\section{Lie group structure of $X_\om$ and of $\K$-lines}
\label{sect:topology_space}

Here we recover a natural group structure on every space
$X_\om$, $\om\in X$,
which is a simply connected nilpotent group Lie.

\subsection{Groups of shifts}
\label{subsect:group_shifts}

Recall that by Lemma~\ref{lem:pure_homothety} a shift
$\eta:X_\om\to X_\om$
is an isometry that preserves every foliation of 
$X_\om$
by (oriented) Busemann parallel Ptolemy lines. Clearly, the shifts of
$X_\om$
form a group which we denote by 
$N_\om$.
Then
$N_\om$
is a subgroup of the group
$\aut X$
of the M\"obius automorphisms of
$X$.

\begin{lem}\label{lem:set_with_lines} Let
$V\sub X$
be a closed subset containing with every point
$z\in V$, $z\neq\om'$,
every Ptolemy circle through
$z$
and a fixed point 
$\om'\in X$
(we assume that
$V$
contains at least two points). Then
$V=X$.
\end{lem}

\begin{proof} Assume there is
$\om\in X\sm V$.
Since 
$V$
is closed,
$\om$
is contained in
$X\sm V$
together with some its neighborhood. Then
$V$
is compact in the space 
$X_\om$.
Thus the image
$\ov V=\pi_\om(V)\sub B_\om$
under the canonical projection
$\pi_\om:X_\om\to B_\om$
is compact. By Proposition~\ref{pro:base_euclidean}, 
$B_\om$
is isometric to
$\R^k$
for 
$k\ge 1$,
thus we can find a hyperplane
$E\sub B_\om$
supporting to
$\ov V$
at some point
$\ov z\in\ov V$,
i.e.
$\ov z\in E$
and 
$\ov V$
is contained in a closed half-space of
$B_\om$
bounded by
$E$.
Let 
$\ov l\sub B_\om$
be a geodesic line through
$\ov z$
that is transversal to
$E$.
We take
$z\in V$
with
$\pi_\om(z)=\ov z$,
and let
$l\sub X_\om$
be the Ptolemy line through
$z$
with 
$\pi_\om(l)=\ov l$. 
By Corollary~\ref{cor:tangent_unique}
there is a (unique) Ptolemy circle 
$\si\sub X_\om$
through
$z$
and
$\om'$
that is tangent to
$l$
(note that
$l$
misses
$\om'$
since otherwise
$l\sub V$
in contradiction with compactness of
$V$.)
By the assumption,
$\si\sub V$
and thus
$\ov\si=\pi_\om(\si)\sub\ov V$.
This is a contradiction, because the curve
$\ov\si$
is tangent to
$\ov l$
at
$\ov z$,
and therefore
$\ov\si\not\sub\ov V$.
\end{proof}

\begin{pro}\label{pro:simply_transitivity_shifts} The group
$N_\om$
acts simply transitively on
$X_\om$. 
\end{pro}

\begin{proof} Given
$x$, $x'\in X_\om$,
the shift
$\eta_{xx'}$
moves
$x$
to
$x'$, $\eta_{xx'}(x)=x'$,
by construction, see sect.~\ref{subsect:parallel_lines_pure_homothethies_shifts}. Thus
$N_\om$
acts transitively on
$X_\om$.

Assume that
$\eta(x)=x$
for some shift
$\eta:X_\om\to X_\om$
and some
$x\in X_\om$.
We denote by
$V$
the fixed point set of
$\eta$, $\eta(y)=y$
for every
$y\in V$
(in particular,
$\om\in V$).
We show that
$V=X$.
Note that every Ptolemy line 
$l\sub X_\om$,
which meets
$V$,
is contained in
$V$
because the isometry
$\eta$
preserves every foliation of
$X_\om$
by Busemann parallel Ptolemy lines.
Applying Lemma~\ref{lem:set_with_lines}, we obtain 
$V=X$,
and thus
$\eta=\id$,
i.e. the group
$N_\om$
acts simply transitively on
$X_\om$.
\end{proof}

We fix
$o\in X_\om$
and using Proposition~\ref{pro:simply_transitivity_shifts} identify
$N_\om$
with
$X_\om$
by
$\eta\mapsto\eta(o)$.
Then
$N_\om$
is a locally compact topological group.

An automorphism
$\tau:N_\om\to N_\om$
is said to be {\em contractible} if for every
$\eta\in N_\om$
we have
$\lim_{n\to\infty}\tau^n(\eta)=\id$.
If
$N_\om$
admits a contractible automorphism, then 
$N_\om$
is also said to be {\em contractible}.

\begin{lem}\label{lem:contract_auto} There is a contractible
automorphism
$\tau:N_\om\to N_\om$. 
\end{lem}

\begin{proof} We take any pure homothety
$\phi:X_\om\to X_\om$
with
$\phi(o)=o$
and with the coefficient
$\la\in(0,1)$.
Then we define
$\tau(\eta)=\phi\circ\eta\circ\phi^{-1}$.
The map 
$\eta'=\tau(\eta):X_\om\to X_\om$
is an isometry preserving every foliation of
$X_\om$
by Busemann parallel Ptolemy lines, i.e.
$\eta'$
is a shift, and it is clear that
$\tau$
is an automorphism of
$N_\om$.

For the sequence of shifts
$\eta_n=\tau^n(\eta)$
we have
$\eta_n(o)=\phi^n\circ\eta(o)\to o$
as
$n\to\infty$.
Thus
$\eta_n$
converges to a shift
$\eta_\infty$
with
$\eta_\infty(o)=o$,
hence,
$\eta_\infty=\id$. 
\end{proof}

\begin{cor}\label{cor:nilponent_lie_group} The group
$N_\om$
is a simply connected nilpotent Lie group. In particular, the space 
$X_\om$
is homeomorphic to
$\R^n$,
and the space 
$X$
is homeomorphic to the sphere
$S^n$
with 
$n=\dim X$. 
Furthermore, every metric ball
$B=B_r(o)=\set{x\in X_\om}{$|xo|\le r$}$
in
$X_\om$
is homeomorphic to the standard ball in
$\R^n$.
\end{cor}

\begin{proof} The group
$N_\om$
is connected and locally compact because the space 
$X_\om$
is. By Lemma~\ref{lem:contract_auto},
$N_\om$
is contractible. Then by \cite[Corollary~2.4]{Sieb}
$N_\om$
is a simply connected nilpotent Lie group. Thus
$X_\om\simeq\R^n$.

Let 
$h_\la:X_\om\to X_\om$
be the pure homothety with coefficient
$\la>0$
centered at
$o\in X_\om$, $h_\la(o)=o$, $h_1=\id$.
Then
$B=o\cup\set{h_\la(S)}{$0<\la\le 1$}$,
where
$S=\d B=\set{x\in X_\om}{$|xo|=r$}$.
This gives a representation
$\Int B=o\cup\set{(x,\la)}{$x\in S,\ 0<\la<1$}$.
Composing the homeomorphism
$g:[0,1)\to [0,\infty)$, $g(\la)=\tan\frac{\la\pi}{2}$
with the embedding
$f_x:[0,\infty)\to X_\om$, $f_x(0)=o$, $f_x(\la)=h_\la(x)$
for every
$x\in S$
we obtain a homeomorphism
$F:\Int B\to X_\om$, $F(x,\la)=f_x\circ g(\la)$.
Hence
$\Int B$
is homeomorphic to
$X_\om\simeq\R^n$
and
$B$
is homeomorphic to the standard ball in
$\R^n$.
\end{proof}

We denote by
$Z_\om$
a subgroup in
$N_\om$
which consists of all shifts
$\eta\in N_\om$
acting identically on the base 
$B_\om$, $\pi_\ast(\eta)=\id$,
where
$\pi_\ast(\eta):B_\om\to B_\om$
is the shift induced by the projection
$\pi_\om:X_\om\to B_\om$,
see Corollary~\ref{cor:pi_submetry}. Every
$\eta\in Z_\om$
preserves every fiber ($\K$-line) of 
$\pi_\om$,
see Lemma~\ref{lem:homothety_vert_shift} and 
Lemma~\ref{lem:vert_shift}.

\begin{pro}\label{pro:fiber_nilpotent_lie_group} The group
$Z_\om$
acts simply transitively on every $\K$-line
$F\sub X_\om$,
and thus it is a contractible, connected, locally compact 
topological group. Therefore,
$Z_\om$
is a simply connected nilpotent Lie group, and
$F$
is homeomorphic to
$\R^p$
for some
$0\le p<n$, $k+p=n$
for 
$k=\dim B_\om$.
\end{pro}

\begin{proof} The group
$Z_\om$
acts transitively on
$F$
by Lemma~\ref{lem:homothety_vert_shift}. The action is
simply transitive by Proposition~\ref{pro:simply_transitivity_shifts}.
We fix 
$o\in F$
and identify
$Z_\om$
with
$F$
by
$\eta\mapsto\eta(o)$.
By the same argument as in Lemma~\ref{lem:contract_auto} we see
that the group
$Z_\om$
is contractible. Furthermore,
$F$
is locally compact. Given
$x$, $x'\in F$,
there is a Ptolemy circle
$\si\sub X_\om$
through
$x$, $x'$.
By ($\K$), through any point 
$z\in\si$
there is a uniquely determined Ptolemy line that hits
$F$.
This defines a continuous map 
$\si\to F$.
Thus
$F$
is linearly connected. Hence,
$Z_\om$
is a contractible, locally compact, connected topological group.
By \cite[Corollary~2.4]{Sieb},
$Z_\om$
is a simply connected nilpotent Lie group, and thus
$F$
is homeomorphic to
$\R^p$
for some
$0\le p\le n$.
In fact 
$p<n$
because
$X$
contains Ptolemy circles and thus
$k=\dim B_\om>0$,
while
$n=k+p$.
\end{proof}

\subsection{Non-integrability of the canonical distribution}
\label{subsect:nonintegrability}

Given
$o\in X_\om$, $\ov x\in B_\om$,
by the property~($\K$) there is a unique
$x\in F_{\ov x}=\pi_\om^{-1}(\ov x)$
that is connected with
$o$
by a geodesic segment
$ox$.
The point
$x$
is called the {\em lift} of
$\ov x$
with respect to
$o$,
and we use notation
$x=\lift_o(\ov x)$.
This defines an embedding
$\lift_o:B_\om\to X_\om$
with
$\pi_\om\circ\lift_o=\id$
for every
$o\in X_\om$.
We denote
$D_o=\lift_o(B_\om)$.
The embedding
$\lift_o$
is radially isometric,
$|o\lift_o(\ov x)|=|\pi(o)\ov x|$
for every
$\ov x\in B_\om$.
Though there is no reason for
$\lift_o$
as well as for the projection
$\pi_\om|D_o$
to be isometric, the map
$\lift_o$
is continuous which follows the uniqueness property of
($\K$) and compactness of
$X$.

The family of subspaces
$D_o$, $o\in X_\om$,
is called the (canonical) {\em distribution} on
$X_\om$.
We say that the canonical distribution
$\cD=\set{D_o}{$o\in X_\om$}$
on
$X_\om$
is {\em integrable} if for any
$o\in X_\om$
and any
$o'\in D_o$,
the subspaces
$D_o$
and
$D_{o'}$
of
$X_\om$
coincide,
$D_o=D_{o'}$.
For example, if the base 
$B_\om$
is one-dimensional, then the canonical distribution
$\cD$
is obviously integrable.

\begin{pro}\label{pro:nonintegrable_canonical_distribution} The canonical distribution on
$X_\om$
is integrable for some
$\om\in X$
if and only if
$p=0$, 
i.e. every fiber of the projection
$\pi_\om$
is a point for every
$\om\in X$. 
In this case the space
$X$
is M\"obius equivalent to
$\wh\R^n$
with
$n=\dim X$.
\end{pro}

\begin{proof} Assume that the canonical distribution is integrable
for some
$\om\in X$.
Since
$D_z=D_o$
for every
$o\in X_\om$
and every
$z\in D_o$,
every Ptolemy line in
$X_\om$
through
$z$
is contained in
$D_o$.
By Lemma~\ref{lem:set_with_lines},
$D_o\cup\om=X$.
Hence
$p=0$, $n=k$,
and
$X$
is M\"obius equivalent to
$\wh\R^n$
with 
$n=\dim X$.
Conversely, assume 
$p=0$,
and 
$o'\in D_o\sub X_\om$, $\om\in X$.
Then by property~($\K$) for every
$x\in D_{o'}$
there is a unique Ptolemy line in
$X_\om$
through
$x$
that hits the fiber
$F_o=\{o\}$.
Hence
$D_{o'}\sub D_o$,
and similarly
$D_o\sub D_{o'}$,
that is,
$\cD$
is integrable for every
$\om\in X$.
\end{proof}

\begin{cor}\label{cor:nonintegrability} Assume
$p>0$,
that is, fibers of the canonical projections
$\pi_\om$, $\om\in X$,
are nondegenerate. Then the canonical distribution
on
$X_\om$
is non-integrable for every
$\om\in X$.
\qed
\end{cor}

The next corollary follows immediately from 
Proposition~\ref{pro:nonintegrable_canonical_distribution}.

\begin{cor}\label{cor:onedimensional_base} If the base
$B_\om$
of
$X_\om$
is one-dimensional (this is independent of
$\om\in X$), 
then
$p=0$
and
$X=\wh\R$.
\qed
\end{cor}

\subsection{Classification of 2-transitive actions and Theorem~\ref{thm:moebius}}
\label{subsect:2-transitive}

By Proposition~\ref{pro:two_point_homogeneous}, the group
$\aut X$
of M\"obius transformations of
$X$
is 2-transitive, see Remark~\ref{rem:2-transitive}.
It is known from the Tits' classification of 2-transitive actions that:

{\em If a topological group
$G$,
acting effectively and 2-transitively on a compact topological space
$X$,
which is not totally disconnected, is locally compact and
$\si$-compact
(that is,
$G$
is a countable union of compact subsets), then 
$G$
is a Lie group and
$X=G/G_x$
is a smooth and connected manifold, homeomorphic to 
either a projective space, or 
a sphere
$S^n$.
In the second case,
$G$
is isomorphic to the isometry group of a rank one symmetric space of non-compact
type (up to a subgroup of index 2),} 

\noindent
see \cite[Theorems~A, B and 3.3(a)]{Kr}.
Corollary~\ref{cor:nilponent_lie_group} excluded the first possibility.
Thus to complete the proof of Theorem~\ref{thm:moebius}, it would be
sufficient to check that 
$\aut X$
is locally compact and
$\si$-compact,
and that the M\"obius structure of
$X$
is uniquely determined by the respective appropriately normalized
symmetric space. However, from our point of view, this formal 
classification argument is not satisfactory. Instead, we give a direct, 
classification free proof of Theorem~\ref{thm:moebius} by showing how 
the symmetric space structure emerges from the M\"obius structure.
This is done in the next two sections.

\section{Extension of M\"obius automorphisms of circles}
\label{sect:extension_circles}

The main purpose of this section is to prove
the the following extension property.

\noindent
(${\rm E}_2$) Extension: any M\"obius map between any  Ptolemy circles in
$X$
extends to a M\"obius automorphism of
$X$.

\begin{pro}\label{pro:extension_property} Any compact Ptolemy space 
with properties (E) and (I) possesses the extension property
(${\rm E}_2$).
\end{pro}

The proof of this result is based on study of second order properties 
of Ptolemy circles like Lemma~\ref{lem:quaratic_excess}.

\subsection{Distance and arclength parameterizations of a circle}
\label{subsect:dist_unit_speed_param}

In this section, we establish existence of a distance
parameterization in a Ptolemy circle and study its 
relationship with an arclength parameterization. A distance
parameterization is convenient to obtain an important estimate
(\ref{eq:quadratic_reduce}) below. On the other hand, in 
an application of this estimate we compute slops, and that is
most convenient to do in an arclength parameterization.

In what follows, we consider a (bounded) Ptolemy circle
$\si\sub X_\om$
and points 
$x$, $y\in\si$
with
$a:=|xy|>0$.

\begin{lem}\label{lem:dist_parameter} Let 
$\si_+$, $\si_-$
be the two components of
$\si\sm\{x,y\}$.
Then for all
$0<t<a$ 
there exists exactly one point
$x_t^+\in\si_+$
(resp. 
$x_t^-\in\si_-$) 
with
$|xx_t^+|=|xx_t^-|=t$.
Therefore
$\ga:(-a,a)\to\si$
with
$\ga(0)=x$, $\ga(t)=x_t^+$
for 
$t>0$,
and
$\ga(t)=x_{-t}^-$
for 
$t<0$
parameterizes a neighborhood of
$x$ 
in 
$\si$.
\end{lem}

\begin{proof} The existence of a point
$x_t^+ \in\si_+$
with
$|xx_t^+|=t$
is clear by continuity. To prove uniqueness consider points
$x<p<q<y$ 
in this order on
$\si_+$
and assume
$b:=|xp|<a=|xy|$.
Let
$c:=|xq|$, $\la_a:=|pq|$, $\la_b:=|qy|$, $\la_c:=|py|$.

The Ptolemy equality and the triangle inequality give
$$a\la_a+b\la_b=c\la_c\le c(\la_a+\la_b).$$
Therefore
$$c \ge \frac{\la_a}{\la_a+\la_b}a+\frac{\la_b}{\la_a+\la_b}b >b$$
where the last equality holds, since 
$a>b$.
In particular
$|xp|\neq |xq|$.
\end{proof}

In what follows, we use the parameterization
$\ga:(-a,a)\to\si$
of a neighborhood of
$x\in\si$,
and call it a {\em distance} parameterization.

\begin{lem}\label{lem:smooth_concave} The function
$g(t):=|\ga(t)y|$
is concave and $C^1$-smooth on
$(-a,a)$.
\end{lem}

\begin{proof} For
$-a\le t_1<t_2\le a$
the Ptolemy equality for the points
$x$, $\ga(t_1)$, $\ga(t_2)$, $y$
implies
$$t_2g(t_1)-t_1g(t_2)=a|\ga(t_1)\ga(t_2)|.$$
Thus for 
$-a\le t_1<t_2<t_3\le a$
the triangle inequality 
$|\ga(t_1)\ga(t_3)|\le|\ga(t_1)\ga(t_2)|+|\ga(t_2)\ga(t_3)|$
implies
$$t_3g(t_2)-t_2g(t_3)+t_2g(t_1)-t_1g(t_2)\,\ge\,t_3g(t_1)-t_1g(t_3)$$
which is equivalent to
$$\frac{g(t_2)-g(t_1)}{t_2-t_1}\,\ge\, \frac{g(t_3)-g(t_2)}{t_3-t_2}.$$
Therefore,
$g$
is concave. It follows, in particular, that
$g$
has the left 
$g_-'(t)$
and the right derivative 
$g_+'(t)$
at every
$t\in(-a,a)$, $g_-'(t)\ge g_+'(t)$
and these derivatives are nonincreasing,
$g_+'(t)\ge g_-'(t')$
for
$t<t'$.
Furthermore,
$g_-'(t)\to g_-'(t')$
as
$t\nearrow t'$,
$g_+'(t')\to g_+'(t)$
as
$t'\searrow t$.
These are standard well known facts about concave functions, 
see e.g. \cite{H-UL}.

We fix
$t_0\in(-a,a)$
and consider the Ptolemy line 
$l\sub X_\om$
tangent to
$\si$
at
$x_0=\ga(t_0)$.
We assume that
$l$
is oriented and that its orientation is compatible with
the orientation of
$\si$
given by the distance parameterization
$\ga$.
Let
$c:\R\to X_\om$
be the unit speed parameterization of
$l$
compatible with the orientation,
$c(0)=x_0$.
By Corollary~\ref{cor:tangent_unique},
$y\not\in l$. 
By Proposition~\ref{pro:smooth_convex}, the function
$\wt g(s)=|c(s)y|$, $s\in\R$,
is
$C^1$-smooth.
If
$t_0=0$,
then
$g_-'(t_0)=\frac{d\wt g}{ds}(0)=g_+'(t_0)$,
because
$l$
is tangent to 
$\si$
at
$x_0=x$,
and thus
$g$
is differentiable at
$t_0=0$.

Consider now the case
$t_0\neq 0$.
Then again by Corollary~\ref{cor:tangent_unique},
$x\not\in l$,
thus the function
$\wt f(s)=|c(s)x|$, $s\in\R$,
is $C^1$-smooth.
We show that
$\frac{d\wt f}{ds}(0)\neq 0$.
We suppose W.L.G. that
$t_0>0$.
Then for all
$t_1\in(0,t_0)$
sufficiently close to
$t_0$
we have
$\frac{d\wt h}{ds}(0)\neq 0$,
where
$\wt h(s)=|x_1c(s)|$, $x_1=\ga(t_1)$.
We fix such a point 
$t_1$, 
and using Lemma~\ref{lem:dist_parameter}
consider the distance parameterization of a neighborhood of
$x_0=\ga(t_0)$
in
$\si$, $|z(\tau)x_1|=\tau$
for all
$z\in\si$
sufficiently close to
$x_0$.
Then 
$t=t(\tau)$
and the function
$f(\tau)=|x\ga\circ t(\tau)|$
is concave by the first part of the proof. 
Since the functions
$\wt f(s)=|xc(s)|$, $\wt h(s)=|x_1c(s)|$
are $C^1$-smooth, and
$\frac{d\wt h}{ds}(0)\neq 0$,
the function 
$\wt f(\tau)=\wt f\circ\wt h^{-1}(\tau)$
is $C^1$-smooth
in a neighborhood of
$\tau_0=|x_1x_0|$
by the inverse function theorem. Therefore,
$f_-'(\tau_0)=\frac{d\wt f}{ds}(0)=f_+'(\tau_0)$
because
$l$
is tangent to
$\si$
at
$x_0$.
The assumption
$\frac{d\wt f}{ds}(0)=0$
implies
$\frac{df}{d\tau}(\tau_0)=0$.
By concavity, 
$\tau_0$
is a maximum point of the function
$f(\tau)$, 
and there are different
$\tau$, $\tau'$
arbitrarily close to
$\tau_0$
with
$f(\tau)=f(\tau')$.
This contradicts properties of the parameterization
$\ga(\tau)=\ga\circ t(\tau)$.
Hence,
$\frac{d\wt f}{ds}(0)\neq 0$.

Again, by the inverse function theorem, the function
$\wt g(t)=\wt g\circ\wt f^{-1}(t)$
is
$C^1$-smooth 
in a neighborhood of
$t_0$.
However,
$\frac{d\wt g}{dt}(t_0)$
coincides with the left as well as with the right
derivative of the function
$g$
at
$t_0$
because
$l$
is tangent to
$\si$
at
$x_0$.
Therefore,
$g$
is differentiable at
$t_0$.
It follows from continuity properties of
one-sided derivatives of concave functions that the derivative
$g'$
is continuous, i.e.,
$g$
is 
$C^1$-smooth.
\end{proof}

\begin{lem}\label{lem:circle_rectifiable} Every Ptolemy circle
$\si\sub X_\om$
is rectifiable and
$$L(xx')=|xx'|+o(|xx'|^2)$$
as
$x'\to x$
in
$\si$,
where
$L(xx')$
is the length of the (smallest) arc
$xx'\sub\si$.
\end{lem}

\begin{proof} We fix 
$y\in\si$, $y\neq x$,
and introduce a distance parameterization
$\ga:(-a,a)\to\si$
of a neighborhood of
$x=\ga(0)$
in
$\si$.
Rescaling the metric of
$X_\om$
we assume that
$a=|xy|=1$
for simplicity of computations. We use notation
$d(z,z')=|zz'|$
for the distance in
$X_\om$,
and
$|zz'|_y$
for the distance in
$X_y$,
assuming that
$$|zz'|_y=\frac{|zz'|}{|zy||z'y|}$$
is the metric inversion of the metric 
$d$.

Recall that 
$\si\sm y$
is a Ptolemy line in
$X_y$.
Thus for a given
$r\in(0,1)$,
and for every partition
$0=t_0\le\dots\le t_n=r$
we have
$$\sum_i|\ga(t_i)\ga(t_{i+1})|\le\La|xx_r|_y,$$
where
$\La=\max\set{g^2(t)}{$0\le t\le r$}$, $g(t)=|\ga(t)y|$,
$x_r=\ga(r)$.
Hence
$\si$
is rectifiable and
$L(xx_r)\le\La|xx_r|_y$.
Moreover, using
$|\ga(t_i)\ga(t_{i+1})|=g(t_i)g(t_{i+1})|\ga(t_i)\ga(t_{i+1})|_y$,
we actually have
$$L(r)=L(xx_r)=\int_0^{r/g(r)}g^2(s)ds,$$
where
$g(s)=g\circ t(s)$
with
$s=\frac{t}{g(0)g(t)}=t/g(t)$.
Recall that the function
$g(t)$
is $C^1$-smooth by Lemma~\ref{lem:smooth_concave}. Then
$ds=\frac{g(t)-tg'(t)}{g^2(t)}dt$
and
$\frac{dt}{ds}=\frac{g^2(t)}{g(t)-tg'(t)}$,
in particular,
$\frac{dt}{ds}(0)=1$.

Using developments
$g(s)=1+\frac{dg}{ds}(0)s+o(s)$,
$g^2(s)=1+2\frac{dg}{ds}(0)s+o(s)$,
we obtain
$$L(r)=\frac{r}{g(r)}+\frac{dg}{ds}(0)\frac{r^2}{g^2(r)}+o(r^2),$$
where
$\frac{dg}{ds}(0)=g'(0)\frac{dt}{ds}(0)=g'(0)$.
Since
$g(r)=1+g'(0)r+o(r)$, $g^2(r)=1+2g'(0)r+o(r)$,
we finally have
$$L(r)=r(1-g'(0)r)+g'(0)r^2(1-2g'(0)r)+o(r^2)=r+o(r^2).$$
Hence
$L(xx')=|xx'|+o(|xx'|^2)$
as
$x'\to x$
in
$\si$. 
\end{proof}

\begin{lem}\label{lem:quaratic_excess} Assume that a (bounded) oriented
Ptolemy circle
$\si\sub X_\om$
has two different points in common with a Ptolemy line 
$l\sub X_\om$, $x$, $y\in\si\cap l$,
and the line 
$l$
is oriented from
$y$
to
$x$.
Let
$\si_+\sub\si$
be the arc of
$\si$
from
$x$
to
$y$
chosen according to the orientation of
$\si$.
Let
$x_t$, $y_t$
be the distance parameterizations of neighborhoods of
$x$, $y$
respectively
such that
$x_t$, $y_t\in\si_+$
for 
$t>0$, $|x_tx|=t=|y_ty|$.
Furthermore, let
$b^\pm:X_\om\to\R$
be the opposite Busemann functions of
$l$
normalized by
$b^+(x)=0$, $b^+(y)=-a$, $b^-(x)=-a$, $b^-(y)=0$,
where
$a=|xy|$.
Then
\begin{equation}\label{eq:quadratic_reduce}
b^+(x_t)+b^-(y_t)\le 2\al t-\frac{1}{a}(1-\al^2)t^2
\end{equation}
for all 
$t>0$
in the domain of the parameterizations, where
$\al=\slope(\si,l)$. 
\end{lem}

\begin{proof} By Lemma~\ref{lem:smooth_concave} the functions
$g(t)=|x_ty|$, $f(t)=|y_tx|$
are $C^1$-smooth and concave. Furthermore, their first derivatives
at 0,
$g'(0)$
and
$f'(0)$,
coincide with first derivatives of the distance functions to the 
respective tangent lines,
$g'(0)=\wt g'(0)$
and
$f'(0)=\wt f'(0)$,
where
$\wt g(s)=|c_x(s)y|$, $\wt f(s)=|c_y(s)x|$,
and the unit speed parameterizations of the tangent lines
$l_x$
to
$\si$
at
$x$
and
$l_y$
to
$\si$
at 
$y$
are chosen compatible with the distance parameterizations
$x_t$, $y_t$
so that
$c_x(0)=x$, $c_y(0)=y$.

Using that
$l$
is oriented from
$y$
to
$x$,
and
$\si_+$
from
$x$
to
$y$
and applying equation~(\ref{eq:re_first_variation}), we find
$\wt g'(0)=\slope(l_x,l)$.
By the same equation~(\ref{eq:re_first_variation}) we have
$\wt f'(0)=\slope(l_y,-l)=-\slope(l_y,l)$.
The sign
$-1$
appears because the orientation of
$l$
from
$x$
to
$y$
is opposite to the chosen orientation.
Note that the orientation of
$l_x$
is compatible with that of
$\si$,
while the orientation of
$l_y$
is opposite to that of 
$\si$.
Therefore,
$g'(0)=\al=\slope(\si,l)$
and
$f'(0)=-\slope(l_y,l)=\slope(\si,l)=\al$.

Using concavity we obtain
$g(t)\le g(0)+g'(0)t=a+\al t$
and similarly
$f(t)\le a+\al t$
for all 
$0\le t<a$.
The Ptolemy equality applied to the ordered quadruple
$(x,x_t,y_t,y)\sub\si$
gives
$g(t)f(t)=t^2+a|x_ty_t|$,
hence
$$|x_ty_t|\le a+2\al t-\frac{1}{a}(1-\al^2)t^2.$$

Let
$H_t^+$, $H_t^-$
be the horospheres of
$b^+$, $b^-$
through
$x_t$, $y_t$
respectively,
$x_t\in H_t^+$, $y_t\in H_t^-$.
Since
$X$
is Busemann flat, see Proposition~\ref{pro:smooth_convex},
$H_t^+$
is also a horosphere of
$b^-$
and
$H_t^-$
is a horosphere of
$b^+$.
Thus
$|x_ty_t|\ge|b^+(x_t)-b^+(y_t)|:=\xi$
because Busemann functions are 1-Lipschitz. On the other hand,
$\xi$
is the distance between the points
$l\cap H_t^+$, $l\cap H_t^-$,
and thus
$\xi=a+b^+(x_t)+b^-(y_t)$.
Therefore,
$$b^+(x_t)+b^-(y_t)\le 2\al t-\frac{1}{a}(1-\al^2)t^2.$$
\end{proof}

\subsection{Proof of the extension property (${\rm E}_2$)}
\label{subsect:extension_property}

Here we prove Proposition~\ref{pro:extension_property}.

\begin{lem}\label{lem:extend_cirle_map} Any M\"obius
automorphism of any Ptolemy circle
$\si\sub X$
preserving orientations extends to a M\"obius automorphism of
$X$. 
\end{lem}

\begin{proof} We represent
$\si$
as the boundary at infinity of the real hyperbolic plane,
$\si=\di\hyp^2$,
so that the M\"obius structure of
$\si$
induced from
$X$
is identified with the canonical M\"obius structure of
$\di\hyp^2$. 
Then the group 
$\aut_\si X$
of preserving orientations M\"obius automorphisms of
$\si$
is identified with the group of preserving orientations
isometries of
$\hyp^2$.
The last is generated by central symmetries, and
any central symmetry of
$\hyp^2$
induces an s-inversion of
$\si$.
Thus
$\aut_\si X$
is generated by s-inversions of
$\si$.

Now, any s-inversion of
$\si$
can be obtained as follows. Take distinct
$\om$, $\om'\in\si$
and a metric sphere
$S\sub X$
between
$\om$, $\om'$.
Then an s-inversion
$\phi=\phi_{\om,\om',S}:X\to X$
restricts to an s-inversion of
$\si$.
Thus any M\"obius automorphism of
$\si$
from
$\aut_\si X$
extends to a M\"obius automorphism of
$X$.
\end{proof}

The group
$\aut X$
of M\"obius automorphisms of
$X$
is non-compact: a sequence of homotheties of
$X_\om$
with coefficients
$\la_i\to\infty$
and with the same fixed point has no converging subsequences.
However, we have the following standard compactness result.

\begin{lem}\label{lem:nondegenerate_morphism} Assume that
for a nondegenerate triple
$T=(x,y,z)\sub X$
and for a sequence 
$\phi_i\in\aut X$
the sequence
$T_i=\phi_i(T)$
converges to a nondegenerate triple
$T'=(x',y',z')\sub X$.
Then there exists
$\phi\in\aut X$
with
$\phi(T)=T'$.
\end{lem}

\begin{proof} For every
$u\in X\sm T$
the quadruple
$Q=(T,u)$
is nondegenerate in the sense that its cross-ratio triple
$\crt(Q)=(a:b:c)$
has no zero entry. Since
$\crt(\phi_i(Q))=\crt(Q)$,
any accumulation point 
$u'$
of the sequence
$u_i=\phi_i(u)$
is not in
$T'$.
Thus for the nondegenerate triple
$S=(x,y,u)$
any sublimit 
$S'=(x',y',u')$
of the sequence
$S_i=\phi(S)$
is nondegenerate. Applying the same argument to any
$v\in X\sm S$,
we observe that the sequences
$u_i$, $v_i=\phi_i(v)$
have no common accumulation point. This shows that
any limiting map
$\phi$
of the sequence
$\phi_i$,
obtained e.g. by taking a nonprincipal ultra-filter limit,
is injective, and hence it is a M\"obius automorphism of
$X$
with
$\phi(T)=T'$. 
\end{proof}

\begin{pro}\label{pro:transitive_circle} Assume
$X\neq\wh\R$.
Then the group of M\"obius automorphisms of
$X$
acts transitively on the set of the oriented Ptolemy circles in
$X$. 
In particular, for any oriented circle
$\si\sub X$
there is a M\"obius automorphism
$\phi:X\to X$
such that
$\phi(\si)=\si$
and 
$\phi$
reverses the orientation of
$\si$.
\end{pro}

\begin{proof} We fix an oriented Ptolemy circle
$\si\sub X$
and distinct points
$x$, $y\in\si$.
For an oriented circle
$\si_0\sub X$
we denote by
$A$
the set of all the circles
$\phi(\si_0)$, $\phi\in\aut X$,
with the induced orientation which pass also through
$x$
and
$y$.
Let
$$\al=\inf\set{\slope(\si,\si')}{$\si'\in A$}.$$
By two-point homogeneity property, see Proposition~\ref{pro:two_point_homogeneous},
$A\neq\es$.
Applying Lemma~\ref{lem:nondegenerate_morphism} we find
$\si'\in A$
with
$\slope(\si,\si')=\al$.
Next, we show that
$\al<1$.
Since
$X\neq\wh\R$,
there is a shift which makes
$\si_0$
disjoint with 
$\si$.
Taking a point 
$\om\in\si_0$
as infinitely remote, we consider all Ptolemy lines in
$X_\om$
which are Busemann parallel to
$\si_0\sm\om$
and intersect
$\si$.
Since 
$\si$
is bounded in
$X_\om$,
at least one of them,
$l$,
is not tangent to
$\si$.
Then
$\slope(\si,l)<1$.
This 
$l$
can be obtained from
$\si_0\sm\om$
by a shift. Applying another shift to
$l$
in the space 
$X_{\om'}$
with
$\om'\in\si\cap l$
(this does not change the slope), we can assume that
$x\in\si\cap l$.
Repeating this in the space 
$X_x$,
we find
$\wt\si\in A$
with 
$\slope(\si,\wt\si)<1$.
Thus
$\al<1$.

We show that
$\al=-1$.
Then 
$\si=\phi(\si_0)$
as oriented Ptolemy circles for some
$\phi\in\aut X$,
which would complete the proof.

Assume that
$\al>-1$.
The points
$x$, $y$
subdivide each of the circles
$\si$, $\si'$
into two arcs. We choose an arc
$\si_+\sub\si$
leading from
$x$
to
$y$
according to the orientation of
$\si$, 
and an arc
$\si_+'\sub\si'$
leading from
$y$
to
$x$
according to the orientation of
$\si'$.
Taking a point 
$\om\in\si'$
inside of the opposite to
$\si_+'$
arc, we see that
$l=\si'\sm\om$
is a Ptolemy line in the space
$X_\om$
oriented from
$y$
to
$x$.

Given
$x'\in\si_+$,
for every Ptolemy line 
$l_{x'}\sub X_\om$
through
$x'$,
which is Busemann parallel to
$l$
and is oriented as
$l$,
we have
$\slope(\si,l_{x'})\ge\al$
by the definition of
$\al$,
because by the same argument as above
$l_{x'}$
can be put in the set 
$A$
without changing the slope.

Let
$b^\pm:X_\om\to\R$
be the opposite Busemann functions of
$l$
normalized by
$b^+(x)=0$, $b^+(y)=-a$, $b^-(x)=-a$, $b^-(y)=0$,
where
$a=|xy|$.
Using Lemma~\ref{lem:circle_rectifiable} we consider
for a sufficiently small
$\ep>0$
arclength parameterizations
$c_x$, $c_y:(-\ep,\ep)\to\si$
with
$c_x(0)=x$, $c_y(0)=y$, $c_x(s)$, $c_y(s)\in\si_+$
for
$s>0$,
of neighborhoods of
$x$, $y$
respectively in
$\si$.
Since Busemann functions on
$X_\om$
are affine and hence differentiable along Ptolemy lines,
and since the derivative
$\frac{db^+\circ c_x}{ds}(s)$
coincides with the derivative of
$b^+$
along the tangent line to
$\si$
at
$c_x(s)$,
we have
$$\frac{db^+\circ c_x}{ds}(s)=\slope(\si,l_{c_x(s)})\ge\al.$$
For a sufficiently small
$t>0$
let
$x_t\in\si_+$
be a point at the distance
$t$
from
$x$, $|x_tx|=t$, $x_t=c_x(\tau)$
for some
$\tau=\tau(t)$.
By integrating we obtain 
$b^+(x_t)\ge\al L(xx_t)\ge\al t$.
A similar argument shows that
$b^-(y_t)\ge\al t$,
where
$y_t=c_y(\tau')$
for some
$\tau'=\tau'(t)$, $|yy_t|=t$.
Therefore,
$b^+(x_t)+b^-(y_t)\ge 2\al t$.
This contradicts the estimate~(\ref{eq:quadratic_reduce})
of Lemma~\ref{lem:quaratic_excess}. Thus
$\al=-1$.
\end{proof}

\begin{proof}[Proof of Proposition~\ref{pro:extension_property}]
Given a M\"obius map
$\psi:\si\to\si'$
between Ptolemy circles
$\si$, $\si'\sub X$,
we choose orientations of
$\si$, $\si'$
so that
$\psi$
preserves the orientations. By Proposition~\ref{pro:transitive_circle}
there is
$\phi\in\aut X$
with 
$\phi(\si)=\si'$
preserving the orientations. Then
$\phi^{-1}\circ\psi:\si\to\si$
preserves the orientation of
$\si$,
and hence it extends by Lemma~\ref{lem:extend_cirle_map} to
$\phi'\in\aut X$,
$\phi'|\si=\phi^{-1}\circ\psi$.
Then
$\phi\circ\phi'\in\aut X$
is a required M\"obius automorphism.
\end{proof}

\section{Filling of $X$}
\label{sect:filling}

We assume that a M\"obius space
$X$
satisfies the assumptions of Theorem~\ref{thm:moebius},
and furthermore that
$\dim X\ge 2$,
since in the case
$\dim X=1$
the space
$X$
is M\"obius equivalent to
$\wh\R=\di\hyp^2$.

\subsection{Basic facts}
\label{subsect:basic_facts}

We introduce the notion of a strict space inversion, 
or an ss-inversion for brevity. A space inversion
$\phi_a:X\to X$, $a=(\om,\om',S)$,
is said to be {\em strict,} if for any
$x$, $x'\in X$
with
$\phi_a(x)=x'$
there exists a sphere
$T$
between
$x$, $x'$
such that
$\phi_a=\phi_b$,
where
$b=(x,x',T)$.

The following lemma is a modification of 
Lemma~\ref{lem:sinversion_minversion}.

\begin{lem}\label{lem:invariant_sphere} Let
$\phi:X\to X$
be a M\"obius involution,
$\phi^2=\id$,
of a Ptolemy space 
$X$
with
$\phi(\om)=\om'$
for distinct
$\om$, $\om'\in X$.
Then there is a unique sphere
$S\sub X$
between
$\om$, $\om'$
invariant for
$\phi$, $\phi(S)=S$.
\end{lem}

\begin{proof} Let
$d$
be a metric of the M\"obius structure with infinitely
remote point 
$\om'$.
Since
$\phi(\om)=\om'$,
the point 
$\om$
is infinitely remote for the induced metric
$\phi^\ast d$.
Thus for some
$\la>0$
we have
$$(\phi^\ast d)(x,y)=\frac{\la d(x,y)}{d(x,\om)d(y,\om)}$$
for each
$x$, $y\in X$
which are not equal to
$\om$
simultaneously. We let
$S=S_r^d(\om)\sub X$
be a metric sphere between
$\om$, $\om'$
with
$r^2=\la$, $d(x,\om)=r$
for every
$x\in S$.
Then
$$d(\phi(x),\om)=d(\phi(x),\phi(\om'))=(\phi^\ast d)(x,\om')
  =\la/d(x,\om)=r$$
for every
$x\in S$.
Hence
$\phi(S)=S$.
For any
$x\in X$
with 
$d(x,\om)\lessgtr r$
the same argument shows that
$d(\phi(x),\om)\gtrless r$,
thus an invariant sphere between
$\om$, $\om'$
is unique.
\end{proof}

\begin{pro}\label{pro:sinversion_strict} Every s-inversion
$\phi:X\to X$
is strict.
\end{pro}

\begin{proof} Let 
$\phi=\phi_a$
for 
$a=(\om,\om',S)$
and assume that
$\phi_a(x)=x'$
for some
$x$, $x'\in X$.
By Lemma~\ref{lem:invariant_sphere}, there is a unique sphere
$T$
between
$x$, $x'$
with
$\phi_a(T)=T$.
We put 
$b=(x,x',T)$
and consider the s-inversion
$\phi_b:X\to X$.
Then the s-inversions 
$\phi_a$, $\phi_b$
both permute
$x$, $x'$
and preserve the sphere
$T$.
In addition,
$\phi_b$
preserves every Ptolemy circle through
$x$, $x'$.
We show that
$\phi_a$
also preserves every Ptolemy circle through
$x$, $x'$.
This would imply
$\phi_b=\phi_a$
by uniqueness from the property (I), i.e. that
$\phi$
is strict.

Let 
$D=D_{x'}\sub X_x$
be the fiber through
$x'$
of the canonical distribution on
$X_x$,
that is,
$D$
consists of all Ptolemy lines in
$X_x$
through
$x'$.
Since
$\phi_b$
preserves every Ptolemy circle through
$x$, $x'$,
the intersection
$T\cap D$
is
$\phi_b$-invariant,
$\phi_b(T\cap D)=T\cap D$,
and moreover
$y$, $x'$, $\phi_b(y)$
lie on a Ptolemy line through
$x'$
in this order for every
$y\in T\cap D$. 
Thus
$\phi_b|T\cap D$
induces the antipodal involution
$\ov\phi_b:S_r(o)\to S_r(o)$
of the sphere
$S_r(o)\sub B_x$
of radius
$r$
centered at
$o=\pi_x(x')$
in the Euclidean space
$B_x$,
where
$r$
is the radius of the sphere
$T$
in the metric of
$X_x$, $T=\set{y\in X_x}{$|yx'|=r$}$.
In particular,
$\ov\phi_b$
is an isometry of
$S_r(o)$.

On the other hand, the composition
$\mu=\phi_a\circ\phi_b:X\to X$
preserves
$x$, $x'$
and
$T$,
hence
$\mu:X_x\to X_x$
is an isometry. Thus it projects to an isometry
$\ov\mu:B_x\to B_x$
of the base. Since
$\phi_a=\mu\circ\phi_b$
and both
$\mu$, $\phi_b$
project to isometries of
$S_r(o)$,
we see that
$\phi_a|T\cap D$
projects to an isometry 
$\ov\phi_a=\ov\mu\circ\ov\phi_b$
of
$S_r(o)$.
Moreover,
$\ov\phi_a$
is an involution, 
$\ov\phi_a^2=\id$,
because
$\phi_a$
is. However, every isometric involution of
$S_r(o)$,
obviously, preserves every pair of antipodal points.
This means that
$\phi_a$
preserves every Ptolemy circle through
$x$, $x'$.
\end{proof}

The set
$Y$
of all the s-inversions is called
the {\em filling} of
$X$, $Y=\fil X$.
Any pair 
$\om$, $\om'\in X$
of distinct points determines a {\em line} in
$Y$:
a point
$t\in Y$
lies on a line
$(\om,\om')$
iff 
$t(\om)=\om'$.
In particular, two lines
$(\om,\om')$, $(x,x')$
in
$Y$
intersect if there is
$t\in Y$
such that
$t(\om)=\om'$, $t(x)=x'$.

The M\"obius group of
$X$
naturally acts on
$Y$
by conjugation: for every M\"obius automorphism
$g:X\to X$
we have
$$g^\ast(t)=g\circ t \circ g^{-1}$$
for the induced
$g^\ast:Y\to Y$.
In particular,
$t\in Y$
is fixed for 
$g^\ast$
iff
$g\circ t=t\circ g$.

The following lemma describes the upper half-space model of
$\fil X$.

\begin{lem}\label{lem:unique_represent} For a fixed
$\om\in X$
and every
$t\in\fil X$
there exist a uniquely determined sphere
$S$
between
$\om$, $t(\om)$
such that 
$t=\phi_b$
for 
$b=(\om,t(\om),S)$. 
Thus for a fixed metric on
$X_\om$,
the space
$Y=\fil X$
is canonically identified with
$Y=X_\om\times\R_+$
by
$t=(t(\om),r)$,
where 
$r>0$
is the radius of the sphere
$S$.
In particular, 
$Y$
is a smooth manifold diffeomorphic to 
$\R^{n+1}$, $n=\dim X$.
\end{lem}

\begin{proof} By Proposition~\ref{pro:sinversion_strict}, there is a uniquely
determined sphere
$S$
between
$\om$, $\om'=t(\om)$
with 
$\phi_b=t$,
where
$b=(\om,\om',S)$.
Thus
$Y=X_\om\times\R_+$.
By Corollary~\ref{cor:nilponent_lie_group}, the group
$N_\om$
acting on
$X_\om$
simply transitively is a Lie group diffeomorphic to
$\R^n$. 
Therefore, 
$Y$
is a smooth manifold diffeomorphic to
$\R^{n+1}$. 
\end{proof}

\begin{lem}\label{lem:sphere_intersection} Assume a space inversion
$s\in\fil X$
has  representations
$s=\phi_a=\phi_b$
for 
$a=(\om,\om',S)$, $b=(\xi,\xi',T)$,
where 
$S$, $T\sub X$
are spheres between
$\om$, $\om'$
and
$\xi$, $\xi'$
respectively.
Then the intersection
$S\cap T$
is not empty,
$S\cap T\neq\es$.
\end{lem}

\begin{proof} The intersection
$S\cap T$
is not empty (and invariant under
$s$)
because the spheres
$S$, $T$
are invariant under
$s$,
connected (see Corollary~\ref{cor:nilponent_lie_group}), and
$s$
has no fixed point in
$X$,
thus it permutes the components of
$X\sm S$
(as well as those of 
$X\sm T$).
\end{proof}

\subsection{Lines in the filling}
\label{subsect:lines_filling}

We let
$Y=\fil X$.

\begin{lem}\label{lem:line_R} Every line
$\ga=(a,a')\sub Y$
is homeomorphic to
$\R$.
\end{lem}

\begin{proof} By definition,
$\ga=\set{t\in Y}{$t(a)=a'$}$.
In the upper half-space model
$Y=X_a\times\R_+$,
we have
$\ga=\{a'\}\times\R_+\simeq\R$
because every sphere 
$S\sub X_\om$
between
$a$
and
$a'$
is centered at
$a'$.
\end{proof}

\begin{pro}\label{pro:unique_line_filling} Through any two distinct
point
$s$, $t\in Y$
there is a unique line
$\ga=(a,a')\sub Y$.
\end{pro}

For a fixed
$t\in Y$
we define a map 
$h:Y\to\aut X$
by
$s\mapsto h_s=s\circ t:X\to X$.
Note that
$h_t=\id$.
We put
$Y_t=Y\sm t$
and 
$$A=\set{s\in Y_t}{$s\ \text{and}\ t\ \text{have a common pair
of antipodal points in}\ X$}.$$
An equivalent definition is
$$A=\set{s\in Y_t}{$h_s\ \text{has a fixed point in}\ X$}.$$

\begin{lem}\label{lem:A_nonempty} The set
$Y_t$
is connected, and
$A\sub Y_t$
is nonempty and closed in
$Y_t$. 
\end{lem}

\begin{proof} We fix 
$\om\in X$
and a metric of the M\"obius structure on
$X_\om$.
Then by Lemma~\ref{lem:unique_represent}, we have an upper half-space
model 
$Y=X_\om\times\R_+$.
By Corollary~\ref{cor:nilponent_lie_group},
$X_\om$
is homeomorphic to
$\R^n$
with
$n\ge 1$.
Thus
$Y_t$
is connected.

We have
$t=(a,R)\in X_\om\times\R_+$,
where
$R$
is the radius of the 
$t$-invariant 
sphere
$S_t\sub X_\om$
between
$a$
and
$\om$, $S_t=\set{x\in X_\om}{$|xa|=R$}$.
Then for every
$s=(a,r)\in Y$,
$r\neq R$,
the composition
$h_s=s\circ t:X_\om\to X_\om$
is a nontrivial pure homothety with
$h_s(a)=a$,
see Proposition~\ref{pro:homothety_property}. Thus
$s\in A$,
and
$A$
is nonempty.

Assume
$\lim_is_i=s\in Y_t$
for a sequence
$s_i\in A$.
For every
$i$
there is
$a_i\in X$
with
$s_i\circ t(a_i)=a_i$.
Since
$X$
is compact, the sequence
$a_i$
subconverges to
$a\in X$.
We put 
$\om:=t(a)$.
Then
$t=(a,R)$
for some 
$R>0$
in the upper half-space model
$Y=X_\om\times\R_+$.
We also have
$s=(b,r)\in X_\om\times\R_+$.
To prove that
$s\in A$
it suffices to show that
$b=a$.
We have
$s_i=(b_i,r_i)$
with 
$b_i\to b$, $r_i\to r$.
Note that
$s_i(a_i)=t(a_i)\to\om$,
i.e. the sequence
$s_i(a_i)$
tends to infinity in the space
$X_\om$.
But if
$b\neq a$,
the sequence
$b_i$
is separated from
$a$,
and since
$r_i<2r$,
the sequence
$s_i(a_i)$
is bounded in
$X_\om$, $|s_i(a_i)b|\le4r^2/|ab|$
for all sufficiently large 
$i$.
This contradiction shows that
$s\in A$
and thus
$A$
is closed in
$Y_t$.
\end{proof}

\begin{lem}\label{lem:line_fill} $A=Y_t$.
\end{lem}

\begin{proof} In view of Lemma~\ref{lem:A_nonempty},
it suffices to check that
$A$
is open in
$Y_t$.
If not, there is sequence
$C\ni s_i\to s\in A$, $C=Y_t\sm A$.
Then
$s$, $t$
have a common pair
$o$, $\om$
of antipodal points in
$X$, $s(o)=t(o)=\om$,
and
$h_s=s\circ t:X_\om\to X_\om$
is a pure homothety with coefficient
$\la\neq 1$
centered at
$o$, $h_s(o)=o$.
We can assume W.L.G. that
$0<\la<1$.
By definition,
$h_i=h_{s_i}:X\to X$
has no fixed point in
$X$
for every
$i$.

Let 
$B=B_r(o)\sub X_\om$
be a metric ball centered at
$o$, $B_r(o)=\set{x\in X_\om}{$|xo|\le r$}$
for some
$r>0$.
Then
$h_s(B)\sub\Int B$.
Since
$h_i\to h_s$
as
$i\to\infty$
in the M\"obius group of
$X$,
one easily checks that
$h_i(B)\sub B$
for every sufficiently large
$i$.
By Corollary~\ref{cor:nilponent_lie_group},
$B$
is homeomorphic to the standard ball in
$\R^n$.
Thus
$h_i$
has a fixed point in
$B$
for all sufficiently large
$i$.
This contradicts the assumption
$s_i\in C$.
Hence
$A=Y_t$.
\end{proof}

\begin{proof}[Proof of Proposition~\ref{pro:unique_line_filling}]
By Lemma~\ref{lem:line_fill}, the space inversions
$s$, $t:X\to X$
have a common pair
$(a,a')$
of antipodal points,
$s(a)=a'=t(a)$.
It means that the line
$\ga=(a,a')\sub Y$
passes through
$s$, $t$. 

Assume there is another line
$\ga'=(b,b')$
through
$s$, $t$.
Then 
$a$, $a'$, $b$, $b'$
are fixed points of the composition
$h=s\circ t$.
In particular,
$h$
acts on
$X_a$
as a homothety with coefficient
$\la\neq 1$. 
But any such homothety of
$X_a$
has a unique fixed point. Thus we have W.L.G.
$b=a$, $b'=a'$,
that is,
$\ga=\ga'$.
\end{proof}

\subsection{Distance in the filling}
\label{subsect:dist_filling}

Let 
$Y=\fil X$
be the filling of
$X$.
Given a line
$\ga=(\om,\om')\sub Y$
and points
$s$, $t\in\ga$,
let
$S$, $T\sub X$
be spheres between
$\om$, $\om'$
such that
$s=\phi_a$, $t=\phi_b$
for 
$a=(\om,\om',S)$, $b=(\om,\om',T)$.
We pick
$x\in S$, $y\in T$
and set 
$$\rho(s,t)=|\ln\langle\om,x,y,\om'\rangle|,$$
where
$\langle\om,x,y,\om'\rangle=\frac{|\om y|\cdot|x\om'|}{|\om x|\cdot|y\om'|}$
is the cross-ratio of the quadruple
$(\om,x,y,\om')$.
This definition of the distance
$\rho$
in
$Y$
is independent of the choice
$x\in S$, $y\in T$.
Indeed, let
$r>0$
be the radius of the sphere
$S$
(centered at
$\om'$)
w.r.t. a metric in
$X_\om$,
and
$R$
the radius of the sphere
$T$.
Then
$\langle\om,x,y,\om'\rangle=|x\om'|/|y\om'|=r/R$
and
$\rho(s,t)=|\ln(r/R)|$.
It  follows from Proposition~\ref{pro:unique_line_filling} that
$\rho(s,t)>0$
is well defined for each distinct
$s$, $t\in Y$,
and we set 
$\rho(s,s)=0$
for every
$s\in Y$.
Furthermore, the distance
$\rho$
is symmetric, 
$\rho(s,t)=\rho(t,s)$.
However, it is not at all obvious that it satisfies the 
triangle inequality. Nevertheless, the M\"obius group of
$X$
acts on
$Y$
via conjugation, see sect.~\ref{subsect:basic_facts}, by
$\rho$-isometries 
because M\"obius transformations preserve the cross-ratio.

\begin{lem}\label{lem:rho_geodesic} Every line
$\ga=(a,a')\sub Y$
is a 
$\rho$-geodesic,
$$\rho(s_1,s_3)=\rho(s_1,s_2)+\rho(s_2,s_3)$$
for any 
$s_1$, $s_2$, $s_3\in\ga$
in this order.
\end{lem}

\begin{proof} By Lemma~\ref{lem:line_R},
$\ga=\{a'\}\times\R_+$
in the upper half-space model
$Y=X_a\times\R_+$.
Thus
$s_i=(a',r_i)$,
and we assume W.L.G. that
$r_1\le r_2\le r_3$.
Then
$\rho(s_i,s_j)=\ln\frac{r_j}{r_i}$ 
for 
$i<j$.
Hence, the claim.
\end{proof}

\subsection{Embeddings of $\hyp^2$ into $\fil X$}
\label{subsect:hyp2_fill}

\begin{pro}\label{pro:hyp2_fill} For every Ptolemy circle
$\si\sub X$
the set 
$$Y_\si=\set{s\in Y=\fil X}{$s(\si)=\si$}$$
is
$\rho$-isometric
to the hyperbolic plane
$\hyp^2$.
\end{pro}

\begin{proof} We fix 
$\om\in\si$
and show that in the upper half-space model
$Y=X_\om\times\R_+$
the set 
$Y_\si$
coincides with
$\si_\om\times\R_+$,
where
$\si_\om=\si\sm\{\om\}\sub X_\om$
is a Ptolemy line. Indeed, every
$s\in Y$
of type
$s=(o,r)$
with
$o\in\si_\om$
preserves
$\si$
because
$s$
permutes
$o$, $\om\in\si$,
i.e.
$\si_\om\times\R_+\sub Y_\si$.
In the opposite direction, we have
$o=s(\om)\in\si_\om$,
thus
$s=(o,r)\in\si_\om\times\R_+$
for some
$r>0$.

Since
$s$, $t\in Y_\si$
preserve
$\si$,
we have
$o$, $\om\in\si$
for the line
$(o,\om)\sub Y$
through
$s$, $t$.
Then the 
$\rho$-distance
between
$s$, $t$
can be computed in the upper half-space model
$X_\om\times\R_+$
as
$\rho(s,t)=|\ln\frac{R}{r}|$,
where
$s=(o,r)$, $t=(o,R)$.
But we have exactly this formula in the upper half-space model for 
$\hyp^2$.
Therefore
$Y_\si=\si_\om\times\R_+$
is 
$\rho$-isometric
to
$\hyp^2$.
\end{proof}

\subsection{Semi-norm on the tangent bundle to the filling}
\label{subsect:finsler_seminorm}

Recall that
$Y=\fil X$
is a smooth manifold, see Lemma~\ref{lem:unique_represent},
and thus the tangent space
$T_sY$
to
$Y$
is well defined for every
$s\in Y$.
For every
$a\in X$
the line
$(a,a')\sub Y$, $a'=s(a)$,
passes through
$s$
by definition. Then the speed vector
$v=dh(\frac{d}{d\tau})$
at
$\tau=0$
of some smooth parameterization 
$h:\R\to Y$, $h(0)=s$,
of the line
$(a,a')$
is tangent to
$Y$
at
$s$, $v\in T_sY$.
This defines a map 
$f:X\to S_sY$
to the space 
$S_sY$
of oriented directions in
$T_sY$,
which is continuous (to see the continuity one should think of the line
$(a,a')\sub Y$
through
$s$
as the set of homotheties
$\Ga_{a,a'}\sub\aut X$
by composing every
$s'\in(a,a')$
with
$s$).
By Proposition~\ref{pro:unique_line_filling}, the set 
of directions
$A\sub S_sY$
tangent to the lines
$(a,s(a))$, $a\in X$,
is dense in
$S_xY$.
Then
$A=S_sY$
because
$X$
is compact and hence the set 
$A=f(X)$
is closed in
$S_sY$.

Thus any nonzero vector
$v\in T_sY$
is tangent to a line
$(a,a')$
in
$Y$
through
$s$,
where
$a=a(v)$, $a'=s(a)\in X$.
We assume that
$s=h(0)$
for some smooth parameterization 
$h:\R\to Y$
of the line
$(a,a')$
with
$v=dh(\frac{d}{d\tau})=dh/d\tau$.
Then the formula
$$\|v\|:=\lim_{\tau\to 0}\frac{1}{|\tau|}\rho(h(\tau),s)$$
defines a semi-norm on
$T_sY$,
because it follows from the definition that
$ds(v)=-v$
and 
$\|-v\|=\|ds(v)\|=\|v\|$.
The semi-norm
$\|v\|$
is called the 
$\rho$-{\em semi-norm}.

For example, in the upper half-space model
$Y=X_{a'}\times\R_+$,
we have
$s=(a,u)$
for some
$u>0$.
Consider a parameterization
$h=h_\al:\R\to Y$
of the line
$(a,a')$
with
$h(0)=s$
given by
$h(\tau)=(a,ue^{\al\tau})$
for some
$\al>0$.
Then
$$\|dh/d\tau\|
  =\lim_{\de\to 0}\frac{1}{|\de|}\rho(h(\tau+\de),h(\tau)),$$
and we have
$\rho(h(\tau),h(\tau+\de))=|\ln(R/r)|$
for 
$R=ue^{\al(\tau+\de)}$, $r=ue^{\al\tau}$.
Thus
$\|dh/d\tau\|=\al$
for every
$\tau\in\R$,
and 
$h$
is a constant
$\rho$-speed
parameterization. There is
$\al>0$
such that
$v$
is a speed vector of 
$h=h_\al$,
hence
$\|v\|=\al$. 

We identify the direction space
$S_sY$
with the 
$\rho$-unit
sphere,
$S_sY=\set{v\in T_sY}{$\|v\|=1$}$.

\begin{lem}\label{lem:transitive} For every
$s\in Y$
the stabilizer of
$s$
in the 
$\rho$-isometry
group of
$Y$
acts transitively on the 
$\rho$-unit
sphere
$S_sY\sub T_sY$.
\end{lem}

\begin{proof} Given
$v$, $v'\in S_sY$
we let
$(a,a')$, $(b,b')\sub Y$, $a$, $b\in X$, $a'=s(a)$, $b'=s(b)$,
be the lines through
$s$
with tangent vectors
$v$, $v'$
at
$s$
respectively. There are Ptolemy circles
$\si$, $\si'\sub X$
through
$a$, $a'$
and
$b$, $b'$
respectively. The subspaces
$Y_\si$, $Y_{\si'}\sub Y$
isometric to
$\hyp^2$
contain
$s$
and their tangent spaces
$H_\si$, $H_{\si'}\sub T_sY$
contain
$v$, $v'$
respectively. There is a (uniquely determined up to the reflection
in the line
$(a,a')\sub Y_\si)$ 
M\"obius automorphism
$\phi:\si\to\si'$
with
$\phi(a)=b$, $\phi(a')=b'$
such that its extension to
$Y_\si$
is an isometry
$\phi:Y_\si\to Y_{\si'}$
with
$\phi(s)=s$.
Then
$d\phi(v)=v'$.

By Proposition~\ref{pro:extension_property}, the space
$X$
possesses the extension property~(${\rm E}_2$), which implies
that
$\phi:\si\to\si'$
extends to a M\"obius
$\phi:X\to X$
and therefore to a
$\rho$-isometry
$\phi:Y\to Y$
with
$\phi(s)=s$.
Thus the stabilizer of
$s$
in the 
$\rho$-isometry
group of
$Y$
acts transitively on 
$S_sY$.
\end{proof}

\begin{pro}\label{pro:seminorm_euclid} For every
$s\in Y$,
the 
$\rho$-semi-norm on the tangent space
$T_sY$ 
is an Euclidean norm.
\end{pro}

\begin{proof} Let 
$B\sub T_sY$
be the 
$\rho$-unit
ball centered at
$0$.
We first show that
$B$
is convex and, hence, the 
$\rho$-semi-norm
is a norm. Since
$\|v\|>0$
for every nonzero
$v\in T_sY$,
$B$
is compact and thus there is a supporting affine hyperplane
$H\sub T_sY$,
i.e. 
$H$
touches
$B$
such that
$B$
is contained in one of half-spaces determined by
$H$.
By Lemma~\ref{lem:transitive}, every boundary points of
$B$
possesses this property and thus
$B$
coincides with the intersection of half-spaces, i.e.
$B$
is convex.

With the transitivity of the isometry action, it is well known that then
$B$
is an ellipsoid and hence the
$\rho$-norm is an Euclidean one. (We briefly sketch the argument. It is well known
that every isometry of
$T_sY$
is affine. Hence, the group
$K$
of linear automorphisms of
$T_sY$
that fix
$0$
and preserves
$B$
acts transitively on
$\d B$.
Fixing a background Euclidean metric on
$T_sY$,
we find the L\"ovner ellipsoid
$L$
inscribed in
$B$,
that is, the ellipsoid of maximal volume. This ellipsoid
is unique and invariant under 
$K$.
Thus
$B=L$).
\end{proof}

\begin{cor}\label{cor:filling_symmetric} The filling
$Y=\fil X$
is a rank one symmetric space of non-compact type w.r.t. the metric
$\rho$
such that
$\di Y=X$.
\end{cor}

\begin{proof} The Riemannian metric on
$Y$
obtained in Proposition~\ref{pro:seminorm_euclid} is smooth
because the isometry group of
$Y$,
which acts transitively on
$Y$,
is a Lie group. Therefore
$Y$
is a Riemannian symmetric space of non-compact type. 
We check that there is no Euclidean geodesic plane
$E\sub Y$, $\dim E=2$.
Assume that 
$Y$
contains such an
$E$.
By Proposition~\ref{pro:unique_line_filling}, for each distinct
$s$, $t\in E$
there exists a line
$\ga=(a,a')$
through
$s$, $t$.
Then
$\ga\sub E$
is invariant for 
$h=s\circ t$
and
$h(a)=a$, $h(a')=a'$.
For any 
$u\in E$
the line
$\ga'=u(\ga)\sub E$
is parallel to
$\ga$
because
$u$
acts on
$E$
as a central symmetry. However every line
$\ga'\sub E$
parallel to
$\ga$
is also invariant for 
$h$, $h(\ga')=\ga'$.
Thus
$\ga'=(a,a')$
because
$a$, $a'$
are only points in
$X$
fixed by
$h$.
By Lemma~\ref{lem:line_R}, the line
$(a,a')\sub Y$
is homeomorphic to
$\R$
and uniquely determined by the end points
$a$, $a'$. 
Hence
$\ga'=\ga$,
and
$u$
preserves
$\ga$, $u(\ga)=\ga$.
Then
$u\in\ga$
by definition. This is a contradiction because
$\dim E=2$
and
$E\neq\ga$.
It follows that the sectional curvatures of
$Y$
are negative. Then the Gromov boundary
$\di Y$
coincides with the geodesic boundary of
$Y$
which is
$X$
by definition of geodesic lines in
$Y$.
Thus
$\di Y=X$.
\end{proof}

\subsection{Horospheres in the filling}
\label{subsect:horosphere_distance}

In this section we show that the Riemannian distance
$\rho$
on the filling
$Y=\fil X$
induces the initial M\"obius structure of
$X$.

By Proposition~\ref{pro:sect_curvature}, the Riemannian metric
associated with the distance
$\rho$
has negative sectional curvatures. Thus every parabolic isometry of
$Y$
preserving
$\om\in X$
leaves invariant every horosphere in
$Y$
centered at 
$\om$.
For example, the group
$N_\om$
of 
$X_\om$-shifts
acts on $Y$
by parabolic isometries preserving
$\om$.

We fix 
$\om\in X$
and a metric
$|xy|=|xy|_\om$
on
$X_\om$.
Let
$Y=X_\om\times\R_+$
be the respective upper half-space model of the filling
$Y$.
Every space inversion 
$t\in Y$
is represented as
$t=(t(\om),r_t)\in X_\om\times\R_+$,
where
$r_t>0$
is the radius of the 
$t$-invariant
sphere between
$t(\om)$
and
$\om$.
Since the group
$N_\om$
acts on
$X_\om$
by isometries, it leaves invariant every set 
$H_r=X_\om\times\{r\}\sub Y$.

Let 
$b:Y\to\R$
be the Busemann function
associated with a geodesic ray 
$\ga_a=\{a\}\times[R,\infty)\sub X_\om\times\R_+$,
$b(s)=0$
for 
$s=(a,R)\in Y$.
Then
$b(t)=\pm\rho(s,t)=\ln\frac{R}{r_t}$
for every
$t=(a,r_t)\in Y$.
By the remark above, the horosphere
$b^{-1}(r)$
of 
$b$
is
$N_\om$-invariant,
thus
$b^{-1}(r)=H_r$
for every
$r\in\R$
because
$N_\om$
is transitive on
$X_\om$.
It follows that any function
$b:Y\to\R$
of type
$$b(t)=\ln\frac{R}{r_t}$$
for
$t=(t(\om),r_t)$, $R>0$,
is a Busemann function in the upper half-space model
$Y=X_\om\times\R_+$
centered at 
$\om$.

\begin{lem}\label{lem:level_equality} Assume that distinct lines
$(a,a')$, $(\om,\om')\sub Y$
intersect at
$s\in Y$, $s=(a,a')\cap(\om,\om')\in H_r$
for some
$r>0$.
Then
$$|x\om'|\cdot|s(x)\om'|=r^2$$
for every
$x\in X\sm\{\om,\om'\}$.
In particular,
$|a\om'|\cdot|a'\om'|=r^2$.
\end{lem}

\begin{proof} By Proposition~\ref{pro:sinversion_strict}, the space inversion 
$s$
can be represented as
$s=\phi_c$
for 
$c=(\om',\om,T)$,
where
$T$
is the sphere between
$\om'$, $\om$
of radius
$r$
in 
$X_\om$.
Since
$\om'=s(\om)$,
we obtain, using Lemma~\ref{lem:sinversion_minversion} applied to 
$s$,
that
$|s(x)\om'|=|s(x)s(\om)|=\frac{r^2}{|x\om'|}$
for every
$x\in X\sm\{\om,\om'\}$,
\end{proof}

\begin{lem}\label{lem:dist_end_points} Given line
$\ga=(a_0,a_1)\sub Y$
and its points
$s_0=(\om_0,r_0)$, $s_1=(\om_1,r_1)\in\ga$
with distinct
$\om_0$, $\om_1\in X_\om$
in the order
$a_0,s_0,s_1,a_1$,
we have 
\begin{equation}\label{eq:dist_center_estimate}
|a_i\om_i|<\frac{4r_i^2}{|\om_0\om_1|}
\end{equation}
for 
$i=0,1$
and
$\max\{r_0,r_1\}\le|\om_0\om_1|/4$.
\end{lem}

\begin{proof} The space inversion
$s_i$
lies on the lines
$(a_0,a_1)$
and
$(\om,\om_i)$
in
$Y$, $i=0,1$.
Thus by Lemma~\ref{lem:level_equality},
\begin{equation}\label{eq:product_r_2}
|a_0\om_i|\cdot|a_1\om_i|=r_i^2.
\end{equation}
We have
$|\om_0\om_1|\ge 4r$
for 
$r=\max\{r_0,r_1\}$,
thus
$\dist(B_{r_0}(\om_0),B_{r_1}(\om_1))\ge|\om_0\om_1|-(r_0+r_1)\ge 2r$
for the balls in
$X_\om$
of radii
$r_0$, $r_1$
centered at
$\om_0$, $\om_1$
respectively. Therefore, one of the points
$a_0$, $a_1$
lies in
$B_{r_0}(\om_0)$,
while the other in 
$B_{r_1}(\om_1)$.
We show
$|a_i\om_i|<r_i$
for 
$i=0,1$.

Indeed, let
$S_i$, $T_i\sub X$
be the spheres between
$a_0$, $a_1$,
respectively
$\om$, $\om_i$
through
$s_i$, $i=0,1$. 
The order
$a_0,s_0,s_1,a_1$
on
$l$
tells us that the cross-ratio
$$\langle a_0,x_0,x_1,a_1\rangle>1$$
for any 
$x_0\in S_0$, $x_1\in S_1$.

By Lemma~\ref{lem:sphere_intersection} there is
$x_i\in S_i\cap T_i$.
The assumption 
$|a_0\om_1|<r_1$
implies
$|a_1\om_0|<r_0$.
Then
$|a_0x_0|,\ |a_1x_1|>\dist(B_r(\om_0),B_r(\om_1))\ge 2r$
and
$|a_0x_1|<2r_1$, $|a_1x_0|<2r_0$.
Therefore
$$\langle a_0,x_0,x_1,a_1\rangle
  =\frac{|a_0x_1|\cdot|a_1x_0|}{|a_0x_0|\cdot|a_1x_1|}<\frac{4r_0r_1}{4r^2}\le 1,$$
a contradiction. That is,
$a_i\in B_{r_i}(\om_i)$
for 
$i=1,2$.

We have
$|a_0a_1|\ge|\om_0\om_1|-2r\ge 2r$.
By the triangle inequality,
$|a_i\om_j|\ge|a_0a_1|-|a_j\om_j|$.
Using (\ref{eq:product_r_2}), we obtain
$|a_i\om_i|(|a_i\om_i|-|a_0a_1|)+r_i^2\ge 0$
for 
$i=0,1$.
Solving this quadratic inequality, we find
$$|a_i\om_i|\le\frac{|a_0a_1|}{2}\left(1-\sqrt{1-\frac{4r_i^2}{|a_0a_1|^2}}\right)
  <\frac{2r_i^2}{|a_0a_1|}\le\frac{4r_i^2}{|\om_0\om_1|},$$
because
$|\om_0\om_1|\le|a_0a_1|+2r\le 2|a_0a_1|$.
\end{proof}

\begin{pro}\label{pro:filling_dist}  Given 
$s_0=(\om_0,r_0)$, $s_1=(\om_1,r_1)\in Y$
with distinct
$\om_0$, $\om_1\in X_\om$,
we have
$$\exp\left(\frac{1}{2}\rho(s_0,s_1)\right)=\frac{|\om_0\om_1|}{\sqrt{r_0r_1}}(1+O(r)),$$
where for 
$r=\max\{r_0,r_1\}\le|\om_0\om_1|/4$
the function
$O(r)$
can be estimated as
$|O(r)|\le\frac{c_0r}{|\om_0\om_1|}$
with some universal constant
$c_0$.
\end{pro}

\begin{proof} Let 
$\ga=(a_0,a_1)$
be the line through
$s_0$, $s_1$
with the order
$a_0,s_0,s_1,a_1$
of our points.
By Proposition~\ref{pro:sinversion_strict}, the space inversion
$s_i$
can be represented as
$s_i=\phi_{b_i}=\phi_{c_i}$
for 
$b_i=(a_0,a_1,S_i)$, $c_i=(\om,\om_i,T_i)$,
where
$S_i$
is a sphere between
$a_0$, $a_1$,
and
$T_i$
a sphere between
$\om$, $\om_i$, $i=0,1$.
By the assumption,
$T_i$
is the sphere in
$X_\om$
of radius
$r_i$
centered at
$\om_i$.
By Lemma~\ref{lem:sphere_intersection} there is
$x_i\in S_i\cap T_i$.
Then
$|x_ia_i|=|x_i\om_i|+O(r_i^2)=r_i+O(r_i^2)$
with
$|O(r_i^2)|<\frac{4r^2}{|\om_0\om_1|}$
according (\ref{eq:dist_center_estimate}). Next, for 
$i\neq j$
we have
$|x_ia_j|=|\om_0\om_1|+O_i(r_0,r_1)$
with
$|O_i(r_0,r_1)|\le r_i+|a_j\om_j|<r+\frac{4r^2}{|\om_0\om_1|}$
again by (\ref{eq:dist_center_estimate}).
We conclude that
$$\frac{|x_ia_j|}{|x_ia_i|}=\frac{|\om_0\om_1|}{r_i}\left(1+O(r)\right),$$
where
$|O(r)|\le\frac{c_1r}{|\om_0\om_1|}$
for some universal constant
$c_1$.

We have
$\rho(s_0,s_1)=|\ln\langle a_0,x_0,x_1,a_1\rangle|$
for any 
$x_i\in S_i$,
and
$$\langle a_0,x_0,x_1,a_1\rangle=\frac{|a_0x_1|\cdot|x_0a_1|}{|a_0x_0|\cdot|x_1a_1|}
    =\frac{|\om_0\om_1|^2}{r_0r_1}\left(1+O(r)\right).$$
Thus
$$\exp\left(\frac{1}{2}\rho(s_0,s_1)\right)
  =\frac{|\om_0\om_1|}{\sqrt{r_0r_1}}\left(1+O(r)\right)$$
with
$|O(r)|\le\frac{c_0r}{|\om_0\om_1|}$
for some universal constant
$c_0$.
\end{proof}

For 
$s$, $s'\in Y$
we define their {\em Gromov product} w.r.t. a Busemann function
$b$
by
$$(s|s')_b=\frac{1}{2}\left(b(s)+b(s')-\rho(s,s')\right).$$

\begin{cor}\label{cor:gromov_prod_dist} For the Busemann function
$b:Y\to\R$
centered at
$\om\in X$
and defined by
$b(t)=\ln\frac{1}{r_t}$
for 
$t=(t(\om),r_t)\in X_\om\times\R_+$
and for each distinct
$a_0$, $a_1\in X_\om$
there exists a limit
$$(a_0|a_1)_b:=\lim_{r\to 0}(s_0|s_1)_b\ \textrm{and}\ e^{-(a_0|a_1)_b}=|a_0a_1|_\om,$$
where
$s_0$, $s_1\in H_r\cap(a_0,a_1)$.
In particular, the Riemannian distance
$\rho$
on the filling
$Y=\fil X$
induces on
$X=\di Y$
the initial M\"obius structure of
$X$.
\end{cor}

\begin{proof} In the upper half-space model for 
$Y$
with base
$X_\om$,
we have
$s_i=(\om_i,r)$
for some
$\om_i\in X_\om$.
We can assume that 
$s_i$
lie on the line
$(a_0,a_1)$
in the order
$a_0,s_0,s_1,a_1$.
By Lemma~\ref{lem:dist_end_points}, 
$\om_i\to a_i$
as
$r\to 0$.
With Proposition~\ref{pro:filling_dist} we have
$$(s_0|s_1)_b=\ln\frac{1}{r}-\frac{1}{2}\rho(s_0,s_1)=-\ln|a_0a_1|_\om+o(1)$$
as
$r\to 0$,
which implies the corollary.
\end{proof}

\subsection{Sectional curvature of $Y$}
\label{subsect:sect_curvature}

To complete the proof of Theorem~\ref{thm:moebius}, it remains
to check that the M\"obius structure of
$X=\di Y$
is canonical. By Corollary~\ref{cor:gromov_prod_dist}
it suffices to check that the maximum of sectional curvatures of
$Y$
is
$-1$.

We already mentioned that the assumption (E) of Theorem~\ref{thm:moebius}
also serves as a normalization condition.

\begin{pro}\label{pro:sect_curvature} Sectional curvatures 
$K_\si$
of
$Y$
bounded above by
$-1$,
moreover
$\max K_\si=-1$.
\end{pro}

\begin{proof} It suffices to check that
$K_\si\le-1$
because
$Y$
contains geodesic isometric copies of the hyperbolic plane
$\hyp^2$.
We assume that for every
$\om\in X$
the nilpotent group
$N=N_\om$
of shifts of
$X_\om$
is not abelian since otherwise
$X=\wh\R^n$
by Proposition~\ref{pro:nonintegrable_canonical_distribution}
and
$Y=\hyp^{n+1}$.

We fix 
$\om\in X$, $o\in X_\om$.
Then the homothety group
$\Ga=\Ga_{o,\om}$
acts on 
$N$
by conjugation,
$\ga(\eta)=\ga\circ\eta\circ\ga^{-1}$
for every
$\ga\in\Ga$, $\eta\in N$.
Thus
$G=N\ltimes\Ga<\aut X$
is a solvable group of M\"obius automorphisms of
$X$,
which acts simply transitively by isometries on
$Y$.
We fix 
$s_0\in Y$, 
e.g.
$s_0=(o,1)\in X_\om\times\R_+$,
and identify 
$G$
with
$Y$
by
$g\mapsto g^\ast(s_0)=g\circ s_0\circ g^{-1}$
for 
$g\in G$.
Then the tangent space
$T_{s_0}Y$
is endowed with the Lie algebra structure 
$\mathfrak g$
of
$G$.

Let 
$H_1=X_\om\times\{1\}\sub Y$
be the horosphere through
$s_0$
in the upper half-space model. Then the tangent space
$T_{s_0}H_1\sub T_{s_0}Y$
is identified with the Lie algebra of
$N$, $[\mathfrak{g},\mathfrak{g}]=\mathfrak{n}$.
Let 
$\{\mathfrak{g^i}\}$
be the lower central series of
$\mathfrak{n}$
defined by
$\mathfrak{g^1}=\mathfrak{n}$, $\mathfrak{g^{i+1}}=[\mathfrak{n},\mathfrak{g^i}]$.
Decompose
$\mathfrak{g^i}$
orthogonally as
$\mathfrak{g^i}=\mathfrak{a_i}+\mathfrak{g^{i+1}}$.
Since
$\mathfrak{n}$
is nilpotent, we have
$\mathfrak{n}=\sum_1^m\mathfrak{a_i}$
for some 
$m\in\N$.
Note that
$m\ge 2$
by the assumption that 
$N$
is not abelian. We fix a unit vector 
$A_0\in\mathfrak{g}$
which is orthogonal to
$\mathfrak{n}$.
Since
$Y$
is symmetric, its curvature tensor
$\cR$
is parallel,
$\nabla\cR=0$.
Then by \cite{He} the symmetric part
$D_0$
of the endomorphism
$\ad A_0|\mathfrak{n}$
has the eigenvalues
$i\cdot\la$,
$D_0|\mathfrak{a_i}=i\cdot\la\cdot\id$
for some
$\la>0$
and
$i=1,\dots,m$,
see \cite[Proposition 3, (iib)]{He}. By \cite[Corollary]{He} this implies
that nonzero eigenvalues of the curvature operator 
$u\mapsto\cR(u,A_0,A_0)$,
i.e. the respective sectional curvatures
$K_\si$,
are
$-i^2\la^2$, $i=1,\dots,m$.
Note that the subspace
$\mathfrak{g^2}=[\mathfrak{n},\mathfrak{n}]\in T_oX_\om$
is tangent to the 
$\K$-line
through
$o$,
any nonzero vector
$u\in\mathfrak{a_1}\sub T_oX_\om$
is tangent to a Ptolemy line in
$X_\om$
through
$o$.
It means that the 2-direction
$(u,A_0)\sub T_{s_0}Y$
is tangent to a geodesic isometric copy of
$\hyp^2$
in
$Y$.
Thus
$\la=1$
and hence
$\max K_\si=-1$.
\end{proof}

\begin{rem}\label{rem:pinching} It is proved in
\cite[Proposition 3]{He} that actually
$m=2$
and thus 
$Y$
is
$1/4$-pinched, $-4\le K_\si\le -1$.
Again, the proof does not use the classification of 
rank one symmetric spaces.
\end{rem}

\bigskip
\begin{tabbing}

Sergei Buyalo,\hskip11em\relax \= Viktor Schroeder,\\

St. Petersburg Dept. of Steklov \>
Institut f\"ur Mathematik, Universit\"at \\

Math. Institute RAS, Fontanka 27, \>
Z\"urich, Winterthurer Strasse 190, \\

191023 St. Petersburg, Russia\>  CH-8057 Z\"urich, Switzerland\\

{\tt sbuyalo@pdmi.ras.ru}\> {\tt viktor.schroeder@math.uzh.ch}\\

\end{tabbing}


\begin{thebibliography}{ABCDE}

\bibitem[BFW]{BFW} S.~Buckley, K.~Falk, D.~Wraith, Ptolemaic spaces and CAT(0), 
Glasg. Math. J. 51 (2009), no. 2, 301--314. 

\bibitem[FLS]{FLS} T.~Foertsch, A.~Lytchak, V.~Schroeder, Nonpositive curvature 
and the Ptolemy inequality, Int. Math. Res. Not. IMRN 2007, 
no. 22, Art. ID rnm100, 15 pp.

\bibitem[FS1]{FS1} T.~Foertsch, V.~Schroeder,
Hyperbolicity,
$\CAT(-1)$-spaces and Ptolemy inequality, 
 Math. Ann. 350 (2011), no. 2, 339 -- 356. 

\bibitem[FS2]{FS2} T.~Foertsch, V.~Schroeder, A M\"obius Characterization 
of Metric Spheres, arXiv:math/1008.3250, 2010, to appear in
Manuscripta Math.

\bibitem[He]{He} E.~Heintze, On homogeneous manifolds of
negative curvature, Math.Ann. 211 (1974), 23--34.

\bibitem[H-UL]{H-UL} J.-B.~Hiriart-Urruty and C.~ Lemar\'echal,
Fundamentals of Convex analysis. Berlin: Springer, 2001.

\bibitem[Kay]{Kay} D.~Kay, Ptolemaic metric spaces and 
the characterization of geodesics by vanishing metric curvature, 
Ph.D. thesis, Michigan State Univ., East Lansing, MI, 1963.

\bibitem[Kr]{Kr} L.~Kramer, Two-transitive Lie groups, J. reine angew. 
Math. 563 (2003), 83--113.

\bibitem[Sch]{Sch} I.~Schoenberg, A remark on M. M. Day's characterization of 
inner-product spaces and a conjecture of L. M. Blumenthal,  
Proc. Amer. Math. Soc. 3 (1952) 961--964.

\bibitem[Sieb]{Sieb} E.~Siebert, Contractive automorphisms on locally
compact groups, Math.Z.191 (1986), 73--90.

\end{thebibliography}
\end{document}